\documentclass[a4paper]{amsart}

\usepackage{enumerate, amsmath, amsfonts, amssymb, amsthm, wasysym, graphics, graphicx, xcolor, url, hyperref, hypcap, a4wide, xargs, multicol, overpic, pdflscape, multirow, hvfloat, minibox, accents, array, xifthen}
\hypersetup{colorlinks=true, citecolor=darkblue, linkcolor=darkblue}
\usepackage[all]{xy}
\usepackage[bottom]{footmisc}
\usepackage{tikz}\usetikzlibrary{trees,snakes,shapes,arrows,matrix,calc}
\usepackage{pgfplots}
\usepgfplotslibrary{polar}
\graphicspath{{figures/}{3dimensionalCompatibilityFansLabeled/}}
\makeatletter
\def\input@path{{figures/}}
\makeatother

\title{Compatibility fans for graphical nested complexes}

\thanks{VP was partially supported by the Spanish MICINN grant MTM2011-22792 and by the French ANR grants EGOS~(12\,JS02\,002\,01) and SC3A~(15\,CE40\,0004\,01).}

\author{Thibault Manneville}
\address[TM]{LIX, \'Ecole Polytechnique, Palaiseau}
\email{thibault.manneville@lix.polytechnique.fr}
\urladdr{http://www.lix.polytechnique.fr/~manneville/}

\author{Vincent Pilaud}
\address[VP]{CNRS \& LIX, \'Ecole Polytechnique, Palaiseau}
\email{vincent.pilaud@lix.polytechnique.fr}
\urladdr{http://www.lix.polytechnique.fr/~pilaud/}

\newtheorem{theorem}{Theorem}
\newtheorem{corollary}[theorem]{Corollary}
\newtheorem{proposition}[theorem]{Proposition}
\newtheorem{lemma}[theorem]{Lemma}
\newtheorem{definition}[theorem]{Definition}
\newtheorem{conjecture}[theorem]{Conjecture}

\theoremstyle{definition}
\newtheorem{example}[theorem]{Example}
\newtheorem{remark}[theorem]{Remark}
\newtheorem{question}[theorem]{Question}

\newcommand{\R}{\mathbb{R}} 
\newcommand{\N}{\mathbb{N}} 
\newcommand{\HH}{\mathbb{H}} 
\newcommand{\cA}{\mathcal{A}} 
\newcommand{\cF}{\mathcal{F}} 
\newcommand{\cG}{\mathcal{G}} 
\newcommand{\cD}{\mathcal{D}} 
\renewcommand{\b}[1]{\mathbf{#1}} 

\newcommand{\set}[2]{\left\{ #1 \;\middle|\; #2 \right\}} 
\newcommand{\bigset}[2]{\big\{ #1 \;|\; #2 \big\}} 
\newcommand{\biggset}[2]{\bigg\{ #1 \;\bigg|\; #2 \bigg\}} 
\newcommand{\ssm}{\smallsetminus} 
\newcommand{\dotprod}[2]{\langle \, #1 \; | \; #2 \, \rangle} 
\newcommand{\symdif}{\,\triangle\,} 
\newcommand{\eqdef}{\mbox{\,\raisebox{0.2ex}{\scriptsize\ensuremath{\mathrm:}}\ensuremath{=}\,}} 
\newcommand{\defeq}{\mbox{~\ensuremath{=}\raisebox{0.2ex}{\scriptsize\ensuremath{\mathrm:}} }} 
\newcommand{\polar}{^\diamond} 
\newcommand{\HS}{\b{H}^\ge} 

\newcommand{\Asso}{\mathsf{Asso}} 
\newcommand{\Nest}{\mathsf{Nest}} 

\newcommand{\ground}{\mathrm{V}} 
\newcommandx{\graphG}[1][1=G]{\mathrm{#1}} 
\newcommand{\pathG}{\graphG[P]} 
\newcommand{\cycleG}{\graphG[O]} 
\newcommand{\completeG}{\graphG[K]} 
\newcommand{\starG}{\graphG[X]} 
\newcommand{\spiderG}{\mathfrak{X}} 
\newcommand{\octopusG}{\mathcal{X}} 

\newcommandx{\tube}[1][1=t]{\mathsf{#1}} 
\newcommandx{\tubing}[1][1=T]{\mathsf{#1}} 
\newcommandx{\spine}[1][1=S]{\mathsf{#1}} 
\newcommand{\lab}{\lambda} 
\newcommand{\nestedComplex}{\mathcal{N}} 
\newcommand{\compatibilityDegree}[2]{(#1\,\|\,#2)} 
\newcommand{\biggCompatibilityDegree}[2]{\bigg(#1\,\bigg\|\,#2\bigg)} 
\newcommand{\compatibilityVector}[2]{\b{d}(#1, #2)} 
\newcommand{\compatibilityMatrix}[2]{\b{d}(#1, #2)} 
\newcommand{\compatibilityFan}[2]{\cD(#1, #2)} 
\newcommand{\dualCompatibilityVector}[2]{\b{d}^*(#1, #2)} 
\newcommand{\dualCompatibilityMatrix}[2]{\b{d}^*(#1, #2)} 
\newcommand{\dualCompatibilityFan}[2]{\cD^*(#1, #2)} 
\newcommand{\connectedComponents}{\kappa} 
\newcommand{\tsup}{\overline{\tube}}
\newcommand{\tinf}{\underline{\tube}}
\newcommand{\squareTube}[1]{{#1}\design}
\newcommand{\rsup}{\overline{r}}
\newcommand{\ninf}{\underline{n}}
\newcommand{\rot}{\circlearrowright} 
\newcommand{\rev}{\leftrightarrow} 
\newcommand{\design}{^{\protect\scalebox{0.5}{$\square$}}} 
\newcommand{\designNestedComplex}{\mathcal{N}\design} 
\newcommand{\designCompatibilityFan}[2]{\cD\design(#1, #2)} 
\newcommand{\transpose}[1]{#1^t} 

\DeclareMathOperator{\vect}{vect} 

\newcommand{\fref}[1]{Figure~\ref{#1}} 
\newcommand{\ie}{\textit{i.e.}~} 
\newcommand{\eg}{\textit{e.g.}~} 
\newcommand{\ex}{_{\textrm{ex}}} 
\definecolor{darkblue}{rgb}{0,0,0.7} 
\definecolor{green}{RGB}{57,181,74} 
\newcommand{\darkblue}{\color{darkblue}} 
\newcommand{\defn}[1]{\emph{\darkblue #1}} 
\usepackage{todonotes}

\newcommand{\para}[1]{\medskip\noindent\textbf{#1.}} 

\begin{document}

\begin{abstract}
Graph associahedra are natural generalizations of the classical associahedra. They provide polytopal realizations of the nested complex of a graph~$\graphG$, defined as the simplicial complex whose vertices are the tubes (\ie connected induced subgraphs) of~$\graphG$ and whose faces are the tubings (\ie collections of pairwise nested or non-adjacent tubes) of~$\graphG$. The constructions of M.~Carr and S.~Devadoss, of A.~Postnikov, and of A.~Zelevinsky for graph associahedra are all based on the nested fan which coarsens the normal fan of the permutahedron. In view of the combinatorial and geometric variety of simplicial fan realizations of the classical associahedra, it is tempting to search for alternative fans realizing graphical nested complexes.

Motivated by the analogy between finite type cluster complexes and graphical nested complexes, we transpose in this paper S.~Fomin and A.~Zelevinsky's construction of compatibility fans from the former to the latter setting. For this, we define a compatibility degree between two tubes of a graph~$\graphG$. Our main result asserts that the compatibility vectors of all tubes of~$\graphG$ with respect to an arbitrary maximal tubing on~$\graphG$ support a complete simplicial fan realizing the nested complex of~$\graphG$. In particular, when the graph~$\graphG$ is reduced to a path, our compatibility degree lies in~$\{-1,0,1\}$ and we recover F.~Santos' Catalan many simplicial fan realizations of the associahedron.

\medskip
\noindent
\textsc{keywords.} Graph associahedra $\cdot$ finite type cluster algebras $\cdot$ compatibility degrees $\cdot$ compatibility fans.

\medskip
\noindent
\textsc{MSC classes.} 52B11, 52B12, 05E45.
\end{abstract}

\vspace*{.1cm}

\maketitle

\vspace{-.7cm}


\section{Introduction}

\para{Associahedra}
The $n$-dimensional \defn{associahedron} is a simple polytope whose $\frac{1}{n+2}\binom{2n+2}{n+1}$ vertices correspond to Catalan objects (triangulations of an $(n+3)$-gon, binary trees on $n+1$ nodes, ...) and whose edges correspond to mutations between them (diagonal flips, edge rotations, ...). Its combinatorial structure appeared in early works of D.~Tamari~\cite{Tamari} and \mbox{J.~Stasheff~\cite{Stasheff}}, and was first realized as a convex polytope by M.~Haiman~\cite{Haiman} and C.~Lee~\cite{Lee}. Since then, the associahedron has motivated a flourishing research trend with rich connections to combinatorics, geometry and algebra: polytopal constructions~\cite{Loday, HohlwegLange, CeballosSantosZiegler, LangePilaud}, Tamari and Cambrian lattices~\cite{TamariFestschrift, Reading-latticeCongruences, Reading-CambrianLattices}, diameter and Hamiltonicity~\cite{SleatorTarjanThurston, Dehornoy, Pournin, HurtadoNoy}, geometric properties~\cite{BergeronHohlwegLangeThomas, HohlwegLortieRaymond, PilaudStump-barycenter}, combinatorial Hopf algebras \cite{LodayRonco, HivertNovelliThibon-algebraBinarySearchTrees, Chapoton, ChatelPilaud}, to cite a few. The associahedron was also generalized in several directions, in particular to secondary and fiber polytopes~\cite{GelfandKapranovZelevinsky, BilleraFillimanSturmfels}, graph associahedra and nestohedra~\cite{CarrDevadoss, Devadoss, Postnikov, FeichtnerSturmfels, Zelevinsky, Pilaud}, pseudotriangulation polytopes~\cite{RoteSantosStreinu-polytope}, cluster complexes and generalized associahedra~\cite{FominZelevinsky-YSystems, ChapotonFominZelevinsky, HohlwegLangeThomas, Stella, Hohlweg}, and brick polytopes~\cite{PilaudSantos-brickPolytope, PilaudStump-brickPolytope}.

\para{Graph associahedra}
This paper deals with graph associahedra, which were defined by M.~Carr and S.~Devadoss~\cite{CarrDevadoss} in connection to C.~De Concini and C.~Procesi's wonderful arrangements~\cite{DeConciniProcesi}. Given a simple graph~$\graphG$ with~$\connectedComponents$ connected components and~$n+\connectedComponents$ vertices, the \defn{$\graphG$-associahedron}~$\Asso(\graphG)$ is an $n$-dimensional simple polytope whose combinatorial structure encodes the connected subgraphs of~$\graphG$ and their nested structure. More precisely, the face lattice of the polar of the \mbox{$\graphG$-associahedron} is isomorphic to the \defn{nested complex}~$\Nest(\graphG)$ on~$\graphG$, defined as the simplicial complex of all collections of tubes (connected induced subgraphs) of~$\graphG$ which are pairwise compatible (either nested, or disjoint and non-adjacent). As illustrated in Figures~\ref{fig:associahedron}, \ref{fig:cyclohedron} and~\ref{fig:permutahedron}, the graph associahedra of certain special families of graphs happen to coincide with well-known families of polytopes: classical associahedra are path associahedra, cyclohedra are cycle associahedra, and permutahedra are complete graph associahedra. The graph associahedra were extended to the \defn{nestohedra}, which are simple polytopes realizing the nested complex of arbitrary building sets~\cite{Postnikov, FeichtnerSturmfels}. Graph associahedra and nestohedra have been geometrically realized in different ways: by successive truncations of faces of the standard simplex~\cite{CarrDevadoss}, as Minkowski sums of faces of the standard simplex~\cite{Postnikov, FeichtnerSturmfels}, or from their normal fans by exhibiting explicit inequality descriptions~\cite{Devadoss, Zelevinsky}. For a given graph~$\graphG$, the resulting polytopes all have the same normal fan which coarsens the type~$A$ Coxeter arrangement: its rays are the characteristic vectors of the tubes, and its cones are generated by characteristic vectors of compatible tubes. Alternative realizations of graph associahedra with different normal fans are obtained by successive truncations of faces of the cube in~\cite{Volodin, DevadossForceyReisdorfShowers}. The objective of this paper is to provide a new unrelated family of complete simplicial fans realizing the graphical nested complex~$\Nest(\graphG)$ for any graph~$\graphG$.
\begin{figure}[p]
  \capstart
  \centerline{\includegraphics[scale=1.13]{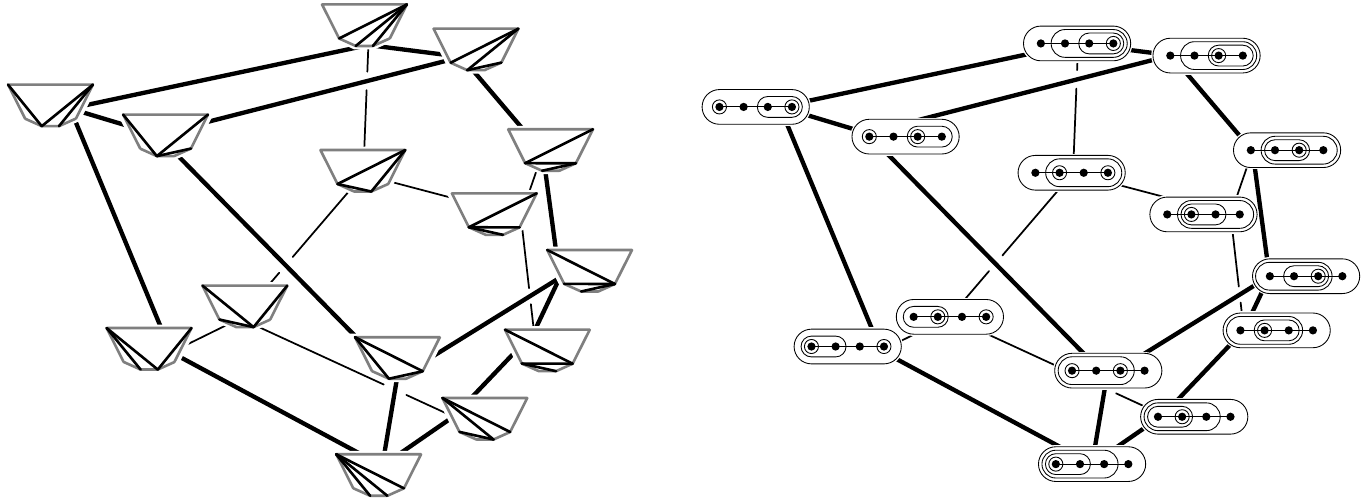}}
  \caption{The classical associahedron (left) is the path associahedron (right).}
  \label{fig:associahedron}
\end{figure}
\begin{figure}[p]
  \capstart
  \centerline{\includegraphics[scale=1.13]{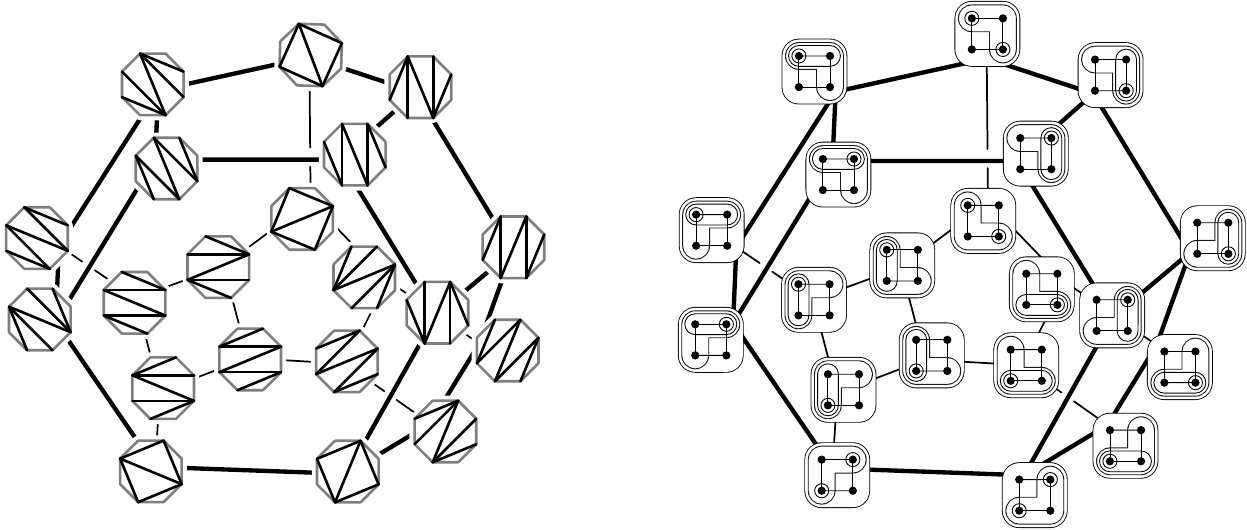}}
  \caption{The cyclohedron (left) is the cycle associahedron (right).}
  \label{fig:cyclohedron}
\end{figure}
\begin{figure}[p]
  \capstart
  \centerline{\includegraphics[scale=1.13]{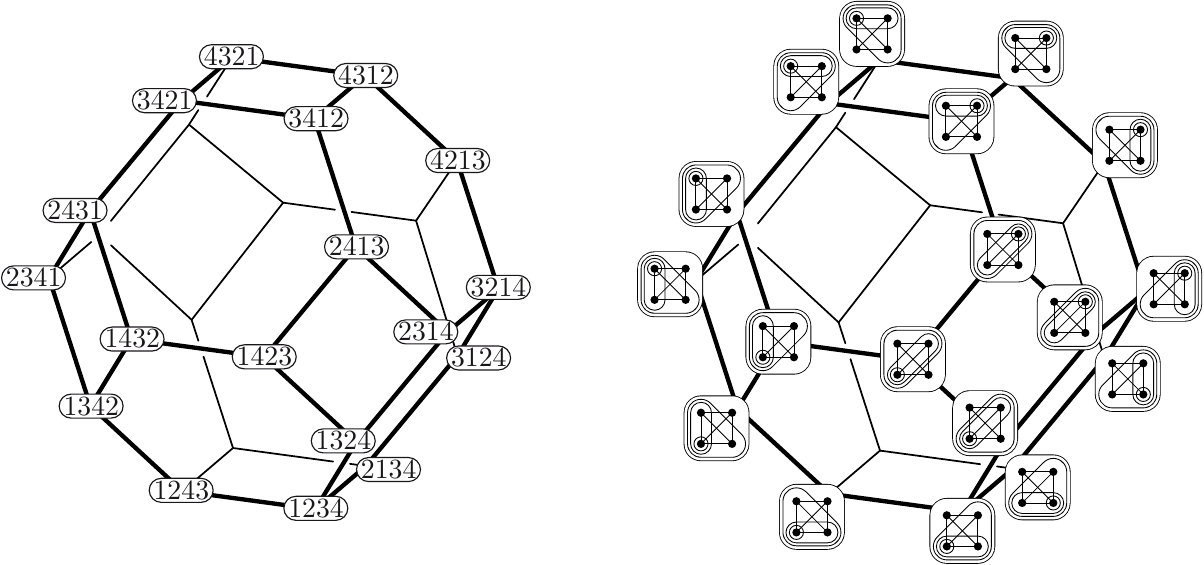}}
  \caption{The permutahedron (left) is the complete graph associahedron (right).}
  \label{fig:permutahedron}
\end{figure}

\para{Cluster algebras and cluster fans} 
Our construction is directly inspired from combinatorial and geometric properties of finite type cluster algebras and generalized associahedra introduced by S.~Fomin and A.~Zelevinsky in~\cite{FominZelevinsky-ClusterAlgebrasI, FominZelevinsky-ClusterAlgebrasII, FominZelevinsky-YSystems}. Note that A.~Zelevinsky~\cite{Zelevinsky} already underlined the closed connection between nested complexes, nested fans, and nestohedra on one hand and cluster complexes, cluster fans and generalized associahedra on the other hand. More recently, T.~Lam and P.~Pylyavskyy established an even deeper connection with linear Laurent Phenomenon algebras~\cite{LamPylyavskyy-LaurentPhenomenonAlgebras, LamPylyavskyy-LinearLaurentPhenomenonAlgebras}, see the discussion in Section~\ref{subsec:LPA}. In this paper, we use ideas from cluster algebras to obtain results on graphical nested complexes, which in turn translate to relevant properties of the geometry of finite type cluster algebras.

\defn{Cluster algebras} are commutative rings generated by a set of \defn{cluster variables} grouped into overlapping \defn{clusters}. The clusters are obtained from an initial cluster~$X^\circ$ by a mutation process. Each mutation exchanges a single variable in a cluster according to a formula controlled by a combinatorial object. We refer to~\cite{FominZelevinsky-ClusterAlgebrasI} since precise details on this mutation are not needed here. Two cluster variables are \defn{compatible} if they belong to a same cluster and \defn{exchangeable} if they are not compatible and belong to two clusters connected by a mutation. The \defn{cluster complex} of a cluster algebra~$\cA$ is the pure simplicial complex whose vertices are cluster variables of~$\cA$ and whose facets are the clusters of~$\cA$. A cluster algebra is of \defn{finite type} if it has finitely many cluster variables, and thus if its cluster complex is finite. Finite type cluster algebras were classified in~\cite{FominZelevinsky-ClusterAlgebrasII}: up to isomorphism, there is one finite type cluster algebra for each finite crystallographic root system.

The Laurent Phenomenon~\cite{FominZelevinsky-ClusterAlgebrasI} asserts that each cluster variable~$x$ can be expressed as a Laurent polynomial in terms of the cluster variables~$x_1^\circ, \dots, x_n^\circ$ of the initial cluster~$X^\circ$. The \defn{$\b{d}$-vector} of~$x$ with respect to~$X^\circ$ is the vector~$\b{d}(X^\circ, x)$ whose $i$th coordinate is the exponent of the initial variable~$x_i^\circ$ in the denominator of~$x$. In finite type, this exponent was also interpreted in~\cite{FominZelevinsky-YSystems, CeballosPilaud} as the \defn{compatibility degree}~$\compatibilityDegree{x_i^\circ}{x}$ between the cluster variables~$x_i^\circ$ and~$x$. The compatiblity degree~$\compatibilityDegree{\cdot}{\cdot}$ has the following properties: for any distinct cluster variables~$x$ and~$x'$, we have~$\compatibilityDegree{x}{x'} \ge 0$ with equality if and only if~$x$ and~$x'$ are compatible, and~$\compatibilityDegree{x}{x'} = 1 = \compatibilityDegree{x'}{x}$ if and only if $x$ and~$x'$ are exchangeable. The $\b{d}$-vectors can be used to construct a simplicial fan realization of the cluster complex, called \defn{$\b{d}$-vector fan}: it is known for certain initial clusters in finite type cluster algebras, that the cones generated by the $\b{d}$-vectors of all collections of compatible cluster variables form a complete simplicial fan realizing the cluster complex. This is proved by S.~Fomin and A.~Zelevinsky~\cite{FominZelevinsky-ClusterAlgebrasII} for the bipartite initial cluster, by S.~Stella~\cite{Stella} for all acyclic initial clusters, and by F.~Santos~\cite[Section~5]{CeballosSantosZiegler} for any initial cluster in type~$A$. In fact, we expect this property to hold for any initial cluster, acyclic or not, of any finite type cluster algebra: this paper proves it for types~$A$, $B$, and~$C$ (our general proof provides as a particular case a new proof in type~$A$, similar to that of~\cite[Section~5]{CeballosSantosZiegler}), and the other finite types are investigated in a current project of the authors in collaboration with C.~Ceballos.

There is another complete simplicial fan realizing the cluster complex, whose rays are now given by the \defn{$\b{g}$-vectors} of the cluster variables, defined as the multi-degrees of the cluster variables expressed in the cluster algebra with principal coefficients~\cite{FominZelevinsky-ClusterAlgebrasIV}. The fact that the cones generated by the $\b{g}$-vectors of all collections of compatible cluster variables form a complete simplicial fan is a consequence of~\cite{Reading-UniversalClusterAlgebra}. When the initial cluster is acyclic, the resulting $\b{g}$-vector fan is the \defn{Cambrian fan} of N.~Reading and D.~Speyer~\cite{ReadingSpeyer} and it coarsens the Coxeter fan.

\para{Compatibility fans for graphical nested complexes}
\enlargethispage{.1cm}
Motivated by the combinatorial and geometric richness of the~$\b{d}$- and $\b{g}$-vector fans described above for finite type cluster algebras, we want to construct various simplicial fan realizations of the nested complex. As it coarsens the Coxeter fan, the normal fan of the graph-associahedra and nestohedra of~\cite{CarrDevadoss, Devadoss, Postnikov, FeichtnerSturmfels, Zelevinsky} should be considered as an analogue of the $\b{g}$-vector fan. A tentative approach to construct alternative $\b{g}$-vector fans for tree associahedra can be found in~\cite{Pilaud}.

In this paper, we construct an analogue of the $\b{d}$-vector fan for graphical nested complexes. For two tubes~$\tube, \tube'$ of a graph~$\graphG$, we define the \defn{compatibility degree}~$\compatibilityDegree{\tube}{\tube'}$ of~$\tube$ with~$\tube'$ to be~$\compatibilityDegree{\tube}{\tube'} = -1$ if~$\tube = \tube'$, $\compatibilityDegree{\tube}{\tube'} = 0$ if~$\tube \ne \tube'$ are compatible, and~${\compatibilityDegree{\tube}{\tube'} = |\{\text{neighbors of $\tube$ in $\tube' \ssm \tube$}\}|}$ otherwise. Similar to the compatibility degree for cluster algebras, it satisfies~$\compatibilityDegree{\tube}{\tube'} \ge 0$ for any distinct tubes~$\tube, \tube'$ of~$\graphG$, with equality if and only if $\tube$ and~$\tube'$ are compatible, and~$\compatibilityDegree{\tube}{\tube'} = 1 = \compatibilityDegree{\tube'}{\tube}$ if and only if $\tube$ and~$\tube'$ are exchangeable. We define the \defn{compatibility vector}~$\compatibilityVector{\tubing^\circ}{\tube} \eqdef \left[ \compatibilityDegree{\tube_1^\circ}{\tube}, \dots, \compatibilityDegree{\tube_n^\circ}{\tube} \right]$ of a tube~$\tube$ with respect to an initial maximal tubing~$\tubing^\circ \eqdef \{\tube_1^\circ, \dots, \tube_n^\circ\}$, and the \defn{compatibility matrix}~$\compatibilityMatrix{\tubing^\circ}{\tubing} \eqdef [\compatibilityDegree{\tube_i^\circ}{\tube_j}]_{i \in [n], j \in [m]}$ of a tubing~$\tubing \eqdef \{\tube_1, \dots, \tube_m\}$ of~$\graphG$ with respect to the initial tubing~$\tubing^\circ$. We write~$\R_{\ge 0} \, \compatibilityMatrix{\tubing^\circ}{\tubing}$ to denote the polyhedral cone generated by the compatibility vectors of the tubes of~$\tubing$ with respect to~$\tubing^\circ$.

Notice that our degree is asymmetric, which yields a natural notion of duality. We define the \defn{dual compatibility vector}~$\dualCompatibilityVector{\tube}{\tubing^\circ} \eqdef \left[ \compatibilityDegree{\tube}{\tube_1^\circ}, \dots, \compatibilityDegree{\tube}{\tube_n^\circ} \right]$ of~$\tube$ with respect to~$\tubing^\circ$ and the \defn{dual compatibility matrix}~$\dualCompatibilityMatrix{\tubing}{\tubing^\circ} \eqdef [\compatibilityDegree{\tube_j}{\tube_i^\circ}]_{i\in [n], j \in [m]}$. To make the distinction clear from the dual compatibility vector, we often call~$\compatibilityVector{\tubing^\circ}{\tube}$ the \defn{primal} compatibility vector.

Although there is no denominators involved anymore, we still use the letter~$\b{d}$ to stand for compatibility \textbf{d}egree vector, and to match with the cluster algebra notations. Indeed, our compatibility degrees for paths and cycles coincide with the compatibility degree of~\cite{FominZelevinsky-YSystems} in types~$A$, $B$, and~$C$. Compatibility degrees on type~$A$ cluster variables correspond to compatibility (and dual compatibility) degrees on tubes of paths while compatibility degrees in type~$C$ (resp.~$B$) cluster variables correspond to compatibility (resp.~dual compatibility) degrees on tubes of cycles.

Our main result is the following analogue of the compatibility fan for path associahedra constructed by F.~Santos in~\cite[Section~5]{CeballosSantosZiegler}. 

\begin{theorem}
For any graph~$\graphG$, the compatibility vectors (resp.~dual compatibility vectors) of all tubes of~$\graphG$ with respect to any initial maximal tubing~$\tubing^\circ$ on~$\graphG$ support a complete simplicial fan realizing the nested complex on~$\graphG$. More precisely, both collections of cones
\[
\compatibilityFan{\graphG}{\tubing^\circ} \eqdef \bigset{\R_{\ge 0} \, \compatibilityVector{\tubing^\circ}{\tubing}}{\tubing \text{ tubing on } \graphG}
\;\;\,\text{and}\;\;\;\;
\dualCompatibilityFan{\graphG}{\tubing^\circ} \eqdef \bigset{\R_{\ge 0} \, \dualCompatibilityMatrix{\tubing}{\tubing^\circ}}{\tubing \text{ tubing on } \graphG}
\]
are complete simplicial fans.
\end{theorem}

We then study the number of distinct compatibility fans. As in~\cite{CeballosSantosZiegler}, we consider that two compatibility fans~$\compatibilityFan{\graphG}{\tubing^\circ}$ and~$\compatibilityFan{\graphG'}{\tubing'^\circ}$ are equivalent if they differ by a linear isomorphism. Such a linear isomorphism induces an isomorphism between the nested complexes~$\nestedComplex(\graphG)$ and~$\nestedComplex(\graphG')$. Besides those induced by graph isomorphisms between~$\graphG$ and~$\graphG'$, there are non-trivial nested complex isomorphisms: for example, the complementation~$\tube \to \ground \ssm \tube$ on the complete graph, or the map on tubes of the path corresponding to the rotation of diagonals in the polygon. Extending these two examples, we exhibit a non-trivial nested complex isomorphism on any spider (a set of paths attached by one endpoint to a clique). Our next statement shows that these are essentially the only non-trivial nested complex isomorphisms.

\begin{theorem}
All nested complex isomorphisms~$\nestedComplex(\graphG) \to \nestedComplex(\graphG')$ are induced by graph isomorphisms~$\graphG \to \graphG'$, except if one of the connected components of~$\graphG$ is a spider.
\end{theorem}

\begin{corollary}
Except if one of the connected components of~$\graphG$ is a path, the number of linear isomorphism classes of compatibility fans of~$\graphG$ is the number of orbits of maximal tubings on~$\graphG$ under graph automorphisms of~$\graphG$.
\end{corollary}

The next step would be to realize all these complete simplicial fans as normal fans of convex polytopes. This question remains open, except for some particular graphs: besides all graphs with at most~$4$ vertices, we settle the case of paths and cycles following a similar proof as~\cite{CeballosSantosZiegler}.

\begin{theorem}
All compatibility and dual compatibility fans of paths and cycles are polytopal.
\end{theorem}

\para{Overview}
The paper is organized as follows. We first recall in Section~\ref{sec:preliminaries} definitions and basic notions on polyhedral fans and graphical nested complexes. In particular, we state in Proposition~\ref{prop:characterizationFan} a crucial sufficient condition for a set of vectors indexed by a set~$X$ to support a simplicial fan realization of a simplicial complex on~$X$. We also briefly survey the classical constructions of graph associahedra of~\cite{CarrDevadoss, Postnikov, Zelevinsky}.

In Section~\ref{sec:compatibilityFan}, we define the compatibility degree between two tubes of a graph, review its combinatorial properties, and state our geometric results on compatibility and dual compatibility~fans.

We study various examples in Section~\ref{sec:specificGraphs}. After an exhaustive description of the compatibility fans of all graphs with at most~$4$ vertices, we study four families of graphs: paths, cycles, complete graphs, and stars. The first two families connect our construction to S.~Fomin and A.~Zelevinsky's $\b{d}$-vector fans for type~$A$, $B$, and $C$ cluster complexes.

Section~\ref{sec:furtherTopics} discusses various further topics. We first study the behavior of the compatibility fans with respect to products and links. We then describe all nested complex isomorphisms in order to show that most compatibility fans are not linearly isomorphic. We also discuss the question of the realization of our compatibility fans as normal fans of convex polytopes. We extend our construction to design nested complexes~\cite{DevadossHeathVipismakul}. Finally, we discuss the connection of this paper to Laurent Phenomenon algebras~\cite{LamPylyavskyy-LaurentPhenomenonAlgebras, LamPylyavskyy-LinearLaurentPhenomenonAlgebras}.

Finally, we have chosen to gather all proofs of our results in Section~\ref{sec:proofs} with the hope that the properties and examples of compatibility fans treated in Sections~\ref{sec:specificGraphs} and~\ref{sec:furtherTopics} help the reader's~intuition.


\section{Preliminaries}
\label{sec:preliminaries}

In this section, we briefly review classical material to recall basic definitions and fix notations. The reader familiar with polyhedral geometry and graph associahedra can skip these preliminaries and proceed directly to Section~\ref{sec:compatibilityFan}.

\subsection{Polyhedral geometry and fans}

We first recall classical definitions from polyhedral geometry. More details can be found in the textbooks~\cite[Lecture~1]{Ziegler} and~\cite[Section~2.1.1]{DeLoeraRambauSantos}.

A \defn{closed polyhedral cone} is a subset of~$\R^n$ defined equivalently as the positive span of finitely many vectors or as the intersection of finitely many closed linear halfspaces. The \defn{dimension} of a cone is the dimension of its linear span. The \defn{faces} of a cone~$C$ are the intersections of~$C$ with the supporting hyperplanes of~$C$. Faces of polyhedral cones are polyhedral cones. The faces of dimension~$1$ (resp.~codimension~$1$) are called \defn{rays} (resp.~\defn{facets}). We will only consider \defn{pointed} cones, which contain no entire line of~$\R^n$. Therefore, $C$ is the positive span of its rays and the intersection of the halfspaces defined by its facets and containing it. We say that~$C$ is \defn{simplicial} if its rays form a linear basis of its linear span.

A \defn{polyhedral fan} is a collection~$\cF$ of polyhedral cones of~$\R^n$ closed under faces and which intersect properly, \ie
\begin{enumerate}
\item if~$C \in \cF$ and~$F$ is a face of~$C$, then~$F \in \cF$;
\item the intersection of any two cones of~$\cF$ is a face of both.
\end{enumerate}
A polyhedral fan is \defn{simplicial} if all its cones are, and \defn{complete} if the union of its cones covers the entire space~$\R^n$. We will use the following characterization of complete simplicial fans, whose formal proof can be found \eg in~\cite[Corollary 4.5.20]{DeLoeraRambauSantos}.

\begin{proposition}
\label{prop:characterizationFan}
For a simplicial sphere~$\Delta$ with vertex set~$X$ and a set of vectors~$\b{V} \eqdef (\b{v}_x)_{x \in X}$ of~$\R^n$, the collection of cones~$\bigset{\R_{\ge 0}\b{V}_\triangle}{\triangle \in \Delta}$, where~$\R_{\ge 0}\b{V}_\triangle$ denotes the positive span of~$\b{V}_\triangle \eqdef \set{\b{v}_x}{x \in \triangle}$, forms a complete simplicial fan if and~only~if
\begin{enumerate}
\item there exists a facet~$\triangle$ of~$\Delta$ such that~$\b{V}_\triangle$ is a basis of~$\R^n$ and such that the open cones~$\R_{> 0}\b{V}_\triangle$ and~$\R_{> 0}\b{V}_{\triangle'}$ are disjoint for any facet~$\triangle'$ of~$\Delta$ distinct from~$\triangle$;
\item for two adjacent facets~$\triangle, \triangle'$ of~$\Delta$ with~$\triangle \ssm \{x\} = \triangle' \ssm \{x'\}$, there is a linear dependence
\[
\alpha \, \b{v}_x + \alpha' \, \b{v}_{x'} + \sum_{y \in \triangle \cap \triangle'} \beta_y \, \b{v}_y = 0
\]
on~$\b{V}_{\triangle \cup \triangle'}$ in which the coefficients~$\alpha$ and~$\alpha'$ have the same sign (different from~$0$). Note that when these conditions hold, this linear dependence is unique up to rescaling.

\end{enumerate}
\end{proposition}

A \defn{polytope} is a subset~$P$ in~$\R^n$ defined equivalently as the convex hull of finitely many points in~$\R^n$ or as a bounded intersection of finitely many closed half-spaces of~$\R^n$. The faces of~$P$ are the intersections of~$P$ with its supporting hyperplanes. The (outer) \defn{normal cone} of a face~$F$ of~$P$ is the cone generated by the outer normal vectors of the facets (codimension~$1$ faces) of~$P$ containing~$F$. Finally, the (outer) \defn{normal fan} of~$P$ is the collection of the (outer) normal cones of all its faces.

\subsection{Graphical nested complexes}
\label{subsec:nestedComplex}

We review the definitions and basic properties of graphical nested complexes. We refer to~\cite{CarrDevadoss, Devadoss} for the original construction of graph associahedra. This construction extends to nested complexes on arbitrary building sets, see~\cite{Postnikov, FeichtnerSturmfels, Zelevinsky}. Although we remain in the graphical situation, our presentation borrows results from these~papers.

Fix a graph~$\graphG$ with vertex set~$\ground$. Let~$\connectedComponents(\graphG)$ denote the set of connected components of~$\graphG$ and define~$n \eqdef |\ground|-|\connectedComponents(\graphG)|$. A \defn{tube} of~$\graphG$ is a non-empty subset~$\tube$ of vertices of~$\graphG$ inducing a connected subgraph~$\graphG{}[\tube]$ of~$\graphG$. The inclusion maximal tubes of~$\graphG$ are its connected components~$\connectedComponents(\graphG)$; all other tubes are called \defn{proper}. Two tubes~$\tube, \tube'$ of~$\graphG$ are \defn{compatible} if they are either nested (\ie $\tube \subseteq \tube'$ or~$\tube' \subseteq \tube$), or disjoint and non-adjacent (\ie $\tube \cup \tube'$ is not a tube of~$\graphG$). A \defn{tubing} on~$\graphG$ is a set~$\tubing$ of pairwise compatible proper tubes of~$\graphG$. The collection of all tubings on~$\graphG$ is a simplicial complex, called \defn{nested complex} of~$\graphG$ and denoted by~$\nestedComplex(\graphG)$.

\begin{example}
\label{exm:exmTubes}
To illustrate the content of the paper, we will follow a toy example, presented in \fref{fig:exmTubes}. We have represented a graph~$\graphG\ex$ on the left with a tube~$\tube^\circ\ex \eqdef \{a,b,d,f,g,h,i,k,l\}$, and a maximal tubing~$\tubing^\circ\ex$ on~$\graphG\ex$ on the right.

\begin{figure}
  \capstart
  \centerline{\includegraphics[scale=.9]{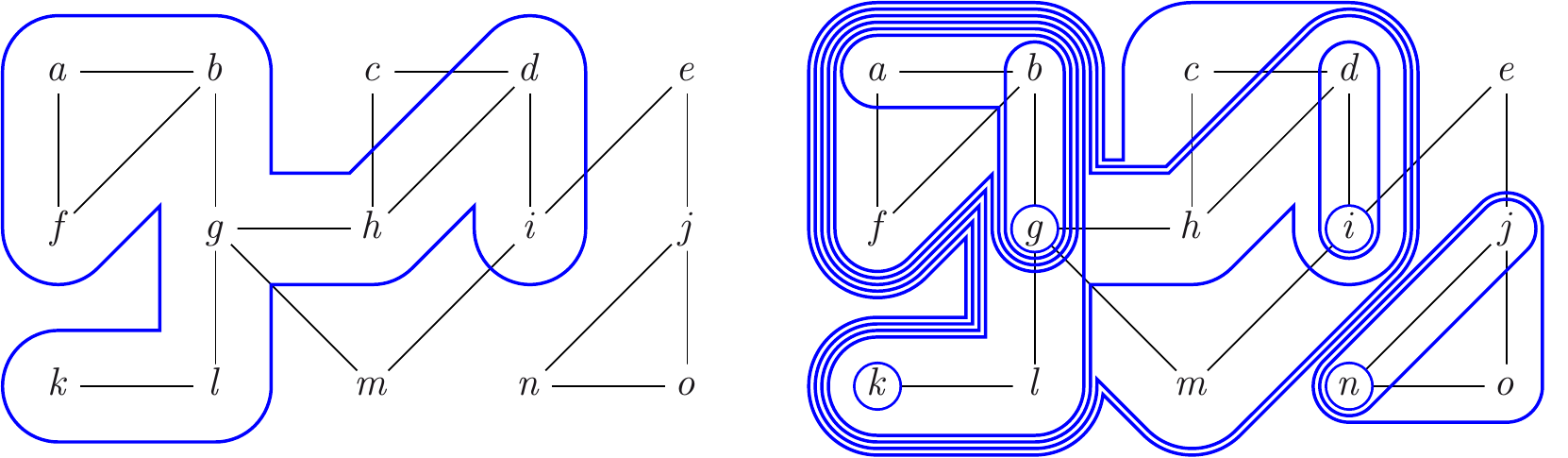}}
  \caption{A tube~$\tube^\circ\ex$ of~$\graphG\ex$ (left) and a maximal tubing~$\tubing^\circ\ex$ on~$\graphG\ex$ (right).}
  \label{fig:exmTubes}
\end{figure}
\end{example}

For a tubing~$\tubing$ on~$\graphG$ and a tube~$\tube$ of~$\tubing \cup \connectedComponents(\graphG)$, we define~$\lambda(\tube, \tubing) \eqdef \tube \ssm \bigcup_{\tube' \in \tubing, \tube' \subsetneq \tube} \tube'$. The sets~$\lambda(\tube, \tubing)$ for~$\tube \in \tubing \cup \connectedComponents(\graphG)$ form a partition of the vertex set of~$\graphG$. When~$\tubing$ is a maximal tubing, each set~$\lambda(\tube, \tubing)$ contains a unique vertex of~$\graphG$ that we call the \defn{root} of~$\tube$ in~$\tubing$.

The nested complex~$\nestedComplex(\graphG)$ is an $(n-1)$-dimensional simplicial sphere. The dual graph of this complex is the graph whose vertices are maximal tubings on~$\graphG$ and whose edges are flips between them. A \defn{flip} is a pair of distinct maximal tubings~$\tubing, \tubing'$ on~$\graphG$ such that~$\tubing \ssm \{\tube\} = \tubing' \ssm \{\tube'\}$ for some tubes~$\tube \in \tubing$ and~$\tube' \in \tubing'$. Since~$\nestedComplex(\graphG)$ is a sphere, any tube of any maximal tubing can be flipped, and the resulting tube is described in the following proposition, whose proof is left to the reader. Remember that we denote by~$\graphG{}[U]$ the subgraph of~$\graphG$ induced by a subset~$U \subseteq \ground$.

\begin{proposition}
\label{prop:flip}
Let~$\tube$ be a tube in a maximal tubing~$\tubing$ on~$\graphG$, and let~$\tsup$ be the inclusion minimal tube of~$\tubing \cup \connectedComponents(\graphG)$ which strictly contains~$\tube$. Then the unique tube~$\tube'$ such that~$\tubing' = \tubing \symdif \{\tube, \tube'\}$ is again a maximal tubing on~$\graphG$ is the connected component of~$\graphG{}[\tsup \ssm \lambda(\tube, \tubing)]$ containing~$\lambda(\tsup, \tubing)$. 
\end{proposition}

We say that two distinct tubes~$\tube, \tube'$ of~$\graphG$ are \defn{exchangeable} if there exists two adjacent maximal tubings~$\tubing, \tubing'$ on~$\graphG$ such that~$\tubing \ssm \{\tube\} = \tubing' \ssm \{\tube'\}$. Note that several such pairs~$\{\tubing, \tubing'\}$ are possible, but they all contain certain tubes. We call \defn{forced tubes} of the exchangeable pair~$\{\tube, \tube'\}$ any tube which belongs to any adjacent maximal tubings~$\tubing, \tubing'$ such that~$\tubing \ssm \{\tube\} = \tubing' \ssm \{\tube'\}$. These tubes are easy to describe: they are precisely the tube~$\tsup \eqdef \tube \cup \tube'$ and the connected components of~$\tsup \ssm (\lambda(\tube,\tubing) \cup \lambda(\tube',\tubing'))$.

\begin{example}
\label{exm:exmFlip}
\fref{fig:exmFlip} illustrates the flip between two maximal tubings~$\tubing\ex$ and~$\tubing'\ex$ on~$\graphG\ex$. The flipped tubes~$\tube\ex = \{a,b,c,d,f,g,h,k,l,m\}$ (with root~$g$) and~$\tube'\ex = \{c,d,e,h,i,m\}$ (with root~$i$) are dashed green, while the forced tubes of the exchangeable pair~$\{\tube\ex, \tube'\ex\}$ are red.

\begin{figure}
  \capstart
  \centerline{\includegraphics[scale=.9]{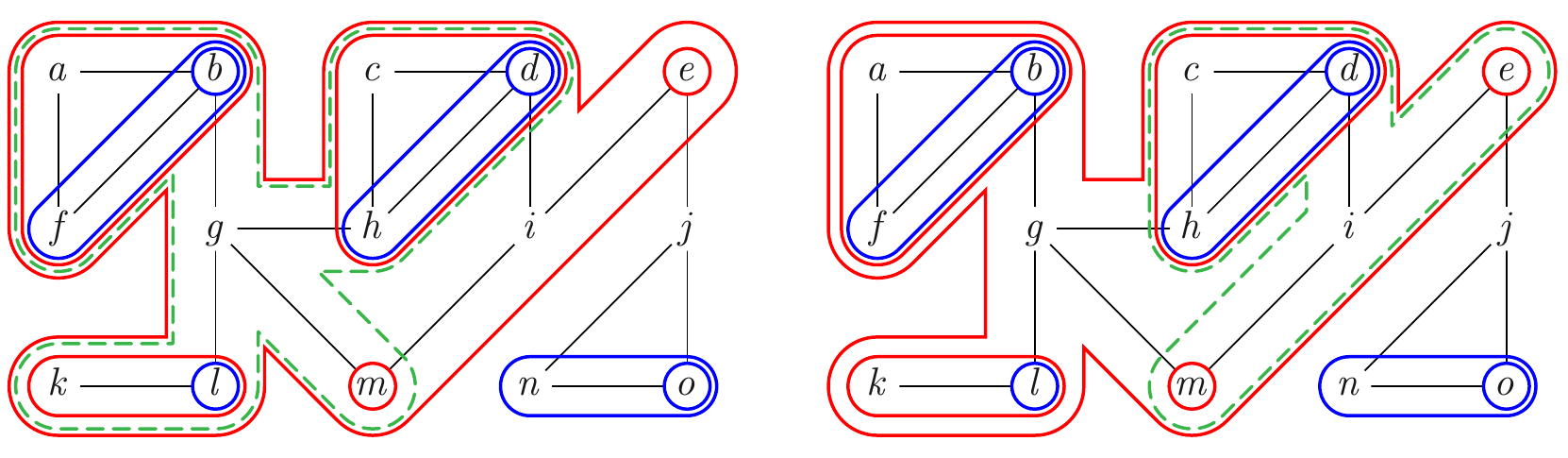}}
  \caption{The flip between two maximal tubings~$\tubing\ex$ and~$\tubing'\ex$ on~$\graphG\ex$.}
  \label{fig:exmFlip}
\end{figure}
\end{example}

\subsection{Graphical nested fans and graph associahedra}
\label{subsec:graphAssociahedra}

Although we will not use them in the remaining of this paper, we briefly survey some constructions of the nested fan and graph associahedron to illustrate the previous definitions. The nested fan is constructed implicitly in~\cite{CarrDevadoss}, and explicitly in~\cite{Postnikov, FeichtnerSturmfels, Zelevinsky} in the more general situation of nested complexes on arbitrary building sets.

Let~$(\b{e}_v)_{v \in \ground}$ be the canonical basis of~$\R^\ground$, let~${\HH \eqdef \bigset{\b{x} \in \R^\ground}{\sum_{w \in W} x_w = 0 \text{ for all } W \in \connectedComponents(\graphG)}}$ and~$\pi : \R^\ground \to \HH$ denote the orthogonal projection on~$\HH$. Let~$\b{g}(\tube) \eqdef \pi \big( \sum_{v \in \tube} \b{e}_v \big)$ denote the projection of the characteristic vector of a tube~$\tube$ of~$\graphG$, and define~$\b{g}(\tubing) \eqdef \set{\b{g}(\tube)}{\tube \in \tubing}$ for a tubing~$\tubing$ on~$\graphG$. These vectors support a complete simplicial fan realization of the nested complex:

\begin{theorem}[\cite{CarrDevadoss, Postnikov, FeichtnerSturmfels, Zelevinsky}]
\label{theo:gFan}
For any graph~$\graphG$, the collection of cones
\[\cG(\graphG) \eqdef \set{\R_{\ge 0} \, \b{g}(\tubing)}{\tubing \text{ tubing on } \graphG}\]
is a complete simplicial fan of~$\HH$, called \defn{nested fan} of~$\graphG$, which realizes the nested complex~$\nestedComplex(\graphG)$.
\end{theorem}

\begin{remark}
The cones of this fan can be encoded as follows. Fix a tubing~$\tubing$ on~$\graphG$. The \defn{spine} of~$\tubing$ is the forest~$\spine$ given by the Hasse diagram of the inclusion poset on~$\tubing \cup \connectedComponents(\graphG)$ where the vertex corresponding to a tube~$\tube$ is labeled by~$\lambda(\tube, \tubing)$. Then the cone~$\R_{\ge 0} \, \b{g}(\tubing)$ is the \defn{braid cone} of~$\spine$ and is polar to the \defn{incidence cone} of~$\spine$:
\[
\R_{\ge 0} \, \b{g}(\tubing) = \set{\b{x} \in \HH}{x_v \le x_w \text{ for all } v \to w \in \spine} = \big(\R_{\ge 0} \set{\b{e}_v-\b{e}_w}{v \to w \in \spine} \big)\polar
\]
where $v \to w \in \spine$ means that there is a directed path in~$\spine$ from the node containing~$v$ to the node of~$\spine$ containing~$w$. Spines are called $B$-trees in~\cite{Postnikov}.
\end{remark}

It is proved in~\cite{CarrDevadoss, Devadoss, Postnikov, FeichtnerSturmfels, Zelevinsky} that the nested fan comes from a polytope. 

\begin{theorem}[\cite{CarrDevadoss, Devadoss, Postnikov, FeichtnerSturmfels, Zelevinsky}]
\label{theo:graphAssociahedron}
For any graph~$\graphG$, the nested fan~$\cG(\graphG)$ is the normal fan of the graph associahedron~$\Asso(\graphG)$.
\end{theorem}

In fact, these different papers define different constructions and realizations for the graph associahedron~$\Asso(\graphG)$. Originally, M.~Carr and S.~Devadoss constructed~$\Asso(\graphG)$ by iterative truncations of faces of the standard simplex. S.~Devadoss then gave explicit integer coordinates for the facets in~\cite{Devadoss}. In a different context, A.~Postnikov~\cite{Postnikov} and independently E.~M.~Feichtner and B.~Sturmfels~\cite{FeichtnerSturmfels} constructed graph associahedra (more generally nestohedra) by Minkowski sums of faces of the standard simplex. Finally, A.~Zelevinsky~\cite{Zelevinsky} discussed which polytopes can realize the nested fan using a characterization of all possible facet inequality descriptions.


\section{Compatibility degrees, vectors, and fans}
\label{sec:compatibilityFan}

In this section, we define our compatibility degree on tubes and state our main results. To go straight to results and examples, the proofs are delayed to Section~\ref{sec:proofs}. We refer to Section~\ref{sec:specificGraphs} for examples of compatibility fans of relevant graphs in connection to the geometry of type~$A$, $B$ and~$C$ cluster complexes, and to Section~\ref{sec:furtherTopics} for further properties of compatibility fans.

\subsection{Compatibility degree}

Motivated by the compatibility degrees in finite type cluster algebras, we introduce an analogous notion on tubes of graphical nested complexes.

\begin{definition}
\label{def:compatibilityDegree}
For two tubes~$\tube, \tube'$ of~$\graphG$, the \defn{compatibility degree} of~$\tube$ with~$\tube'$ is
\[
\compatibilityDegree{\tube}{\tube'} =
\begin{cases}
-1 & \text{if } \tube = \tube', \\
|\{\text{neighbors of $\tube$ in } \tube' \ssm \tube\}| & \text{if } \tube \not\subseteq \tube', \\
0 & \text{otherwise.}
\end{cases}
\]
\end{definition}

\begin{example}
\label{exm:exmCompatibilityDegree}
On the graph~$\graphG\ex$ of Examples~\ref{exm:exmTubes} and~\ref{exm:exmFlip}, the compatibility degrees of the green tubes~$\tube\ex$, $\tube'\ex$ and of the red tube~$\tsup\ex = \tube\ex \cup \tube'\ex$ of \fref{fig:exmFlip} with the tube~$\tube^\circ\ex$ of \fref{fig:exmTubes}\,(left) are given by
\[
\begin{array}{l@{\qquad}l@{\qquad}l}
\compatibilityDegree{\tube\ex}{\tube^\circ\ex} = |\{i\}| = 1, & \compatibilityDegree{\tube'\ex}{\tube^\circ\ex} = |\{g\}| = 1, & \compatibilityDegree{\tsup\ex}{\tube^\circ\ex} = 0, \\
\compatibilityDegree{\tube^\circ\ex}{\tube\ex} = |\{c,m\}| = 2, & \compatibilityDegree{\tube^\circ\ex}{\tube'\ex} = |\{c,e,m\}| = 3, & \compatibilityDegree{\tube^\circ\ex}{\tsup\ex} = 0.
\end{array}
\]
Note that the last~$0$ is forced by the last line of the definition since~$\tube^\circ\ex \subsetneq \tsup\ex$.
\end{example}

We will see in Sections~\ref{subsec:paths} and~\ref{subsec:cycles} that our compatibility degree on tubes of paths (resp.~cycles) corresponds to the compatibility degree on cluster variables in type~$A$ (resp.~$B/C$) cluster algebras defined in~\cite{FominZelevinsky-YSystems}. The compatibility degree in cluster algebras encodes compatibility and exchangeability between cluster variables. The analogous result for the graphical compatibility degree is given by the following proposition, proved in Section~\ref{subsec:proofCompatibilityDegree}.

\begin{proposition}
\label{prop:compatibilityDegree}
For any two tubes~$\tube, \tube'$ of~$\graphG$,
\begin{itemize}
\item $\compatibilityDegree{\tube}{\tube'} < 0 \iff \compatibilityDegree{\tube'}{\tube} < 0 \iff \tube = \tube'$, \label{item:identical}
\item $\compatibilityDegree{\tube}{\tube'} = 0 \iff \compatibilityDegree{\tube'}{\tube} = 0 \iff \tube$ and~$\tube'$ are compatible, and \label{item:compatible}
\item $\compatibilityDegree{\tube}{\tube'} = 1 = \compatibilityDegree{\tube'}{\tube} \iff \tube$ and~$\tube'$ are exchangeable. \label{item:exchangeable}
\end{itemize}
\end{proposition}

\begin{remark}
It can happen that~$\compatibilityDegree{\tube}{\tube'} = 1$ while~$\compatibilityDegree{\tube'}{\tube} \ne 1$, in which case~$\tube$ and~$\tube'$ are not exchangeable. This situation appears as soon as~$\graphG$ contains a cycle or a trivalent vertex. See \eg Example~\ref{exm:exmCompatibilityDegree}.
\end{remark}

Proposition~\ref{prop:compatibilityDegree} should be understood informally as follows: the compatibility degree between two tubes measures how much they are incompatible. It is natural to use this measure to construct fan realizations of the nested complex: intuitively, pairs of tubes with low compatibility degrees should correspond to rays closed to each other. We make this idea precise in the next section.

\subsection{Compatibility fans}

Similar to the $\b{d}$-vector fan defined by S.~Fomin and A.~Zelevinsky~\cite{FominZelevinsky-ClusterAlgebrasII}, we consider the compatibility vectors with respect to an arbitrary initial maximal tubing~$\tubing^\circ$. Remember that any maximal tubing on~$\graphG$ has precisely~$n \eqdef |\ground| - |\connectedComponents(\graphG)|$ tubes.

\begin{definition}
Let~$\tubing^\circ \eqdef \{\tube_1^\circ, \dots, \tube_n^\circ\}$ be an arbitrary initial maximal tubing on~$\graphG$. The \defn{compatibility vector} of a tube~$\tube$ of~$\graphG$ with respect to~$\tubing^\circ$ is the integer vector $\compatibilityVector{\tubing^\circ}{\tube} \eqdef [\compatibilityDegree{\tube_1^\circ}{\tube}, \dots, \compatibilityDegree{\tube_n^\circ}{\tube}]$. The \defn{compatibility matrix} of a tubing~$\tubing \eqdef \{\tube_1, \dots, \tube_m\}$ on~$\graphG$ with respect to~$\tubing^\circ$ is the matrix $\compatibilityVector{\tubing^\circ}{\tubing} \eqdef [\compatibilityDegree{\tube_i^\circ}{\tube_j}]_{i \in [n], j \in [m]}$.
\end{definition}

Remember that we denote by~$\R_{\ge 0} \, \b{M}$ the polyhedral cone generated by the column vectors of a matrix~$\b{M}$. Note that the compatibility vectors of the initial tubes are given by the negative of the basis vectors, while all other compatibility vectors lie in the positive orthant: $\compatibilityVector{\tubing^\circ}{\tubing^\circ} = -\b{I}_n$ and~$\compatibilityVector{\tubing^\circ}{\tube} \in \R_{\ge 0} \, \b{I}_n$ for~$\tube \notin \tubing^\circ$. Our main result asserts that these compatibility vectors support a complete simplicial fan realization of the graphical nested complex.

\begin{theorem}
\label{theo:compatibilityFan}
For any graph~$\graphG$ and any maximal tubing~$\tubing^\circ$ on~$\graphG$, the collection of cones
\[
\compatibilityFan{\graphG}{\tubing^\circ} \eqdef \set{\R_{\ge 0} \, \compatibilityVector{\tubing^\circ}{\tubing}}{\tubing \text{ tubing on } \graphG}
\]
is a complete simplicial fan which realizes the nested complex~$\nestedComplex(\graphG)$. We call it the \defn{compatibility fan} of~$\graphG$ with respect to~$\tubing^\circ$.
\end{theorem}

We prove this statement in Section~\ref{subsec:proofCompatibilityFan}. The proof relies on the characterization of complete simplicial fans presented in Proposition~\ref{prop:characterizationFan}. Unfortunately, we are not able to compute the linear dependence between the compatibility vectors involved in an arbitrary flip. To illustrate the difficulty, we show in the following example that these linear dependences may be complicated. In particular, they do not always  involve only the forced tubes of the flip.

\begin{example}
\label{exm:exmLinearDependence}
Consider the initial maximal tubing~$\tubing^\circ\ex$ of the graph~$\graphG\ex$ of \fref{fig:exmTubes}\,(right) and the flip~$\tubing\ex \ssm \{\tube\ex\} = \tubing'\ex \ssm \{\tube'\ex\}$ illustrated in \fref{fig:exmFlip}.
The linear dependence between the compatibility vectors of the tubes of~$\tubing\ex \cup \tubing'\ex$ with respect to~$\tubing^\circ\ex$ is
\begin{gather*}
2 \, \compatibilityVector{\tubing^\circ\ex}{\tube\ex} + \compatibilityVector{\tubing^\circ\ex}{\tube'\ex} - \compatibilityVector{\tubing^\circ\ex}{\{d\}} - \compatibilityVector{\tubing^\circ\ex}{\{e\}} \qquad \qquad \\ \qquad \qquad - 3 \, \compatibilityVector{\tubing^\circ\ex}{\{m\}} + 4 \, \compatibilityVector{\tubing^\circ\ex}{\{k,l\}} - 3 \, \compatibilityVector{\tubing^\circ\ex}{\{c,d,h\}} = 0.
\end{gather*}
Observe that the tube~$\{d\}$ is involved in this linear dependence although it is not a forced tube of the exchangeable pair~$\{\tube\ex,\tube'\ex\}$.
\end{example}

Theorem~\ref{theo:compatibilityFan} has the following consequences, whose direct proof would require some work.

\begin{corollary}
\label{coro:fullRank}
For any initial tubing~$\tubing^\circ$ on~$\graphG$, 
\begin{itemize}
\item the compatibility vector map~$\tube \longmapsto \compatibilityVector{\tubing^\circ}{\tube}$ is injective: $\compatibilityVector{\tubing^\circ}{\tube} = \compatibilityVector{\tubing^\circ}{\tube'} \; \Longrightarrow \; \tube = \tube'$.
\item the compatibility matrix~$\compatibilityVector{\tubing^\circ}{\tubing}$ of any maximal tubing~$\tubing$ on~$\graphG$ has full rank.
\end{itemize}
\end{corollary}

\subsection{Dual compatibility fan}

It is also interesting to consider the following dual notion of compatibility vectors, where the roles of~$\tube$ and~$\tube_1^\circ, \dots, \tube_n^\circ$ are reversed. The results are similar, and the motivation for this dual definition will become clear in Section~\ref{subsec:cycles}.

\begin{definition}
Let~$\tubing^\circ \eqdef \{\tube_1^\circ, \dots, \tube_n^\circ\}$ be an arbitrary initial maximal tubing on~$\graphG$. The \defn{dual compatibility vector} of a tube~$\tube$ of~$\graphG$ with respect to~$\tubing^\circ$ is the vector $\dualCompatibilityVector{\tube}{\tubing^\circ} \eqdef [\compatibilityDegree{\tube}{\tube_1^\circ}, \dots, \compatibilityDegree{\tube}{\tube_n^\circ}]$. The \defn{dual compatibility matrix} of a tubing~$\tubing \eqdef \{\tube_1, \dots, \tube_m\}$ on~$\graphG$ with respect to~$\tubing^\circ$ is the matrix~$\dualCompatibilityVector{\tubing}{\tubing^\circ} \eqdef [\compatibilityDegree{\tube_j}{\tube_i^\circ}]_{i \in [n], j \in [m]}$.
\end{definition}

The following statement is the analogue of Theorem~\ref{theo:compatibilityFan}.

\begin{theorem}
\label{theo:dualCompatibilityFan}
For any graph~$\graphG$ and any maximal tubing~$\tubing^\circ$ on~$\graphG$, the collection of cones
\[
\dualCompatibilityFan{\graphG}{\tubing^\circ} \eqdef \set{\R_{\ge 0} \, \dualCompatibilityVector{\tubing}{\tubing^\circ}}{\tubing \text{ tubing on } \graphG}\]
is a complete simplicial fan which realizes the nested complex~$\nestedComplex(\graphG)$. We call it the \defn{dual compatibility fan} of~$\graphG$ with respect to~$\tubing^\circ$.
\end{theorem}

The proof of this statement appears in Section~\ref{subsec:proofDualCompatibilityFan}. It is a direct application of Theorem~\ref{theo:compatibilityFan} using duality between compatibility and dual compatibility matrices.

\begin{example}
\label{exm:exmLinearDependenceDual}
Consider the initial maximal tubing~$\tubing^\circ\ex$ on the graph~$\graphG\ex$ of \fref{fig:exmTubes}\,(right) and the flip~$\tubing\ex \ssm \{\tube\ex\} = \tubing'\ex \ssm \{\tube'\ex\}$ illustrated in \fref{fig:exmFlip}. The linear dependence between the dual compatibility vectors of the tubes of~$\tubing\ex \cup \tubing'\ex$ with respect to~$\tubing^\circ\ex$ is
\[
2 \, \dualCompatibilityVector{\tube\ex}{\tubing^\circ\ex} + \dualCompatibilityVector{\tube'\ex}{\tubing^\circ\ex} - \dualCompatibilityVector{\{e\}}{\tubing^\circ\ex} - \dualCompatibilityVector{\{c,d,h\}}{\tubing^\circ\ex} = 0.
\]
\end{example}


\section{Examples for specific graphs}
\label{sec:specificGraphs}

In this section, we provide examples of compatibility fans for particular families of graphs. We start with graphs with few vertices to illustrate the variety of compatibility fans. We then describe compatibility fans for paths, cycles, complete graphs and stars using alternative combinatorial models (triangulations, lattice paths, ...). For paths and cycles, we give an explicit connection to the compatibility degree in cluster algebras of types~$A$, $B$, and~$C$. The examples of this section shall help the intuition for further properties studied in Section~\ref{sec:furtherTopics} and for the proofs gathered in Section~\ref{sec:proofs}. 

\subsection{Graphs with few vertices}
\label{subsec:fewVertices}

\enlargethispage{.3cm}
In view of Proposition~\ref{prop:product} below, we restrict to connected graphs. The only connected graphs with $3$ vertices are the $3$-path and the triangle, whose compatibility fans are represented in \fref{fig:3vertices}. The other possible choices for the initial tubing in these pictures would produce the same fans: it is clear for the triangle as all maximal tubings are obtained from one another by graph isomorphisms; for the path, it is an illustration of the non-trivial isomorphisms between compatibility fans studied in Section~\ref{subsec:many}.

\begin{figure}
  \capstart
  \centerline{\includegraphics[scale=1.05]{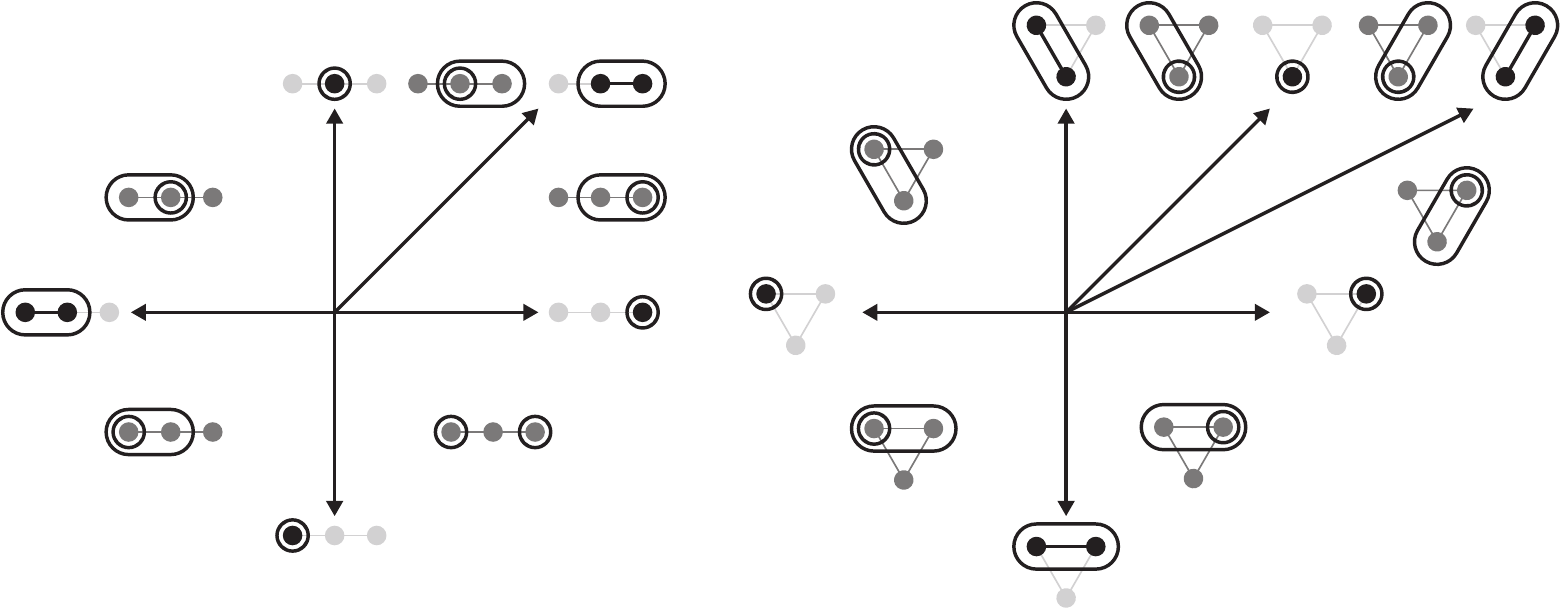}}
  \caption{Compatibility fans of the $3$-path (left) and of the triangle (right).}
  \label{fig:3vertices}
\end{figure}

The first interesting compatibility fans appear in dimension~$3$ for connected graphs on $4$ vertices. All possibilities up to linear transformations are represented in \fref{fig:4vertices}. Instead of representing cones in the $3$-dimensional space, we intersect the compatibility vectors with the unit sphere, make a stereographic projection of the resulting points on the sphere (the pole of the projection is the point of the sphere in direction~$-\b{e}_1-\b{e}_2-\b{e}_3$), and draw the cones on the resulting planar points. Under this projection, the three external vertices correspond to the tubes of the initial tubing, and the external face corresponds to the initial tubing. For the sake of readability, we do not label the remaining vertices of the projection. Their labels can be reconstructed from the initial tubes by flips. For example, the tubes corresponding to the vertices of the top pictures of \fref{fig:4vertices} are given in \fref{fig:allLabels}.

The pictures become more complicated in dimension~$4$. To illustrate them, we have represented in \fref{fig:5vertices} the stereographic projection of the compatibility fan for an arbitrary initial maximal tubing on the path, cycle, complete graph, and star on $5$ vertices.

\begin{figure}[h]
  \capstart
  \centerline{\includegraphics[width=.37\textwidth]{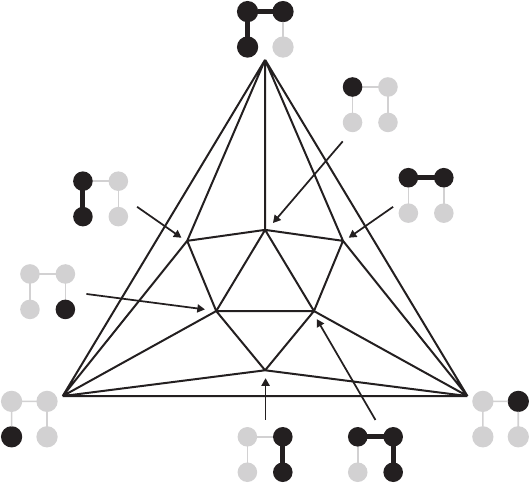} \quad \includegraphics[width=.37\textwidth]{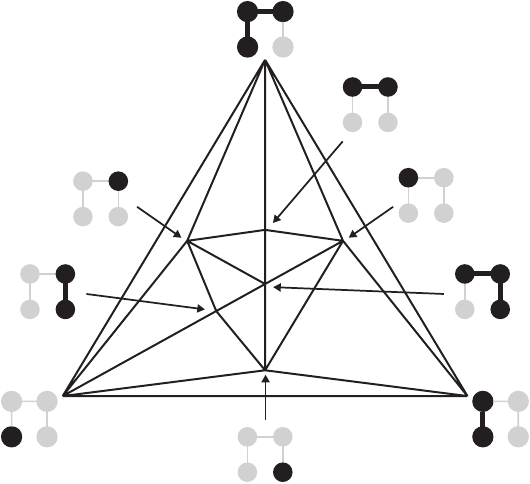} \quad \includegraphics[width=.37\textwidth]{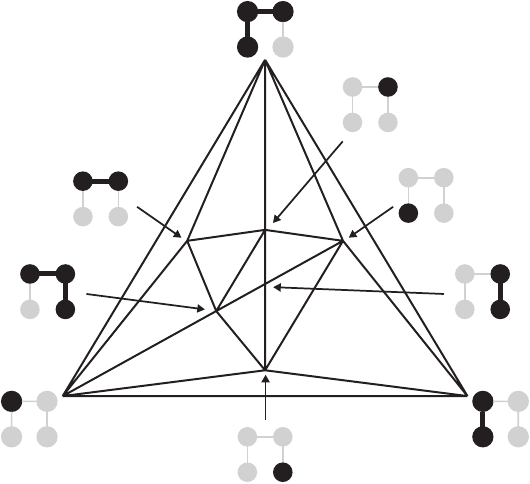}}
  \vspace{-.2cm}
  \caption{All tubes in the top pictures of \fref{fig:4vertices}.}
  \label{fig:allLabels}
\end{figure}

\begin{figure}[p]
  \capstart
  \renewcommand{\arraystretch}{1.3}
  \setlength{\extrarowheight}{3.4cm}
  \centerline{
    \begin{tabular}{@{}|c|ccc|}
    \hline
    \rotatebox{90}{\textbf{Path} (associahedron)} & \includegraphics[width=.37\textwidth]{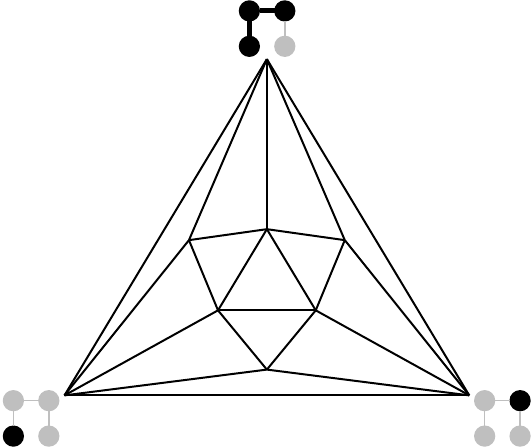} & \includegraphics[width=.37\textwidth]{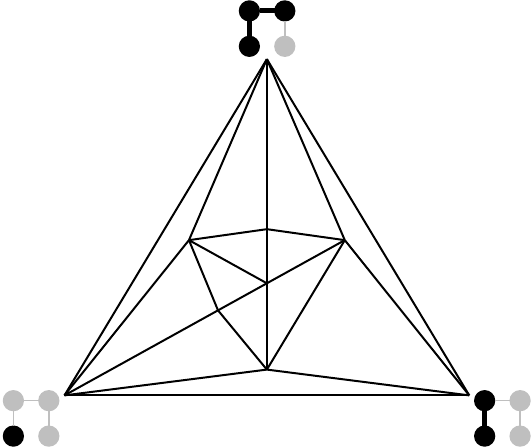} & \includegraphics[width=.37\textwidth]{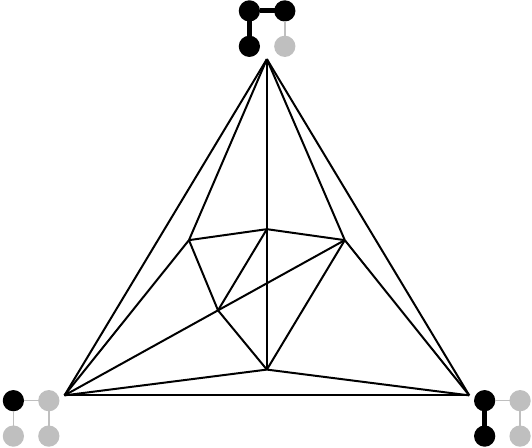} \\
    \hline
    \end{tabular}
  }
  
  \centerline{  
    \begin{tabular}{@{}|c|ccc|}
    \hline
    \rotatebox{90}{\textbf{Cycle} (cyclohedron)} & \includegraphics[width=.37\textwidth]{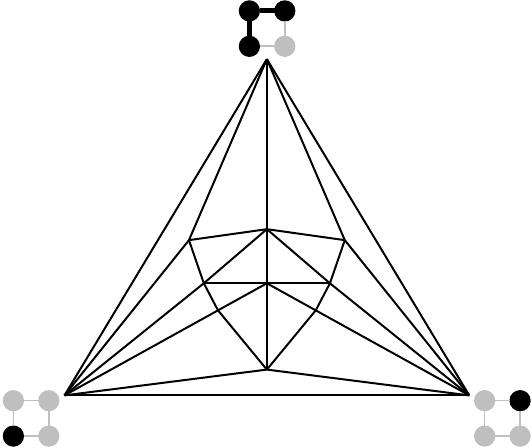} & \includegraphics[width=.37\textwidth]{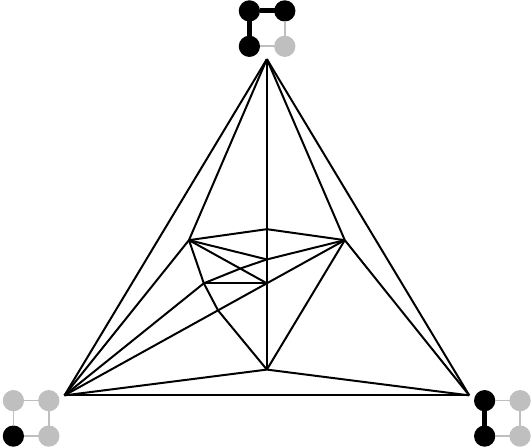} & \includegraphics[width=.37\textwidth]{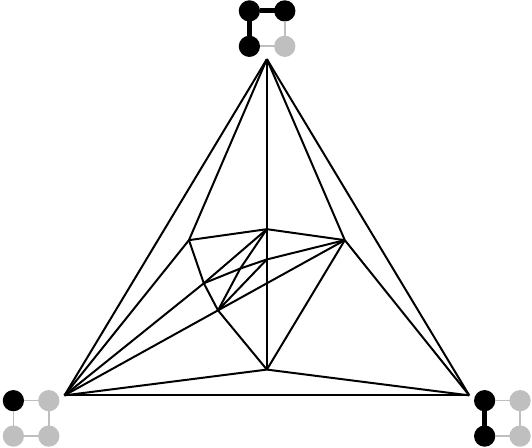} \\
    \hline
    \end{tabular}
  }
  
  \setlength{\extrarowheight}{0cm}
  \centerline{  
    \begin{tabular}[t]{@{}|@{\hspace{.13cm}}c@{\hspace{.13cm}}|}
    \hline
    \textbf{Complete graph} \\[-.15cm] (permutahedron) \\
    \hline
    \\[-.4cm]
    \includegraphics[width=.37\textwidth]{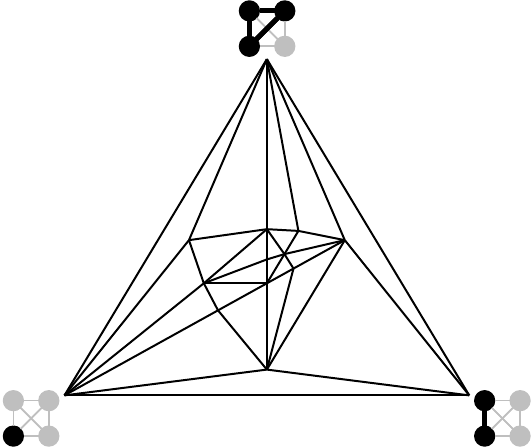} \\
    \hline
    \end{tabular}
    \!\!\!\!\!
    \setlength{\extrarowheight}{3.4cm}
    \begin{tabular}[t]{@{}|c|cc|}
    \hline
     & \includegraphics[width=.37\textwidth]{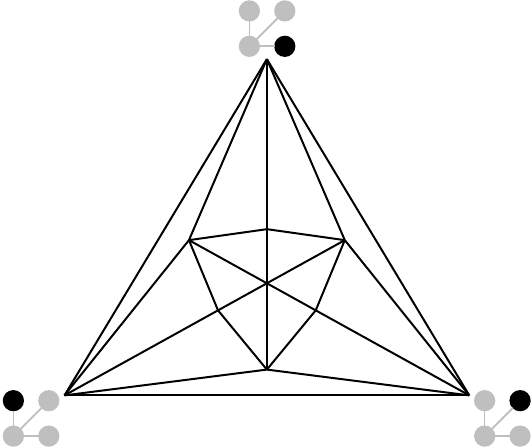} & \includegraphics[width=.37\textwidth]{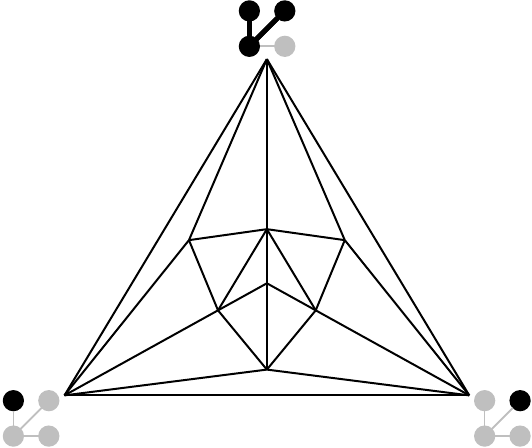} \\
    \rotatebox{90}{\textbf{Star} (stellohedron)} & \includegraphics[width=.37\textwidth]{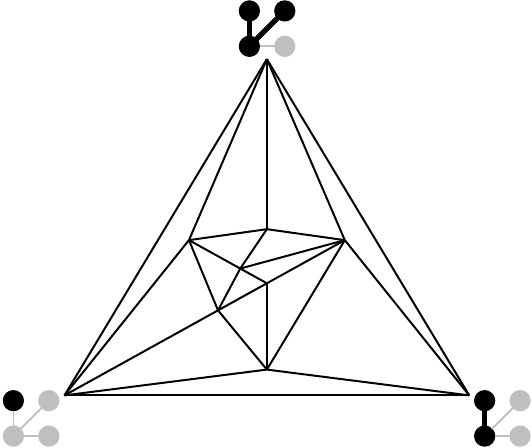} & \includegraphics[width=.37\textwidth]{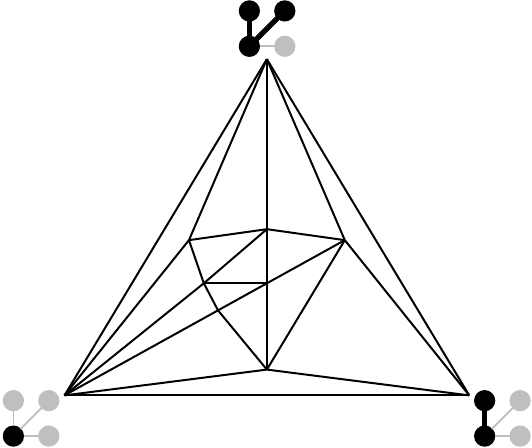} \\
    \hline
    \end{tabular}
  }
    
  \caption{All possible compatibility fans up to linear isomorphism, for all connected graphs on $4$ vertices (see also the end of the picture on page~\pageref{fig:4verticesEnd} for the two remaining graphs). Instead of representing the cones in the $3$-dimensional space, we intersect the compatibility vectors with the unit sphere, make a stereographic projection of the resulting points on the sphere (the pole of the projection is the point of the sphere in direction~$-\b{e}_1-\b{e}_2-\b{e}_3$), and draw the cones on the resulting planar points.}
  \label{fig:4vertices}
\end{figure}

\begin{figure}[p]
  \renewcommand{\arraystretch}{1.3}
  \setlength{\extrarowheight}{3.4cm}
  \centerline{
    \begin{tabular}{@{}|c|ccc|}
    \hline
    & \includegraphics[width=.37\textwidth]{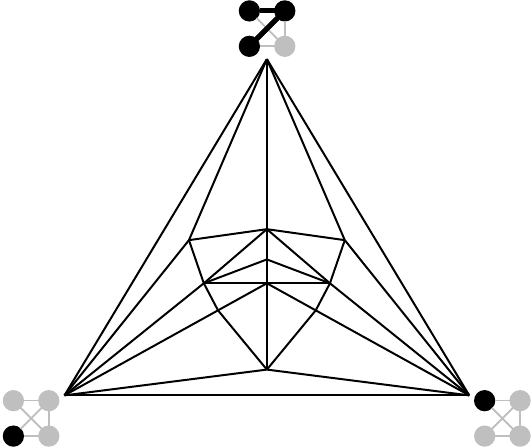} & \includegraphics[width=.37\textwidth]{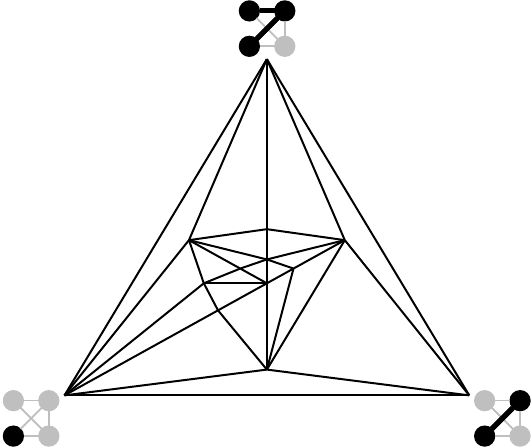} & \includegraphics[width=.37\textwidth]{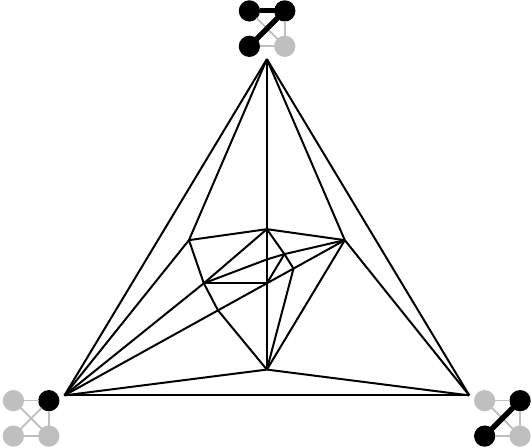} \\
    \rotatebox{90}{\textbf{Complete graph minus one edge} \hspace{-2cm}} & \includegraphics[width=.37\textwidth]{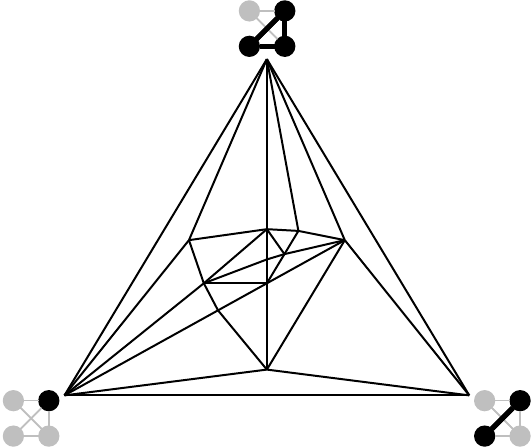} & \includegraphics[width=.37\textwidth]{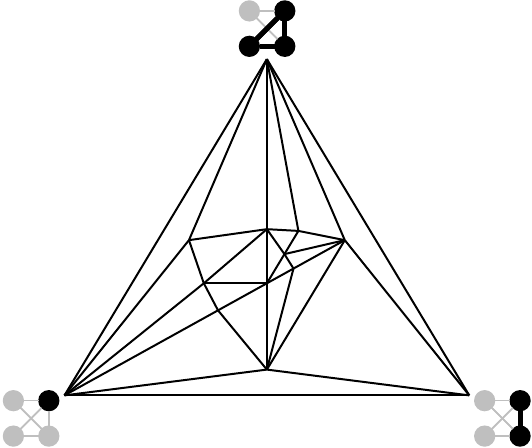} & \includegraphics[width=.37\textwidth]{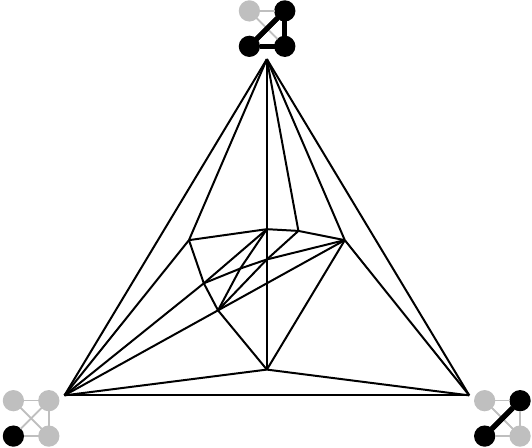} \\
    \hline
    \end{tabular}
  }
  
  \centerline{
    \begin{tabular}{@{}|c|ccc|}
    \hline
    & \includegraphics[width=.37\textwidth]{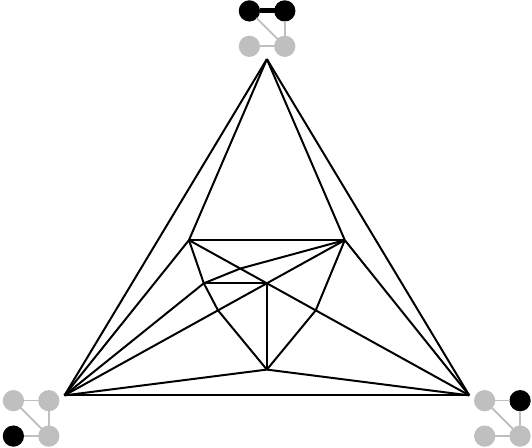} & \includegraphics[width=.37\textwidth]{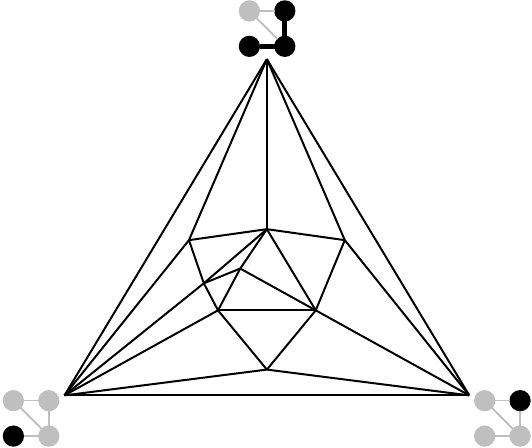} & \includegraphics[width=.37\textwidth]{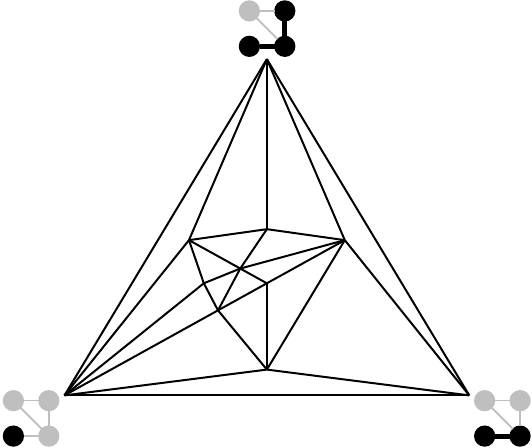} \\
    & \includegraphics[width=.37\textwidth]{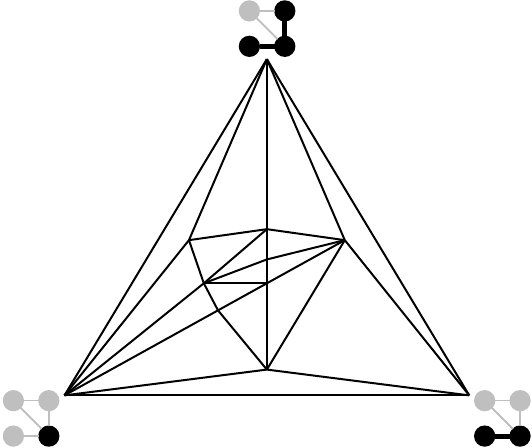} & \includegraphics[width=.37\textwidth]{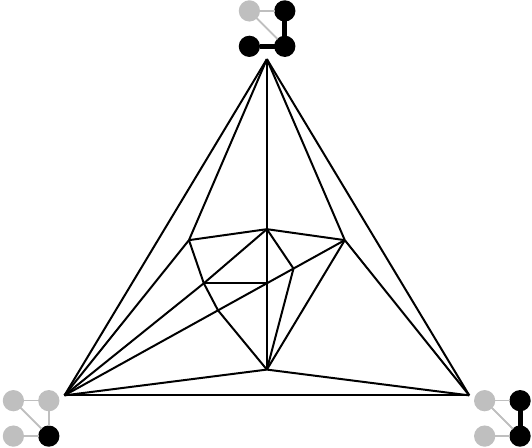} & \includegraphics[width=.37\textwidth]{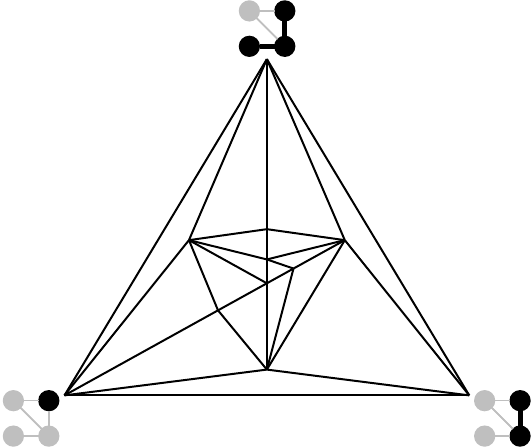} \\
    \rotatebox{90}{\textbf{Complete graph minus two incident edges} \hspace{-4cm}} & \includegraphics[width=.37\textwidth]{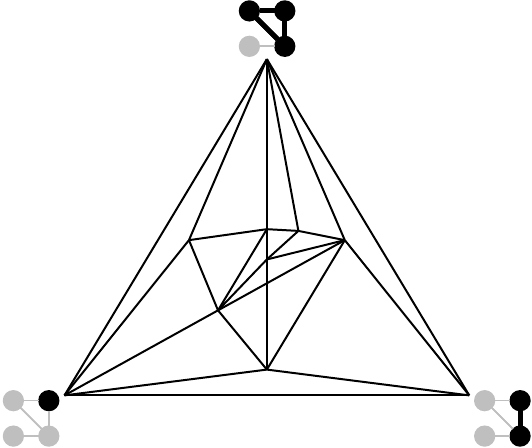} & \includegraphics[width=.37\textwidth]{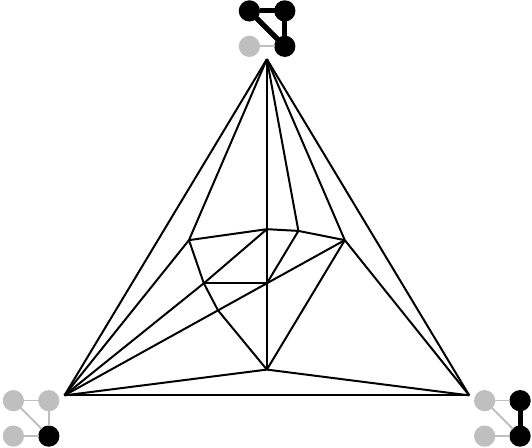} & \includegraphics[width=.37\textwidth]{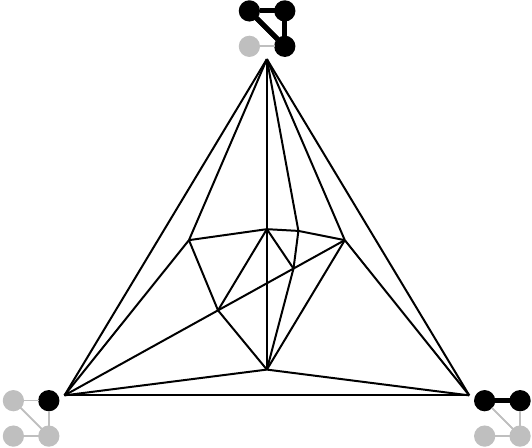}
     \\
    \hline
    \end{tabular}
  }
  \label{fig:4verticesEnd}
\end{figure}

\begin{figure}
  \capstart
  \centerline{\includegraphics[width=.55\textwidth]{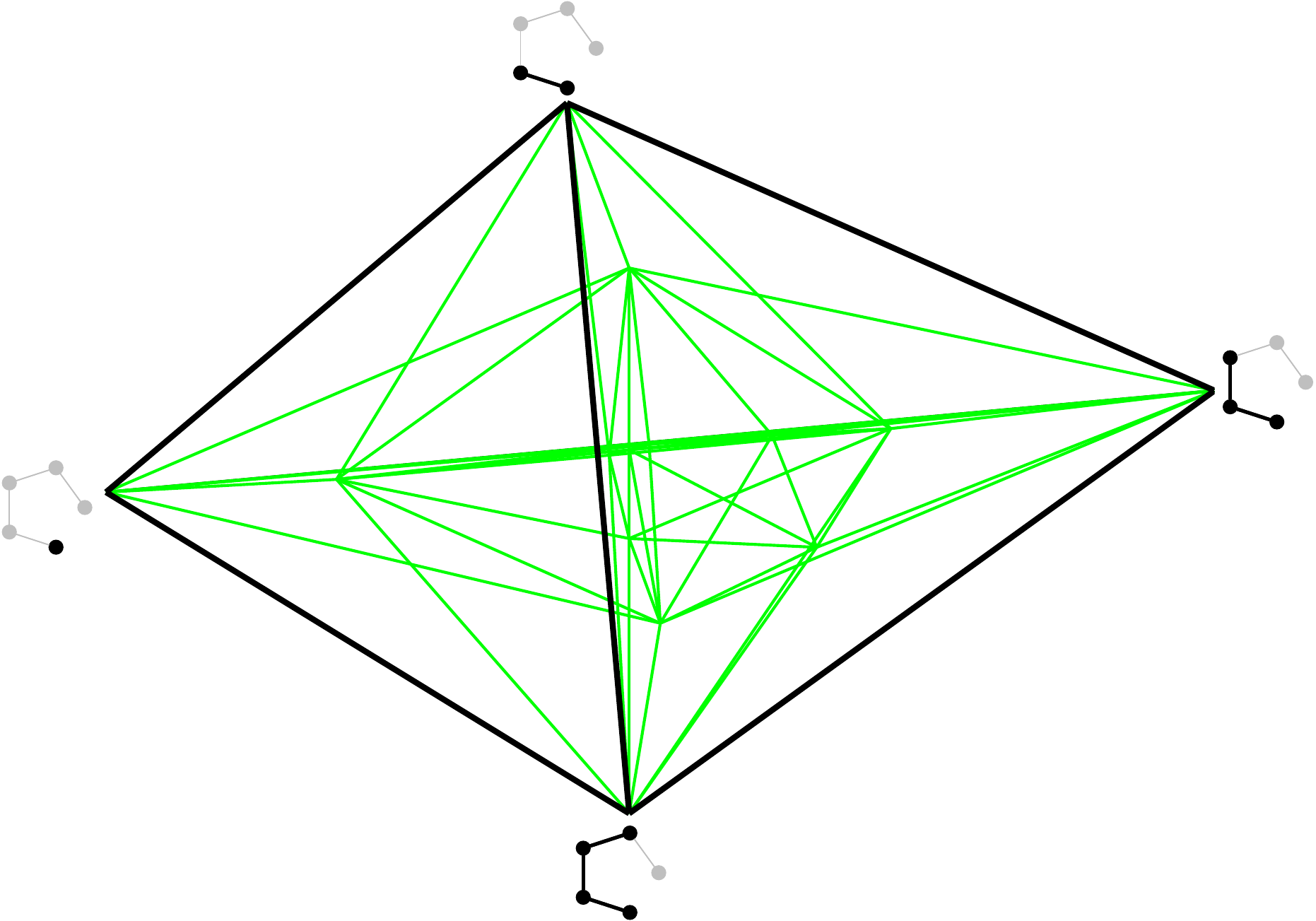} \quad \includegraphics[width=.55\textwidth]{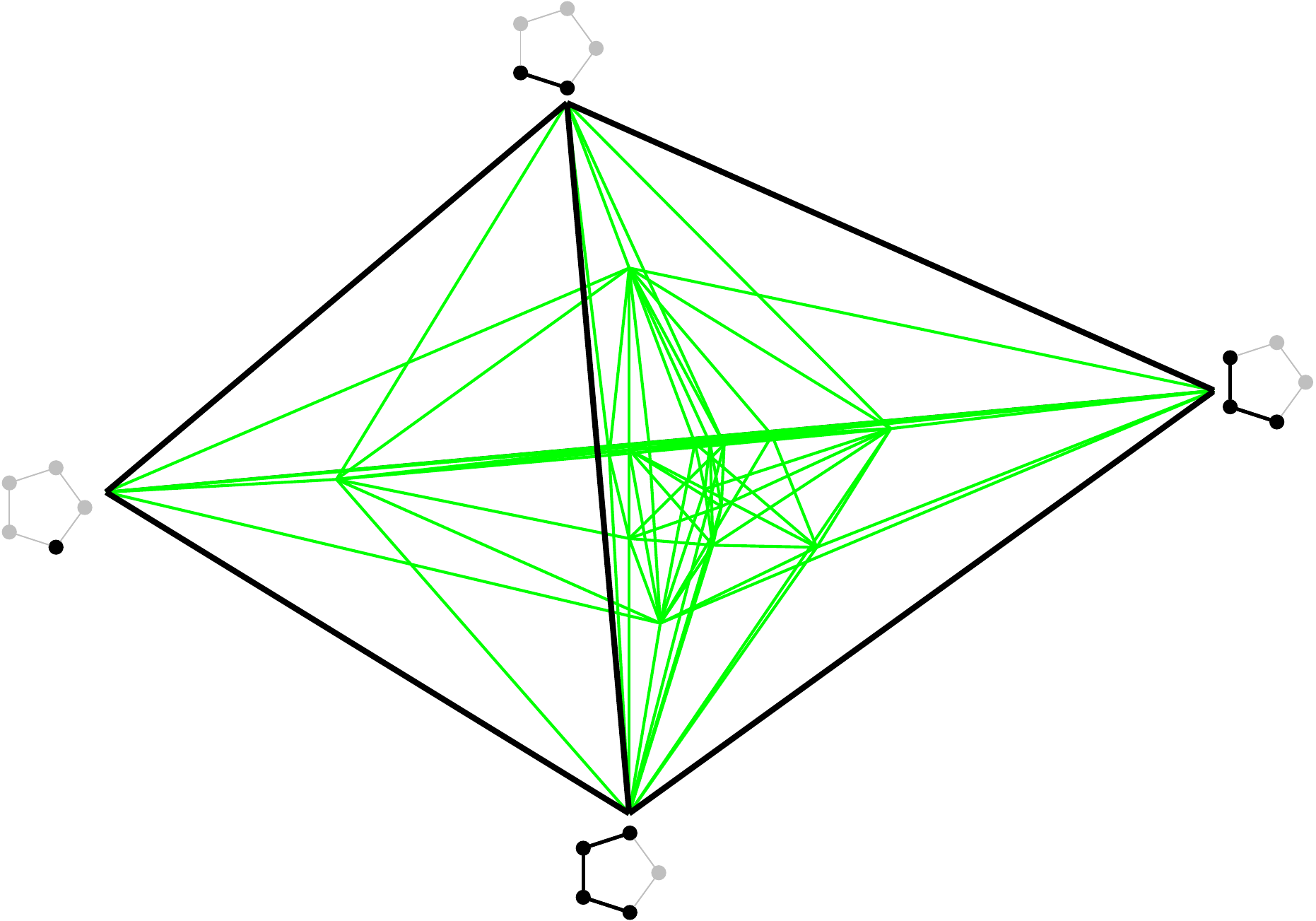}}
  \vspace*{.4cm}
  \centerline{\includegraphics[width=.55\textwidth]{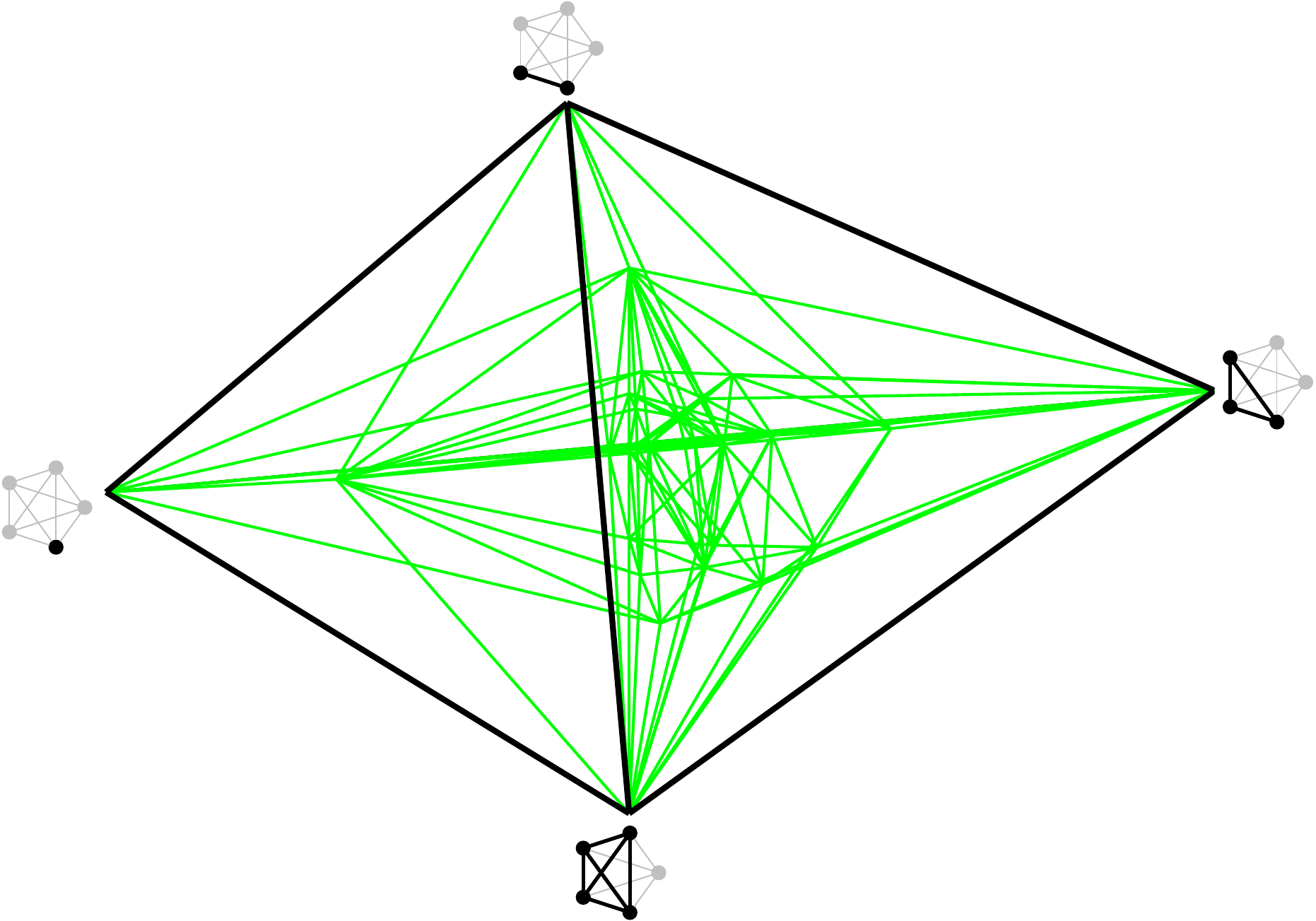} \quad \includegraphics[width=.55\textwidth]{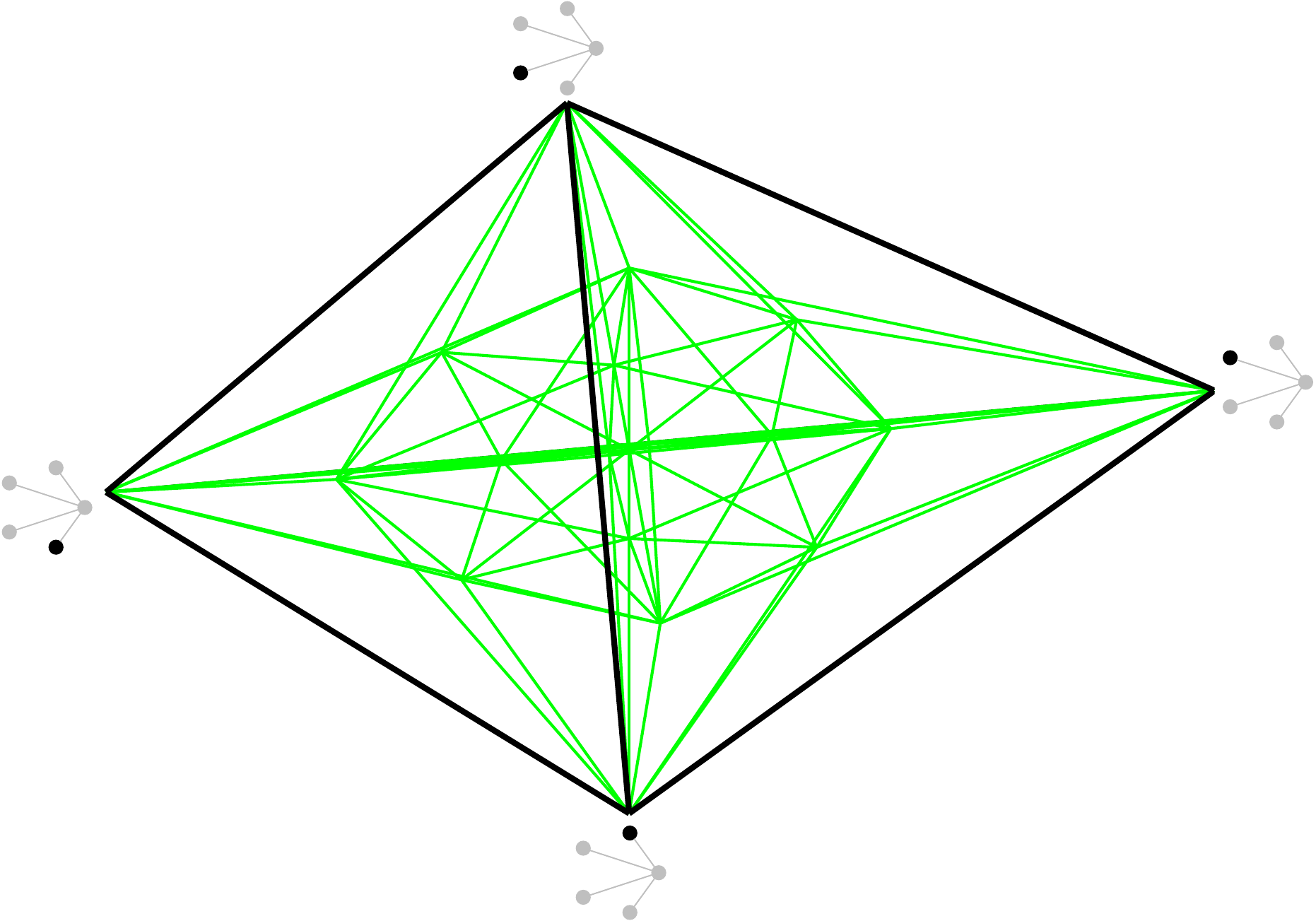}}
  \caption{Stereographic projection of the compatibility fan for particular initial maximal tubings on the path, cycle, complete graph, and star on $5$ vertices.}
  \label{fig:5vertices}
\end{figure}

\subsection{Paths}
\label{subsec:paths}

We now consider the nested complex~$\nestedComplex(\pathG_{n+1})$ and the compatibility fan~$\compatibilityFan{\pathG_{n+1}}{\tubing^\circ}$ for the path~$\pathG_{n+1}$ on~$n+1$ vertices. As already mentioned in the introduction, the nested complex~$\nestedComplex(\pathG_{n+1})$ is isomorphic to the $n$-dimensional \defn{simplicial associahedron}, \ie the simplicial complex of sets of pairwise non-crossing diagonals of an $(n+3)$-gon. It is convenient to present the correspondence as follows. Consider an $(n+3)$-gon~$Q_{n+3}$ with vertices labeled from left to right by~$0, 1, \dots, n+2$ and such that all vertices~$1, \dots, n+1$ are located strictly below the boundary edge~$[0,n+2]$. We can therefore identify the path~$\pathG_{n+1}$ with the path~$1, \dots, n+1$ on the boundary of~$Q_{n+3}$. We then associate to a diagonal~$\delta$ of~$Q_{n+3}$ the tube~$\tube_\delta$ of~$\pathG_{n+1}$ whose vertices are located strictly below~$\delta$, see Figures~\ref{fig:path} and~\ref{fig:associahedron}. Finally, we associate to a set~$\Delta$ of pairwise non-crossing internal diagonals of~$Q_{n+3}$ the set of tubes~$\tubing_\Delta \eqdef \set{\tube_\delta}{\delta \in \Delta}$, see \fref{fig:associahedron}. The reader can check that the map~$\Delta \mapsto \tubing_\Delta$ defines an isomorphism between the simplicial associahedron and the nested complex~$\nestedComplex(\pathG_{n+1})$: two diagonals~$\delta, \delta'$ of~$Q_{n+3}$ are non-crossing if and only if the corresponding tubes~$\tube_\delta, \tube_{\delta'}$ of~$\pathG_{n+1}$ are compatible.

\begin{figure}
  \capstart
  \centerline{\includegraphics[scale=.6]{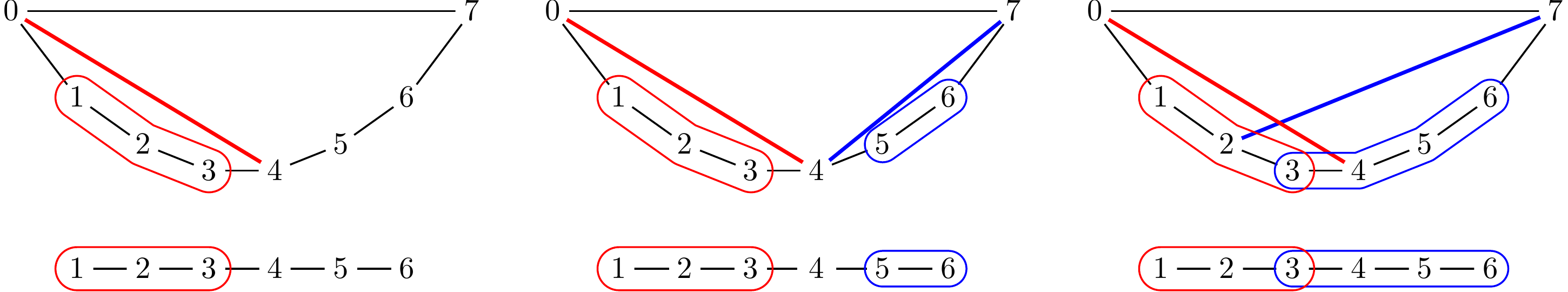}}
  \caption{Isomorphism between the simplicial associahedron and the nested complex of a path: diagonals are sent to tubes (left), preserving the compatibility (middle) and incompatibility (right). See also \fref{fig:associahedron}.}
  \label{fig:path}
\end{figure}

It follows by classical results on the associahedron that the path~$\pathG_{n+1}$ has:
\begin{itemize}
\item $\displaystyle{\frac{n(n+3)}{2}}$ proper tubes~\href{https://oeis.org/A000096}{\cite[A000096]{OEIS}} (internal diagonals of the $(n+3)$-gon),
\item $\displaystyle{\frac{1}{n+2} \binom{2n+2}{n+1}}$ maximal tubings~\href{https://oeis.org/A000108}{\cite[A000108]{OEIS}} (triangulations of the $(n+3)$-gon),
\item $\displaystyle{\frac{1}{k+1}\binom{n}{k}\binom{n+k+2}{k}}$ tubings with~$k$ tubes~\href{https://oeis.org/A033282}{\cite[A033282]{OEIS}} (dissections of the $(n+3)$-gon into~$k$ parts).
\end{itemize}

The following statement, whose proof is left to the reader, describes the behavior of the map~${\delta \mapsto \tube_\delta}$ with respect to compatibility degrees.

\begin{proposition}
\label{prop:typeA}
For any two diagonals~$\delta, \delta'$ of~$Q_{n+3}$, the compatibility degree of the corresponding tubes~$\tube_\delta$ and~$\tube_{\delta'}$ of~$\pathG_{n+1}$ is given by
\[
\compatibilityDegree{\tube_\delta}{\tube_{\delta'}} =
\begin{cases}
-1 & \text{if $\delta = \delta'$,} \\
0 & \text{if $\delta \ne \delta'$ do not cross,} \\
1 & \text{if $\delta \ne \delta'$ cross.}
\end{cases}
\]
\end{proposition}

In other words, our compatibility degree between tubes of~$\pathG_{n+1}$ coincides with the compatibility degree between type~$A$ cluster variables defined by S.~Fomin and A.~Zelevinsky in~\cite{FominZelevinsky-YSystems}, and our graphical compatibility fan coincides with the type~$A$ compatibility fan defined for an acyclic initial cluster in~\cite{FominZelevinsky-ClusterAlgebrasII} and for any initial cluster in~\cite[Section~5]{CeballosSantosZiegler}. We thus obtain an alternative proof of F.~Santos' result~\cite[Section~5]{CeballosSantosZiegler}.

\begin{corollary}
For type~$A$ cluster algebras, the denominator vectors (or compatibility vectors) of all cluster variables with~respect to any initial cluster support a complete simplicial fan which realizes the cluster complex.
\end{corollary}

\begin{remark}[Dual compatibility fan]
\label{rem:primalDualPath}
The compatibility fan~$\compatibilityFan{\pathG_{n+1}}{\tubing^\circ}$ and the dual compatibility fan~$\dualCompatibilityFan{\pathG_{n+1}}{\tubing^\circ}$ coincide since the compatibility degree is symmetric for tubes~of~$\pathG_{n+1}$.
\end{remark}

\begin{remark}[Linear dependences]
In the case of the path~$\pathG_{n+1}$, the linear dependences are explicity described in~\cite{CeballosSantosZiegler}. They are derived from the case of the octagon by edge contraction in the interpretation in terms of triangulations. They can only involve the two flipped tubes and the forced tubes, and the coefficients are either~$1$ or~$2$ for the flipped tubes and~$-1$ or~$0$ for the forced tubes. See Section~\ref{subsec:proofPolytopality} for more details.
\end{remark}

\begin{remark}
\label{rem:naive}
\enlargethispage{1.2cm}
The compatibility degree for tubes of a path takes values in~$\{-1, 0, 1\}$. It is tempting to construct compatibility fans for graphical nestohedra using the naive compatibility degree defined by~$\compatibilityDegree{\tube}{\tube'} = -1$ if~$\tube = \tube'$, $\compatibilityDegree{\tube}{\tube'} = 0$ if $\tube \ne \tube'$ are compatible, and~$\compatibilityDegree{\tube}{\tube'} = 1$ if $\tube \ne \tube'$ are incompatible. This naive approach works for the paths but fails for any other connected graph since two distinct tubes would get the same compatibility vectors. See \fref{fig:ctrexmNaive} for examples on the triangle and on the tripod.

\begin{figure}[h]
  \capstart
  \centerline{\includegraphics[width=\textwidth]{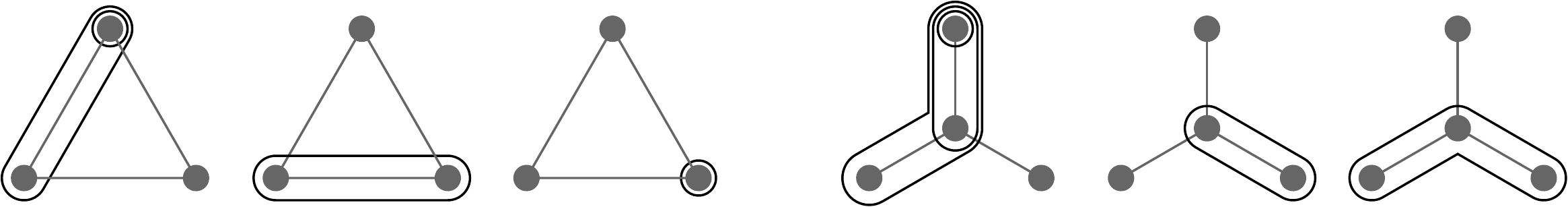}}
  \caption{Counter-examples to the naive definition of compatibility degrees: both on the triangle and on the tripod, all tubes of the initial maximal tubing on the left are incompatible with the two distinct tubes on the right.}
  \label{fig:ctrexmNaive}
\end{figure}
\end{remark}

\subsection{Cycles}
\label{subsec:cycles}

We now consider the nested complex~$\nestedComplex(\cycleG_{n+1})$ and the compatibility fan~$\compatibilityFan{\cycleG_{n+1}}{\tubing^\circ}$ for the cycle~$\cycleG_{n+1}$ on~$n+1$ vertices. As already mentioned in the introduction, the nested complex~$\nestedComplex(\cycleG_{n+1})$ is isomorphic to the $n$-dimensional \defn{simplicial cyclohedron}, \ie the simplicial complex of sets of pairwise non-crossings pairs of centrally symmetric internal diagonals (including duplicated long diagonals) of a regular $(2n+2)$-gon~$R_{2n+2}$. The explicit correspondence works as follows. We label the vertices of~$R_{2n+2}$ cyclically with two copies of~$[n+1]$. We then associate
\begin{itemize}
\item to a duplicated long diagonal~$\delta$ with vertices labeled by~$i$ the tube~$\tube_\delta \eqdef [n+1] \ssm \{i\}$ of~$\cycleG_{n+1}$,
\item to a pair of centrally symmetric diagonals~$\{\delta, \bar\delta\}$ the tube~$\tube_\delta$ of~$\cycleG_{n+1}$ which consists of the labels of the vertices of~$R_{2n+2}$ separated from the center of~$R_{2n+2}$ by~$\delta$ and~$\bar\delta$.
\end{itemize}
Finally, we associate to a set~$\Delta$ of pairwise non-crossing pairs of centrally symmetric internal diagonals of~$R_{2n+2}$ the set of tubes~$\tubing_\Delta \eqdef \set{\tube_\delta}{\delta \in \Delta}$. See Figures~\ref{fig:cycle} and~\ref{fig:cyclohedron}. The reader can check that the map~$\Delta \mapsto \tubing_\Delta$ defines an isomorphism between the simplicial cyclohedron and the nested complex~$\nestedComplex(\cycleG_{n+1})$: two pairs of centrally symmetric diagonals (or duplicated long diagonals)~$\{\delta, \bar\delta\}$ and~$\{\delta', \bar\delta'\}$ of~$R_{2n+2}$ are non-crossing if and only if the corresponding tubes~$\tube_\delta$ and~$\tube_{\delta'}$ of~$\cycleG_{n+1}$ are compatible.

\begin{figure}[h]
  \capstart
  \centerline{\includegraphics[scale=1]{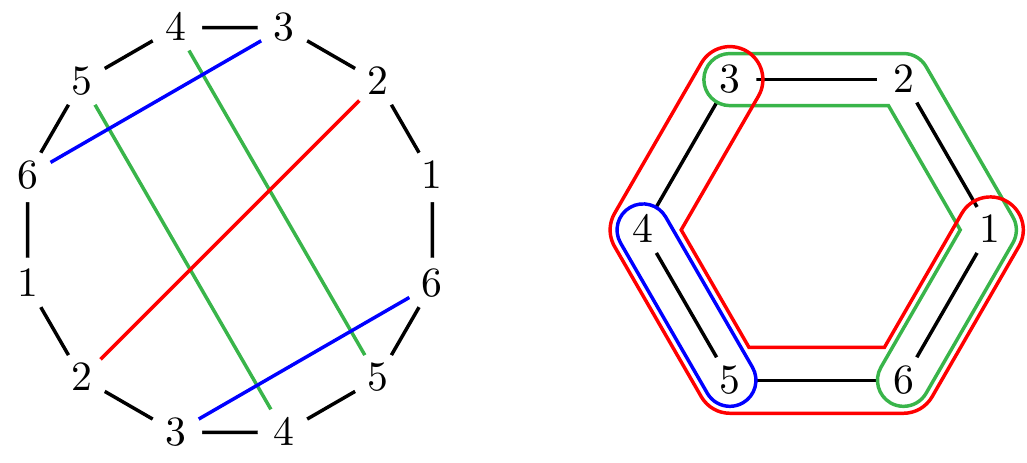}}
  \caption{Isomorphism between the simplicial cyclohedron and the nested complex of a cycle: centrally symmetric pairs of diagonals are sent to tubes, preserving the compatibility and incompatibility. See also \fref{fig:cyclohedron}.}
  \label{fig:cycle}
\end{figure}

It follows by classical results on the cyclohedron that the cycle~$\cycleG_{n+1}$ has:
\begin{itemize}
\item $n(n+1)$ proper tubes~\href{https://oeis.org/A002378}{\cite[A002378]{OEIS}} (centrally symmetric pairs of diagonals),
\item $\displaystyle{\binom{2n}{n}}$ maximal tubings~\href{https://oeis.org/A000984}{\cite[A000984]{OEIS}} (centrally symmetric triangulations),
\item $\displaystyle{\binom{n}{k}\binom{n+k}{k}}$ tubings with~$k$ tubes~\href{https://oeis.org/A063007}{\cite[A063007]{OEIS}} (centrally symmetric dissections).
\end{itemize}

The following statement, whose proof is left to the reader, describes the behavior of the map~$\{\delta,\bar\delta\} \mapsto \tube_\delta$ with respect to compatibility degrees.

\begin{proposition}
\label{prop:typeBC}
For any two pairs of centrally symmetric diagonals (or duplicated long diagonals)~$\{\delta, \bar\delta\}$ and~$\{\delta', \bar\delta'\}$ of~$R_{2n+2}$, the compatibility degree~$\compatibilityDegree{\tube_\delta}{\tube_{\delta'}}$ of the corresponding tubes~$\tube_\delta$ and~$\tube_{\delta'}$ of~$\cycleG_{n+1}$ is the number of crossings between the two diagonals~$\delta$ and~$\bar\delta$ and the diagonal~$\delta'$.
\end{proposition}

In other words, our compatibility degree (resp.~dual compatibility degree) between tubes of~$\cycleG_{n+1}$ coincides with the compatibility degree between type~$C$ (resp.~type~$B$) cluster variables defined by S.~Fomin and A.~Zelevinsky in~\cite{FominZelevinsky-YSystems}. Moreover, our graphical compatibility fan (resp.~dual compatibility fan) coincides with the type~$C$ (resp.~type~$B$) compatibility fan defined for an acyclic initial cluster in~\cite{FominZelevinsky-ClusterAlgebrasII}. This extends for any arbitrary initial cluster to the following corollary.

\begin{corollary}
For type~$B$ and~$C$ cluster algebras, the denominator vectors (or compatibility vectors) of all cluster variables with~respect to any initial cluster support a complete simplicial fan which realizes the cluster complex.
\end{corollary}

\begin{remark}[Dual compatibility fan]
Since the compatibility degree is not symmetric for tubes of~$\cycleG_{n+1}$, the compatibility fan~$\compatibilityFan{\cycleG_{n+1}}{\tubing^\circ}$ and the dual compatibility fan~$\dualCompatibilityFan{\cycleG_{n+1}}{\tubing^\circ}$ do not coincide. Figures~\ref{fig:dualCompatibilityFanCycle3} and~\ref{fig:dualCompatibilityFanCycle4} show both fans for different initial tubings on the cycles~$\cycleG_3$ and~$\cycleG_4$.
\begin{figure}[h]
  \capstart
  \centerline{\includegraphics[scale=1.05]{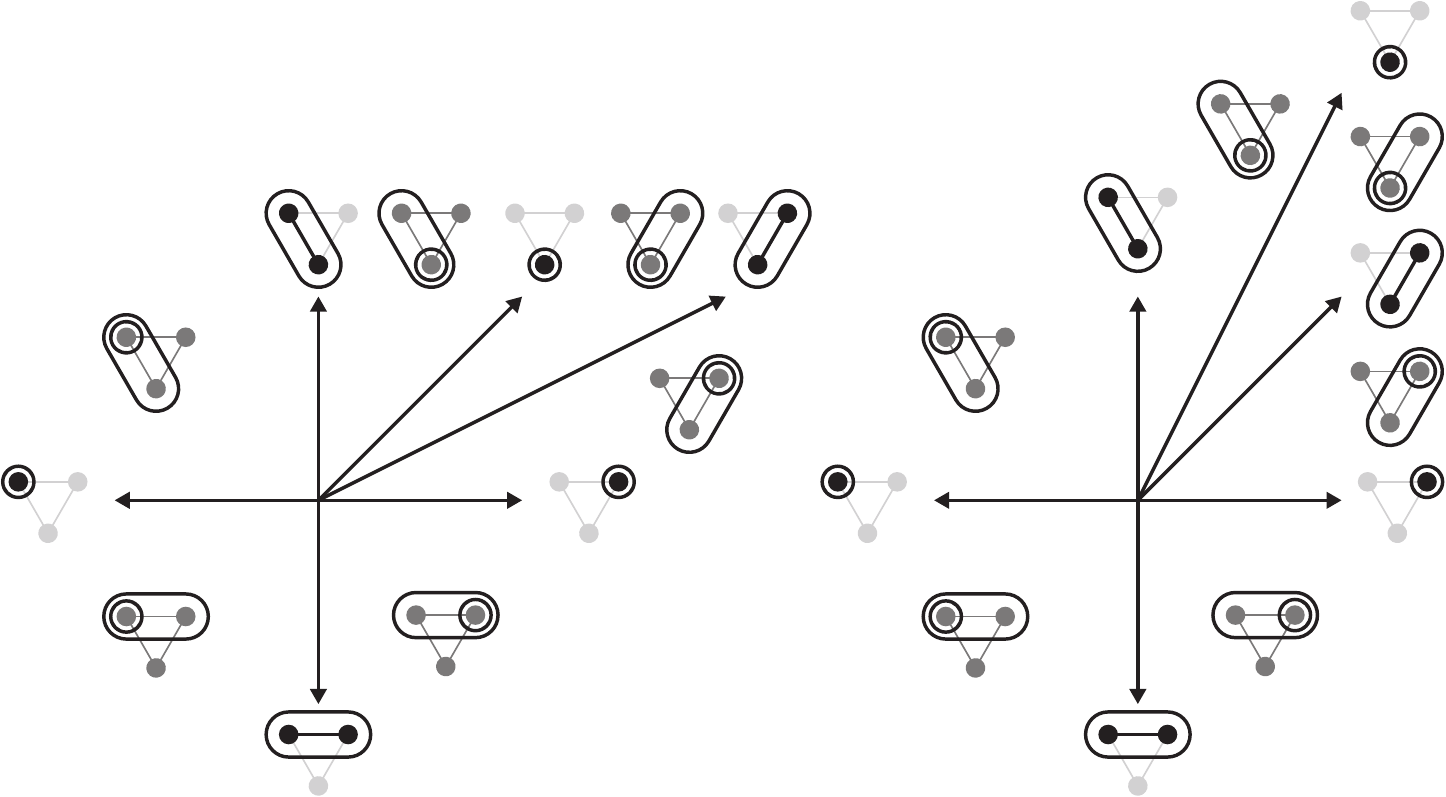}}
  \caption{Compatibility (left) and dual compatibility (right) fans for the triangle.}
  \label{fig:dualCompatibilityFanCycle3}
\end{figure}
\begin{figure}[h]
  \capstart
  \renewcommand{\arraystretch}{1.3}
  \setlength{\extrarowheight}{3.4cm}
  \centerline{  
    \begin{tabular}{@{}|c|ccc|}
    \hline
    \rotatebox{90}{\textbf{Compatibility fan}} & \includegraphics[width=.37\textwidth]{3dimensionalCompatibilityFansLabeled/cycle1} & \includegraphics[width=.37\textwidth]{3dimensionalCompatibilityFansLabeled/cycle2} & \includegraphics[width=.37\textwidth]{3dimensionalCompatibilityFansLabeled/cycle3} \\
	\hline
	\hline
    \rotatebox{90}{\textbf{Dual compatibility fan}} & \includegraphics[width=.37\textwidth]{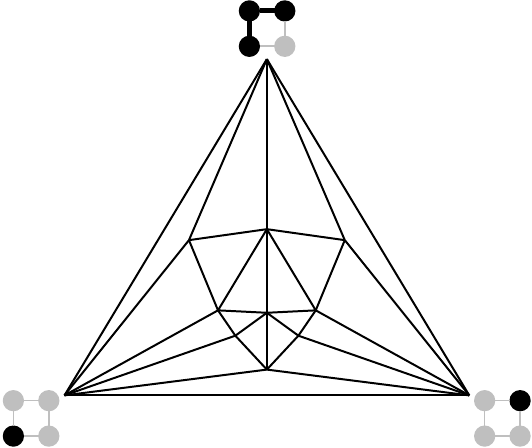} & \includegraphics[width=.37\textwidth]{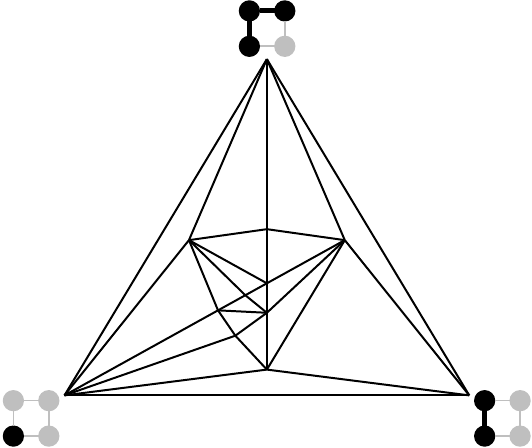} & \includegraphics[width=.37\textwidth]{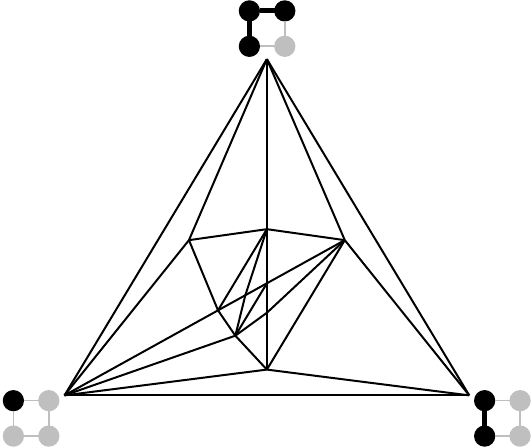} \\
    \hline
    \end{tabular}
  }
  \caption{Compatibility (top) and dual compatibility (bottom) fans for the cycle on $4$ vertices with respect to different initial tubings.}
  \vspace{-1.7cm}
  \label{fig:dualCompatibilityFanCycle4}
\end{figure}
\end{remark}

\newpage
\begin{remark}[Linear dependences]
As for paths, only finitely many linear dependences occur for all cycles~$\cycleG_{n+1}$, both  on compatibility vectors as on dual compatibility vectors. Indeed, with the interpretation of the maximal tubings in terms of centrally symmetric triangulations, the same kind of arguments as in~\cite{CeballosSantosZiegler} ensure that all these dependences can be inferred by checking the cycle~$\cycleG_{8}$ on~$8$ vertices. As for the path, these linear dependences only involve flipped and forced tubes, and the coefficients of the flipped tubes may only be~$1$ or~$2$ and these of the forces tubes may only be~$0,-1$ or~$-2$. See Section~\ref{subsec:proofPolytopality} for more details.

It is easy to find an example of a maximal tubing on the tripod such that one of the linear dependences obtained with respect to this maximal tubing does not only involve forced tubes. It implies in particular that the paths and cycles are the only graphs that have this property. It is then tempting to ask whether it is a coincidence that these graphs also are the only ones whose corresponding associahedra also are generalized associahedra.
\end{remark}

\subsection{Complete graphs}
We now consider the nested complex~$\nestedComplex(\completeG_{n+1})$ and the compatibility fan~$\compatibilityFan{\completeG_{n+1}}{\tubing^\circ}$ for the complete graph~$\completeG_{n+1}$ on~$n+1$ vertices. As already mentioned in the introduction, the nested complex~$\nestedComplex(\completeG_{n+1})$ is isomorphic to the $n$-dimensional \defn{simplicial permutahedron}, \ie the simplicial complex of collections of pairwise nested subsets of~$[n+1]$. See \fref{fig:permutahedron}.

It follows by classical results on the permutahedron that the complete graph~$\completeG_{n+1}$ has:
\begin{itemize}
\item $2^n-2$ proper tubes~\href{https://oeis.org/A000918}{\cite[A000918]{OEIS}} (proper subsets of~$[n]$),
\item $n!$ maximal tubings~\href{https://oeis.org/A000142}{\cite[A000142]{OEIS}} (permutations of~$[n]$),
\item $ k! \, S(n,k)$ tubings with~$k$ tubes, where~$S(n,k)$ is the Stirling number of second kind (\ie the number of ways to partition a set of $n$ elements into $k$ non-empty subsets)~\mbox{\href{https://oeis.org/A008277}{\cite[A008277]{OEIS}}}.
\end{itemize}

For two tubes~$\tube, \tube'$ of~$\completeG_{n+1}$, the compatibility degree of~$\tube$ with~$\tube'$ is $\compatibilityDegree{\tube}{\tube'} = -1$ if~$\tube = \tube'$, $\compatibilityDegree{\tube}{\tube'} = 0$ if~$\tube$ and~$\tube'$ are distinct and nested, and~$\compatibilityDegree{\tube}{\tube'} = |\tube' \ssm \tube|$ otherwise.
This connects the compatibility vector~$\compatibilityVector{\tubing^\circ}{\tube}$ to an alternative combinatorial model for the permutahedron in terms of lattice paths. Since all maximal tubings are equivalent, we can assume that~$\tubing^\circ = \set{[i]}{i \in [n]}$. For any tube~$\tube$ of~$\completeG_{n+1}$, we consider the lattice paths~$\phi(\tube)$ and~$\psi(\tube)$ whose horizontal steps above abscissa~$[i,i+1]$ lie at height~$|\tube \ssm [i]|$ and~$\compatibilityDegree{[i]}{\tube}$ respectively. These lattice paths are illustrated in \fref{fig:latticePaths}, where~$\phi(\tube)$ is the plain path while~$\psi(\tube)$ is dotted until it meets~$\phi(\tube)$. The proof of the following statement is left to the reader.

\begin{proposition}
\begin{enumerate}[(i)]
\item For any tube~$\tube$ of~$\completeG_{n+1}$, the lattice path~$\phi(\tube)$ is decreasing from~$(0,|\tube|)$ to~$(n+1,0)$ with vertical steps of height~$0$ or~$1$. 
\item $\phi$ is surjective on the decreasing paths ending at~$(n+1,0)$ with vertical steps of height~$0$ or~$1$.
\item For any tubes~$\tube, \tube'$ of~$\completeG_{n+1}$, we have~$\tube \subseteq \tube'$ if and only if~$\phi(\tube')$ decreases when~$\phi(\tube)$ decreases. In particular, the paths~$\phi(\tube)$ and~$\phi(\tube')$ are then non-crossing.
\item For a tubing~$\tubing$ on~$\completeG_{n+1}$, the map~$\sigma(\tubing) : i \longmapsto |\set{\tube \in \tubing}{\phi(\tube) \text{ has a descent at abscissa } i}|+1$ is a surjection from~$[n+1]$ to~$[|\tubing|+1]$, and therefore~$\pi(\tubing) \eqdef \bigsqcup_{j \in [|\tubing|+1]} \sigma^{-1}(j)$ is an ordered partition of~$[n+1]$ into~$|\tubing|+1$ parts. The map~$\tubing \mapsto \pi(\tubing)$ defines an isomorphism form the nested complex~$\nestedComplex(\completeG_{n+1})$ to the refinement poset of ordered partitions.
\item For a tube~$\tube$ of~$\completeG_{n+1}$ not in~$\tubing^\circ$, the path~$\psi(\tube)$ is obtained from the path~$\phi(\tube)$ by replacing the initial down stairs by an horizontal path at height~$0$. See \fref{fig:latticePaths}, where~$\psi(\tube)$ is dotted until it meets~$\phi(\tube)$.
\end{enumerate}
\end{proposition}

\begin{figure}[h]
  \capstart
  \centerline{\includegraphics[scale=1]{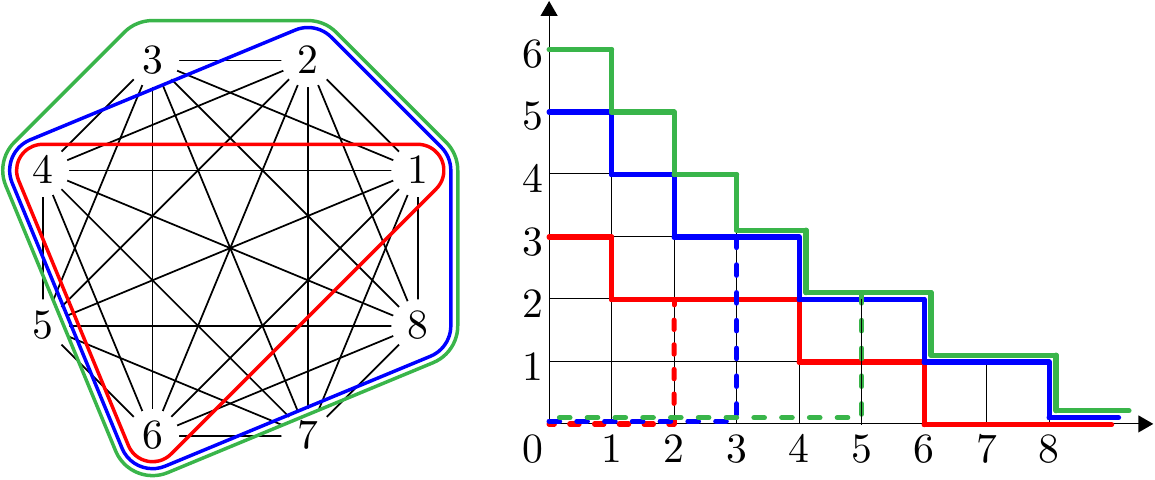}}
  \caption{The tubing~$\{146, 12468, 123468\}$ corresponds to three non-crossing decreasing lattice paths, and to the ordered partition~$57|3|28|146$.}
  \label{fig:latticePaths}
\end{figure}

\begin{remark}[Dual compatibility fan]
\label{rem:dualCompatibilityFanComplete}
As discussed in Section~\ref{subsec:many} below, the complementation~${\tube \mapsto \ground \ssm \tube}$ defines an automorphism of the nested complex~$\nestedComplex(\completeG_{n+1})$, which dualizes the compatibility degree: $\compatibilityDegree{\tube}{\tube'} = \compatibilityDegree{\ground \ssm \tube'}{\ground \ssm \tube}$ for any tubes~$\tube, \tube'$ of~$\completeG_{n+1}$. Therefore, the dual compatibility fans are compatibility fans: for any tubing~$\tubing^\circ$~on~$\completeG_{n+1}$,
\[
\dualCompatibilityFan{\completeG_{n+1}}{\tubing^\circ} = \compatibilityFan{\completeG_{n+1}}{\set{\ground \ssm \tube^\circ}{\tube^\circ \in \tubing^\circ}}.
\]
\end{remark}

\begin{remark}[Linear dependences]
For the complete graph, the linear dependences between compatibility vectors of tubes involved in a flip can already be complicated. However, the coefficients~$(\alpha,\alpha')$ of the flipped tubes in these dependences can only take the following values:
\[
(k,k) \; \text{ with } k > 0, \qquad \text{ or } \qquad (k, kp) \; \text{ with } k,p > 0, \qquad \text{ or } \qquad (kp+p, kp) \; \text{ with } k,p > 0.
\]
\end{remark}

\subsection{Stars}

We finally consider the nested complex~$\nestedComplex(\starG_{n+1})$ and the compatibility fan~$\compatibilityFan{\starG_{n+1}}{\tubing^\circ}$ for the star~$\starG_{n+1}$ with~$n+1$ vertices, \ie the tree with~$n$ leaves~$\ell_1, \dots, \ell_n$ all connected to a central vertex denoted~$*$. The graph associahedron~$\Asso(\starG_{n+1})$ is called \defn{stellohedron}. We have represented in \fref{fig:stellohedron3} two realizations of the $3$-dimensional stellohedron.

\begin{figure}[b]
  \capstart
  \centerline{\includegraphics[scale=1.13]{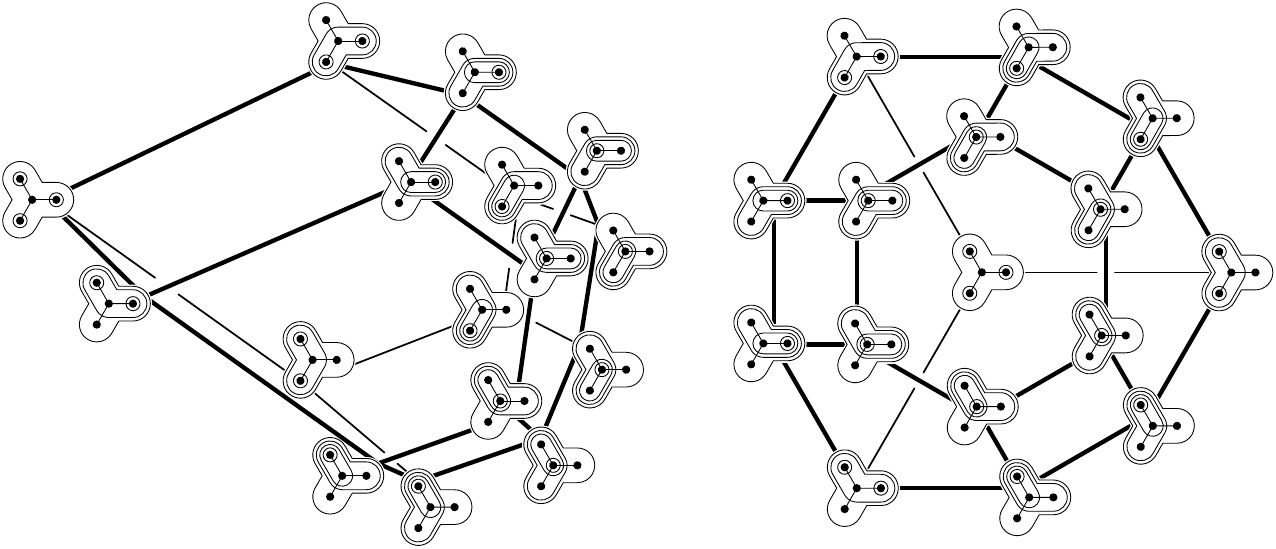}}
  \caption{Two polytopal realizations of the $3$-dimensional stellohedron: their normal fans are the nested fan (left) and a compatibility fan (right).}
  \label{fig:stellohedron3}
\end{figure}

\enlargethispage{.8cm}
One easily checks that the star~$\starG_{n+1}$ has:
\begin{itemize}
\item $2^n + n - 1$ proper tubes~\href{https://oeis.org/A052944}{\cite[A052944]{OEIS}} (distinguish tubes containing~$*$~or~not),
\item $\displaystyle{n! \sum_{i=0}^{n} \frac{1}{i!}}$ maximal tubings~\href{https://oeis.org/A000522}{\cite[A000522]{OEIS}} (consider the minimal tube containing~$*$),
\item $\displaystyle{\sum_{i \in [k]} \binom{n}{k-i} \, (i-1)! \big( i \, S(n-k+i,i) + S(n-k+i,i-1) \big)}$ tubings with $k$ tubes, where $S(m,p)$ denotes the Stirling number of second kind (\ie the number of ways to partition a set of $m$ elements into $p$ non-empty subsets) \href{https://oeis.org/A008277}{\cite[A008277]{OEIS}} (to see it, sum over the number~$i$ of tubes containing~$*$), and
\item $\displaystyle{4 \, n! \sum_{\sum n_i = n} \frac{1}{\prod n_i} - 1 = \sum_{i \ge 1} (i+1)^n/2^i}$ tubings in total (including the empty tubing). This is the number of chains in the boolean lattice on an $n$-element set~\href{https://oeis.org/A007047}{\cite[A007047]{OEIS}} (an immediate bijection is given by the spines of the tubings).
\end{itemize}

We consider the initial maximal tubing~$\tubing^\circ \eqdef \big\{ \{\ell_1\}, \dots, \{\ell_n\} \big\}$ whose tubes are the~$n$ leaves of~$\starG_{n+1}$. The other~$2^n-1$ tubes of~$\starG_{n+1}$ are the tubes containing the central vertex~$*$ and some leaves (but not all). The compatibility degree of such a tube~$\tube$ containing~$*$ with a tube~$\{\ell_i\}$ is~$0$ if~$\ell_i \in \tube$ and~$1$ if~$\ell_i \notin \tube$. The compatibility vector~$\compatibilityVector{\tubing^\circ}{\tube}$ of~$\tube$ with respect to~$\tubing^\circ$ is thus given by the characteristic vector of the leaves of~$\starG_{n+1}$ not contained in~$\tube$. Moreover, two tubes~$\tube, \tube' \notin \tubing^\circ$ are compatible if and only if they are nested (since they both contain the vertex~$*$). Therefore, the compatibility fan~$\compatibilityFan{\starG_{n+1}}{\tubing^\circ}$ is obtained from the coordinate hyperplan fan by a barycentric subdivision of the positive orthant. Examples in dimension~$2$, $3$ and~$4$ are gathered in \fref{fig:star}.

\begin{figure}
  \capstart
  \centerline{\includegraphics[scale=.9]{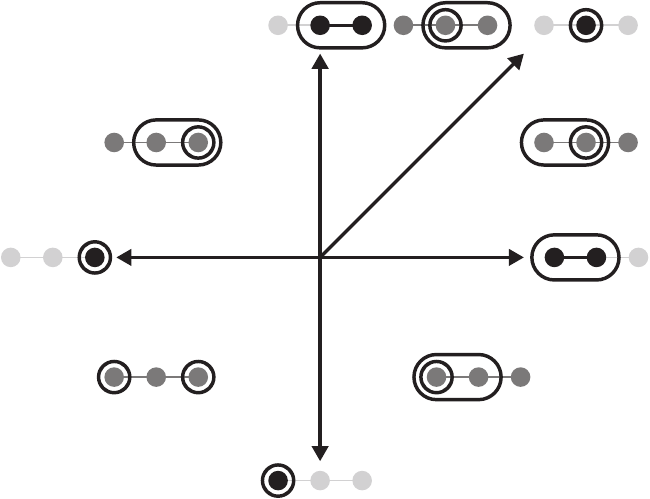}\includegraphics[width=.37\textwidth]{3dimensionalCompatibilityFansLabeled/tripod1}\includegraphics[scale=.35]{4dimensionalCompatibilityFansLabeled/star}}
  \caption{The compatibility fans~$\compatibilityFan{\starG_{n+1}}{\tubing^\circ}$ for the star~$\starG_{n+1}$ and the initial tubing~$\tubing^\circ = \big\{ \{\ell_1\}, \dots, \{\ell_n\} \big\}$ formed by its leaves (for~$n \in \{2,3,4\}$).}
  \label{fig:star}
\end{figure}

\begin{remark}[Dual compatibility fan]
\label{rem:primalDualStar}
Observe that the compatibility degree of any tube of~$\starG_{n+1}$ with any leaf of~$\starG_{n+1}$ belongs to~$\{-1,0,1\}$. Therefore, $\compatibilityVector{\tubing^\circ}{\tube} = \dualCompatibilityVector{\tube}{\tubing^\circ}$ for any tube~$\tube$, so that the compatibility and dual compatibility fans with respect to the initial tubing~$\tubing^\circ$ coincide. This does not hold for arbitrary initial tubings on~$\starG_{n+1}$, see Remark~\ref{rem:naive} and \fref{fig:ctrexmNaive}\,(right).
\end{remark}

\begin{remark}[Linear dependences]
In the special case discussed in this section, all linear dependences between compatibility vectors of tubes involved in a flip are inclusion-exclusion dependences as in the beginning of the proof of Theorem~\ref{theo:compatibilityFan} in Section~\ref{subsec:proofCompatibilityFan}. The coefficients of the flipped tubes thus always equal~$1$ while those of the forced tubes (not in the initial tubing~$\tubing^\circ$) always equal~$-1$.
\end{remark}

\begin{remark}
As pointed out by F.~Santos, the stellohedron~$\Asso(\starG_{n+1})$ coincides with the secondary polytope of two concentric copies of an $(n-1)$-dimensional simplex. See \fref{fig:secondaryPolytopeStellohedron}.
\end{remark}

\begin{figure}[b]
  \capstart
  \centerline{\includegraphics[scale=1.13]{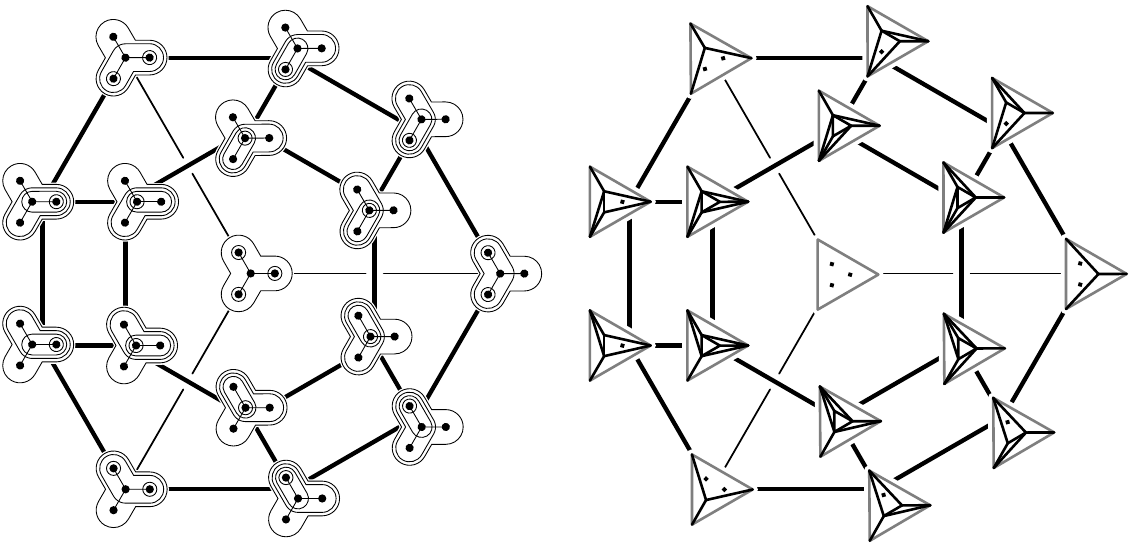}}
  \caption{The stellohedron (left) is a secondary polytope (right).}
  \label{fig:secondaryPolytopeStellohedron}
  \vspace*{-.8cm}
\end{figure}


\section{Further topics}
\label{sec:furtherTopics}

\enlargethispage{.3cm}
In this section, we discuss several further topics in connection to compatibility fans. Note that we only state the results for compatibility fans, but similar statements hold for dual compatibility fans. Section~\ref{subsec:product} studies the behavior of the compatibility fans with respect to products and links. In Section~\ref{subsec:many}, we show that most compatibility fans are not linearly isomorphic, which requires a description of all nested complex isomorphisms. Section~\ref{subsec:polytopality} discusses the question of the realization of our compatibility fans as normal fans of convex polytopes. In Section~\ref{subsec:designNestedComplex}, we extend our construction to design nested complexes~\cite{DevadossHeathVipismakul}. Finally, we discuss in Section~\ref{subsec:LPA} the connection of this paper to Laurent Phenomenon algebras~\cite{LamPylyavskyy-LaurentPhenomenonAlgebras, LamPylyavskyy-LinearLaurentPhenomenonAlgebras}.

\subsection{Products and restrictions}
\label{subsec:product}

In all examples that we discussed earlier, we only considered connected graphs. Compatibility fans for disconnected graphs can be reconstructed from those for connected graphs by the following statement, whose proof is left to the reader.

\begin{proposition}
\label{prop:product}
If~$\graphG$ has connected components~$\graphG_1, \dots, \graphG_k$, then the nested complex~$\nestedComplex(\graphG)$ is the join of the nested complexes~$\nestedComplex(\graphG_1), \dots, \nestedComplex(\graphG_k)$. Moreover, for any maximal tubings~$\tubing_1^\circ, \dots, \tubing_k^\circ$ on~$\graphG_1, \dots, \graphG_k$ respectively, the compatibility fan~$\compatibilityFan{\graphG}{\tubing^\circ}$ with respect to the maximal tubing ${\tubing^\circ \eqdef \tubing_1^\circ \cup \dots \cup \tubing_k^\circ}$ on~$\graphG$ is the product of the compatibility fans~$\compatibilityFan{\graphG_1}{\tubing_1^\circ}, \dots, \compatibilityFan{\graphG_k}{\tubing_k^\circ}$:
\[
\compatibilityFan{\graphG}{\tubing^\circ} = \compatibilityFan{\graphG_1}{\tubing_1^\circ} \times \dots \times \compatibilityFan{\graphG_k}{\tubing_k^\circ} = \set{C_1 \times \dots \times C_k}{C_i \in \compatibilityFan{\graphG_i}{\tubing_i^\circ} \text{ for all } i \in [k]}.
\]
\end{proposition}

\fref{fig:productFans}\,(right) illustrates Proposition~\ref{prop:product} with the compatibility fan of a graph formed by two paths. Compatibility fans of paths are discussed in Section~\ref{subsec:paths}. Besides all compatibility vectors, the cones of three different tubings are represented in \fref{fig:productFans}\,(right).

\begin{figure}[h]
  \capstart
  \centerline{\includegraphics[scale=1.05]{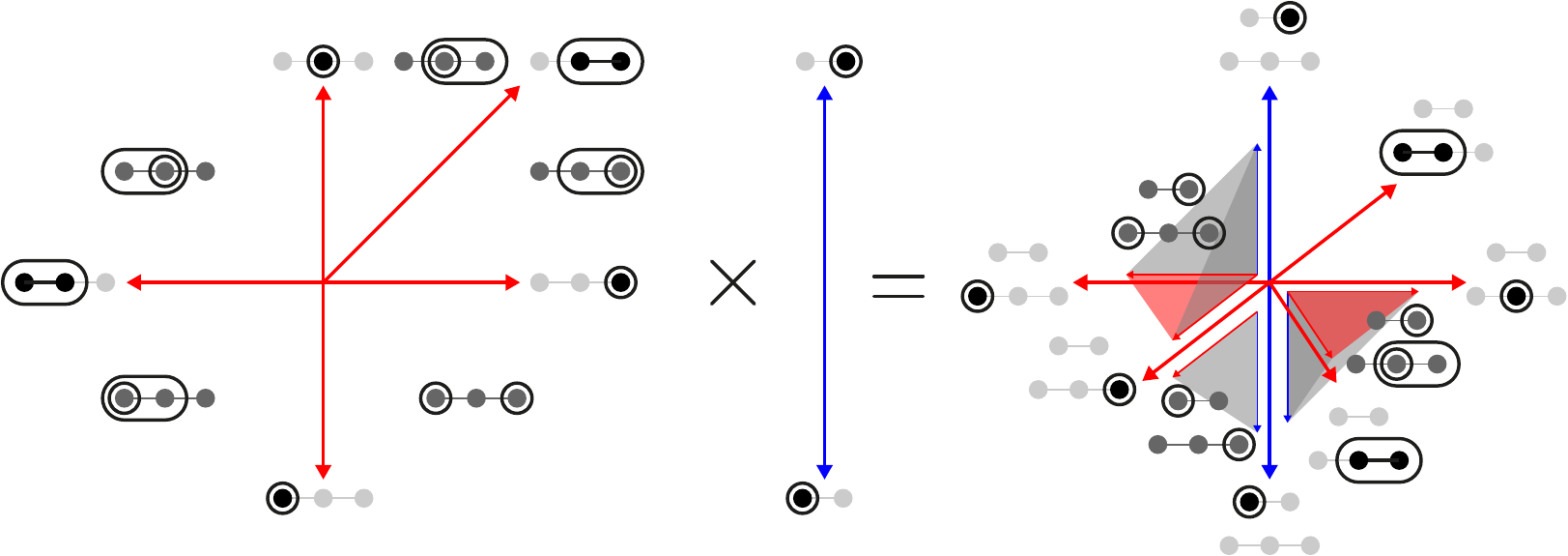}}
  \caption{The compatibility fan of a disconnected graph is the product (right) of the compatibility fans of its connected components (left and~middle).}
  \label{fig:productFans}
\end{figure}

As observed in~\cite{CarrDevadoss}, all links of graphical nested complexes are joins of graphical nested complexes. The following statement asserts that the compatibility fans reflect this property on coordinate hyperplanes. To be more precise, for a tube~$\tube^\circ$ of~$\graphG$, we denote by~$\graphG{}[\tube^\circ]$ the restriction of~$\graphG$ to~$\tube^\circ$ and by~$\graphG{}^\star\tube^\circ$ the~\defn{reconnected complement} of~$\tube^\circ$ in~$\graphG$, \ie the graph with vertex set~$\ground \ssm \tube^\circ$ and edge set~$\bigset{e \in \binom{\ground \ssm \tube^\circ}{2}}{\text{$e$ or~$e \cup \tube^\circ$ is connected in~$\graphG$}}$. A maximal tubing~$\tubing^\circ$ on~$\graphG$ containing~$\tube^\circ$ induces maximal tubings~$\tubing^\circ[\tube^\circ] \eqdef \set{\tube}{\tube \in \tubing^\circ, \tube \subsetneq \tube^\circ}$ on the restriction~$\graphG{}[\tube^\circ]$ and~${\tubing^\circ}^\star\tube^\circ \eqdef \set{\tube \ssm \tube^\circ}{\tube \in \tubing^\circ, \tube \not\subset \tube^\circ}$ on the reconnected complement~$\graphG^\star\tube^\circ$. See \fref{fig:exmReconnectedComplement}.

\begin{figure}
  \capstart
  \centerline{\includegraphics[scale=.9]{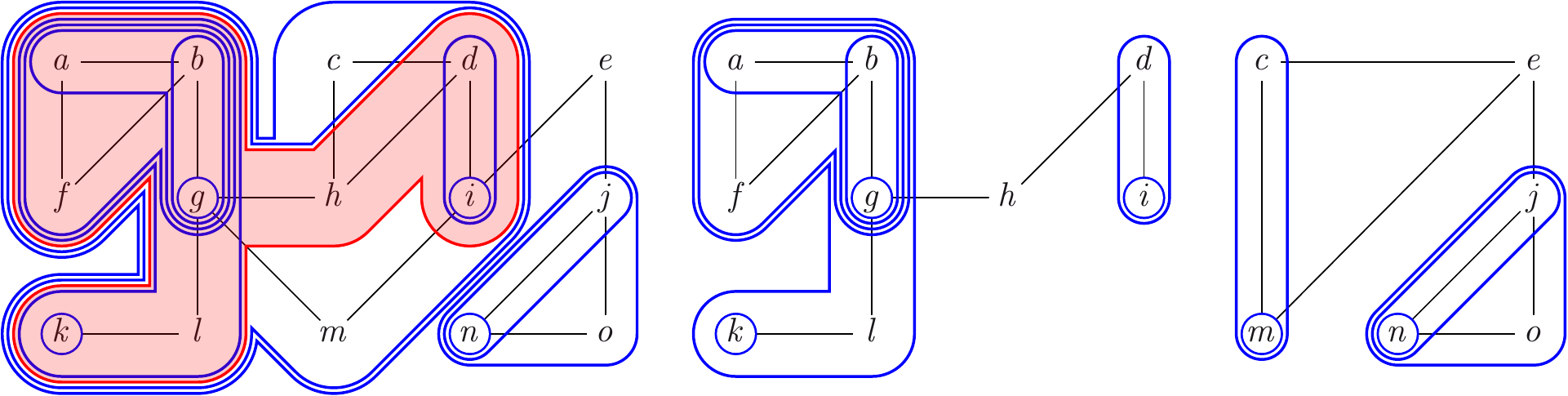}}
  \caption{The red tube~$\tube^\circ\ex = \{a,b,d,f,g,h,i,k,l\}$ in the maximal tubing~$\tubing^\circ$~(left) yields a maximal tubing~$\tubing^\circ\ex[\tube^\circ\ex]$ on the restriction~$\graphG\ex[\tube^\circ\ex]$ (middle) and a maximal tubing~${\tubing^\circ\ex}^\star\tube^\circ\ex$ on the reconnected complement~${\graphG\ex}^\star\tube^\circ\ex$ (right).}
  \label{fig:exmReconnectedComplement}
\end{figure}

\begin{proposition}
\label{prop:restriction}
The link of a tube~$\tube^\circ$ in the nested complex~$\nestedComplex(\graphG)$ is isomorphic to the join of the nested complexes~$\nestedComplex(\graphG{}[\tube^\circ])$ and~$\nestedComplex(\graphG^\star\tube^\circ)$. Moreover, for an initial maximal tubing~$\tubing^\circ$ containing~$\tube^\circ$, the intersection of the compatibility fan~$\compatibilityFan{\graphG}{\tubing^\circ}$ with the coordinate hyperplane orthogonal to~$\b{e}_{\tube^\circ}$ is the product of the compatibility fans~$\compatibilityFan{\graphG{}[\tube^\circ]}{\tubing^\circ[\tube^\circ]}$ and~$\compatibilityFan{\graphG^\star\tube^\circ}{{\tubing^\circ}^\star\tube^\circ}$.
\end{proposition}

This statement follows from Lemmas~\ref{lem:restriction} and~\ref{lem:restrictionCompatibility} and Theorem~\ref{theo:compatibilityFan}, proved in Section~\ref{sec:proofs}.

\subsection{Many compatibility fans}
\label{subsec:many}

In this section, we show that we obtained many distinct compatibility fans. Following~\cite{CeballosSantosZiegler}, we classify compatibility fans up to linear isomorphisms: two fans~$\cF, \cF'$ of~$\R^n$ are linearly isomorphic if there exists an invertible linear map which sends the cones of~$\cF$ to the cones of~$\cF'$. Observe already that if two compatibility fans~$\compatibilityFan{\graphG}{\tubing^\circ}$ and~$\compatibilityFan{\graphG'}{\tubing'^\circ}$ are linearly isomorphic, then the two nested complexes~$\nestedComplex(\graphG)$ and~$\nestedComplex(\graphG')$ are (combinatorially) isomorphic, meaning that there is a bijection~$\Phi$ from the tubes of~$\graphG$ to the tubes of~$\graphG'$ which preserves the compatibility. The converse does not always hold: a nested complex isomorphism can preserve compatibility without preserving the compatibility degree. However, we prove below that the nested complex isomorphisms are so constrained that they all either preserve the compatibility degree and thus induce linear isomorphisms between compatibility fans, or exchange the compatibility and dual compatibility degrees and thus induce linear isomorphisms between compatibility and dual compatibility fans.

In the sequel, we describe all nested complex isomorphisms. Observe first that an isomorphism~$\phi$ between two graphs~$\graphG$ and~$\graphG'$ automatically induces an isomorphism~$\Phi$ between the nested complexes~$\nestedComplex(\graphG)$ and~$\nestedComplex(\graphG')$ defined by~$\Phi(\tube) \eqdef \set{\phi(v)}{v \in \tube}$ for all tubes~$\tube$ on~$\graphG$. We say that such a nested complex isomorphism~$\Phi$ is \defn{trivial}. Trivial isomorphisms clearly preserve compatibility degrees: $\compatibilityDegree{\Phi(\tube)}{\Phi(\tube')} = \compatibilityDegree{\tube}{\tube'}$ for any tubes~$\tube, \tube'$ on~$\graphG$. We are interested in non-trivial nested complex isomorphisms. 
We first want to underline two relevant examples.

\begin{example}
\label{exm:nestedComplexIsomorphisms}
The reader can check that:
\begin{enumerate}[(i)]
\item The complementation~$\tube \mapsto \ground \ssm \tube$ is a non-trivial automorphism of the nested complex~$\nestedComplex(\completeG_{n+1})$ of the complete graph~$\completeG_{n+1}$. It dualizes the compatibility degree: $\compatibilityDegree{\ground \ssm \tube}{\ground \ssm \tube'} = \compatibilityDegree{\tube'}{\tube}$ for any tubes~$\tube, \tube'$ of~$\completeG_{n+1}$.

\item The map~$\rot$ defined for~$1 \le j \le k \le n+1$ by
\[
\rot[j,k] \eqdef
\begin{cases}
[k+1, n+1] & \text{if } j = 1, \\
[j-1, k-1] & \text{if } j > 1,
\end{cases}
\]
is a non-trivial automorphism of the nested complex~$\nestedComplex(\pathG_{n+1})$ of the path~$\pathG_{n+1}$. Indeed, up to conjugation by the bijection~$\delta \mapsto \tube_\delta$ of Section~\ref{subsec:paths}, the map~$\rot$ coincides with the (combinatorial) $1$-vertex clockwise rotation of the $(n+3)$-gon~$Q_{n+3}$. See \fref{fig:rotation}. 

\begin{figure}
  \capstart
  \centerline{\includegraphics[scale=.6]{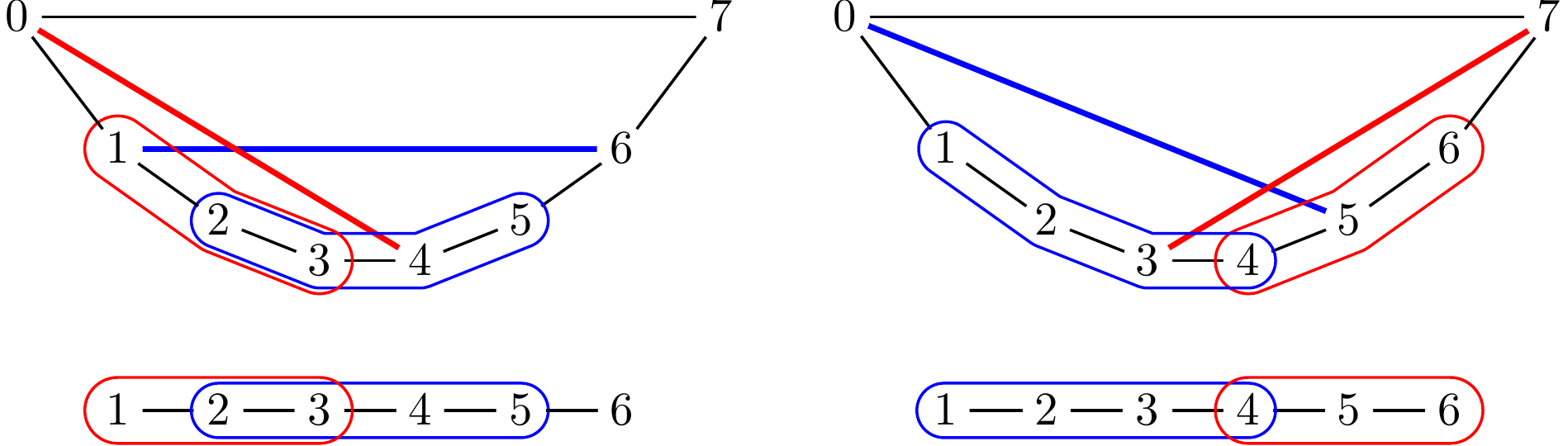}}
  \vspace{.5cm}
  \centerline{
  $\begin{array}{c|c@{\;\,}c@{\;\,}c@{\;\,}c@{\;\,}c@{\;\,}c@{\;\,}c@{\;\,}c@{\;\,}c@{\;\,}c@{\;\,}c@{\;\,}c@{\;\,}c@{\;\,}c@{\;\,}c@{\;\,}c@{\;\,}c@{\;\,}c@{\;\,}c@{\;\,}c}
  \tube & [1] & [2] & [3] & [4] & [5] & [6] & [1,2] & [2,3] & [3,4] & [4,5] & [5,6] & [1,3] & [2,4] & [3,5] & [4,6] & [1,4] & [2,5] & [3,6] & [1,5] & [2,6] \\
  \rot\tube & [2,6] & [1] & [2] & [3] & [4] & [5] & [3,6] & [1,2] & [2,3] & [3,4] & [4,5] & [4,6] & [1,3] & [2,4] & [3,5] & [5,6] & [1,4] & [2,5] & [6] & [1,5]
  \end{array}$
  }
  \vspace{-.1cm}
  \caption{The $1$-vertex clockwise rotation of the~$(n+3)$-gon induces a non-trivial automorphism~$\rot$ of the nested complex~$\nestedComplex(\pathG_{n+1})$.}
  \label{fig:rotation}
\end{figure}

Therefore, $\rot$ has order~$n+3$, and its iterated powers are explicitly described by
\[
\rot^p[j,k] \eqdef
\begin{cases}
[j-p, k-p] & \text{if } 0 \le p < j, \\
[k-p+2, n+j-p+1] & \text{if } j \le p < k+2, \\
[n+p+j-2k-1, n+p-k-1] & \text{if } k+2 \le p < n+3.
\end{cases}
\]
In fact, it is shown in~\cite[Lemma~2.2]{CeballosSantosZiegler} that the automorphism group of the nested complex~$\nestedComplex(\pathG_{n+1})$ is the dihedral group generated by the non-trivial automorphism~$\rot$ (the rotation of the \mbox{$(n+3)$-gon~$Q_{n+3}$}) and the automorphism~$\rev : [j,k] \mapsto [n+2-k, n+2-j]$ induced by the graph automorphism~${j \mapsto n+2-j}$ of~$\pathG_{n+1}$ (the vertical reflection of the $(n+3)$-gon~$Q_{n+3}$). Note that since the compatibility degree on~$\pathG_{n+1}$ is in~$\{-1,0,1\}$, any nested complex automorphism preserves the compatibility degree (and also dualizes it since it is symmetric).
\end{enumerate}
\end{example}

We now describe a generalization of both cases of Example~\ref{exm:nestedComplexIsomorphisms}. For ${\ninf \eqdef \{n_1, \dots, n_\ell\}} \in \N^\ell$ with~$n + 1 = \sum_{i \in [\ell]} (n_i + 1)$, define the \defn{spider}~$\spiderG_{\ninf}$ as the graph with vertices~$\set{v^i_j}{i \in [\ell], 0 \le j \le n_i}$ and edges~${\bigset{\big\{v^i_{j-1}, v^i_j\big\}}{i \in [\ell], j \in [n_i]} \cup \bigset{\big\{v^i_0, v^{i'}_0\big\}}{i \ne i' \in [\ell]}}$. For~$0 \le j \le k \le n_i$, we denote by~$\big[ v^i_j, v^i_k \big]$ the path between~$v^i_j$ and~$v^i_k$ in~$\spiderG_{\ninf}$. Informally, the spider~$\spiderG_{\ninf}$ consists in~$\ell$ paths~$[v^i_1, v^i_{n_i}]$ called \defn{legs} of the spider, each attached to a vertex~$v^i_0$ of a clique called \defn{body} of the spider. See \fref{fig:isomorphismSpiders}. Note that spiders are sometimes called sunlike graphs in the literature.

We now define a non-trivial automorphism~$\Omega$ of the nested complex~$\nestedComplex(\spiderG_{\ninf})$ of the spider~$\spiderG_{\ninf}$. We distinguish two kinds of tubes of~$\nestedComplex(\spiderG_{\ninf})$:
\begin{description}
\item[Leg tubes] A tube~$\tube$ disjoint from the body is included in a leg. The map~$\Omega$ sends~$\tube$ into its image by the transformation which cuts the leg containing~$\tube$ and glues it back to the body by its other endpoint. Formally, for~$i \in [\ell]$ and~$1 \le j \le k \le n_i$,
\[
\Omega\big( \big[ v^i_j, v^i_k \big] \big) \eqdef \big[ v^i_{n_i+1-k}, v^i_{n_i+1-j} \big].
\]
Note that~$\Omega$ sends a leg tube~$\tube$ to a leg tube~$\Omega(\tube)$ with~$|\Omega(\tube)| = |\tube|$.
\item[Body tubes] A tube~$\tube$ intersecting the body is the union of initial segments~$\big[ v^i_0, v^i_{k_i} \big]$, with~$-1 \le k_i \le n_i$ (with the convention that~$[v^i_0, v^i_{-1}] = \varnothing$). We then define
\[
\Omega\big( \bigcup_{i \in [\ell]} \big[ v^i_0, v^i_{k_i} \big] \big) \eqdef \bigcup_{i \in [\ell]} \big[ v^i_0, v^i_{n_i-1-k_i} \big].
\]
Note that~$\Omega$ sends a body tube~$\tube$ to a body tube~$\Omega(\tube)$ with~$|\Omega(\tube)| = |\ground| - |\tube|$.
\end{description}
\fref{fig:isomorphismSpiders} illustrates the map~$\Omega$ on different tubes of the spider~$\spiderG_{\{0,3,2,3,0,3,2,3\}}$.
Observe that~$\Omega$ indeed generalizes both non-trivial nested complex automorphisms of Example~\ref{exm:nestedComplexIsomorphisms}:
\begin{enumerate}[(i)]
\item The complete graph~$\completeG_{n+1}$ is the spider~$\spiderG_{\{0\}^{n+1}}$ whose legs are all empty. The automorphism~$\Omega$ of~$\nestedComplex(\spiderG_{\{0\}^{n+1}})$ specializes to the complementation~$\tube \mapsto \ground \ssm \tube$ on~$\nestedComplex(\completeG_{n+1})$. 
\item The path~$\pathG_{n+1}$ is a degenerate spider whose body can be chosen at different places. Indeed, the path~$\pathG_{n+1}$ coincides with the spider~$\spiderG^1 \eqdef \spiderG_{\{n\}}$ with body~$\{1\}$ and the single leg~$[2,n+1]$, and the automorphism~$\Omega$ of~$\nestedComplex(\spiderG^1)$ is the composition of the rotation automorphism~$\rot$ with the vertical reflection automorphism~$\rev$. Similarly, for any~$2 \le p \le n+1$, the path~$\pathG_{n+1}$ coincides with the spider~$\spiderG^p \eqdef \spiderG_{\{p-2, n+1-p\}}$ with body~$\{p-1,p\}$ and legs~${[1,p-2]}$ and~${[p+1,n+1]}$, and the automorphism~$\Omega$ of~$\nestedComplex(\spiderG^p)$ is the composition of~$\rot^p$ with~$\rev$. Finally, the path~$\pathG_{n+1}$ coincides with the spider~$\spiderG^{n+2} \eqdef \spiderG_{\{n\}}$ with body~$\{n+1\}$ and the single leg~$[n]$, and the automorphism~$\Omega$ of~$\nestedComplex(\spiderG^{n+2})$ is the composition of~$\rot^{n+2}$ with~$\rev$.
\end{enumerate}
This actually suggests an alternative description of~$\Omega$ on arbitrary spiders~$\spiderG_{\ninf}$. Namely, $\Omega$ is equivalently described by the following steps: shift all leg tubes towards the body, complement all body tubes, delete all edges~$\bigset{\big\{v^i_0, v^{i'}_0\big\}}{i \ne i' \in [\ell]}$ of the body, replace them by the clique~${\bigset{\big\{v^i_{n_i}, v^{i'}_{n_{i'}}\big\}}{i \ne i' \in [\ell]}}$ on the feets of the spider, and finally apply the trivial isomorphism from the resulting spider back to the initial spider. Our original presentation of~$\Omega$ will nevertheless be easier to handle in the proofs.
The following statement is proved in Section~\ref{subsec:proofIsomorphisms}.

\begin{proposition}
\label{prop:exmAutomorphism}
The map~$\Omega$ is a non-trivial involutive automorphism of the nested complex~$\nestedComplex(\spiderG_{\ninf})$ of the spider~$\spiderG_{\ninf}$ which dualizes the compatibility degree: $\compatibilityDegree{\Omega(\tube)}{\Omega(\tube')} = \compatibilityDegree{\tube'}{\tube}$.
\end{proposition}

\begin{figure}
  \capstart
  \centerline{\includegraphics[scale=.9]{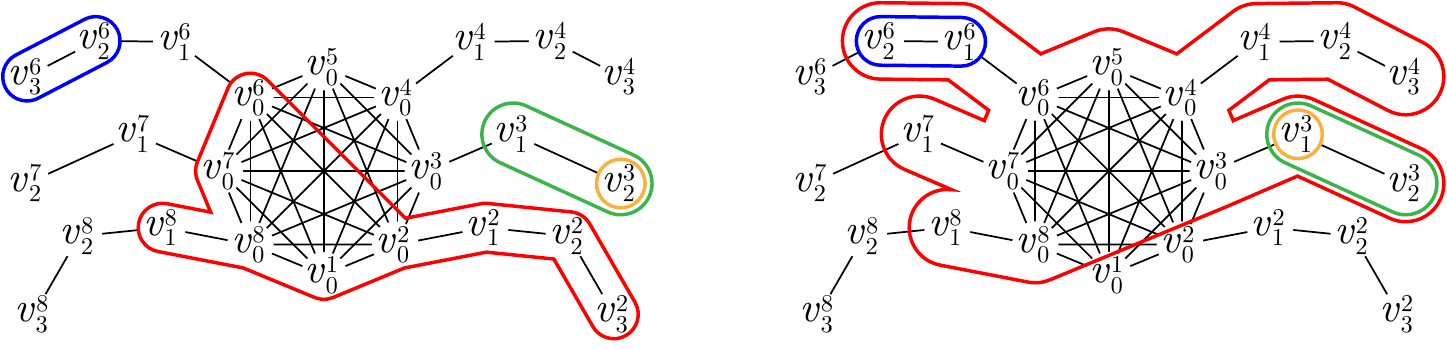}}
  \caption{The spider~$\spiderG_{\{0,3,2,3,0,3,2,3\}}$ and examples of the action of the non-trivial nested complex isomorphism~$\Omega$: the tubing~$\tubing$ on the left is sent to the tubing~$\Omega(\tubing)$ on the right.}
  \label{fig:isomorphismSpiders}
\end{figure}

\begin{remark}
It follows from Proposition~\ref{prop:exmAutomorphism} that all dual compatibility fans of a spider~$\spiderG_{\ninf}$ are also compatibility fans of~$\spiderG_{\ninf}$: we have~$\dualCompatibilityFan{\spiderG_{\ninf}}{\tubing^\circ} = \compatibilityFan{\spiderG_{\ninf}}{\Omega(\tubing^\circ)}$ for any tubing~$\tubing^\circ$ on~$\spiderG_{\ninf}$. Note that we already used this observation for complete graphs in Remark~\ref{rem:dualCompatibilityFanComplete}.
\end{remark}

In fact, these non-trivial automorphisms of the nested complexes of the spiders are essentially the only non-trivial nested complex isomorphisms. The following statements are proved in Section~\ref{subsec:proofIsomorphisms}.

\begin{proposition}
\label{prop:nestedComplexIsomorphismDisconnected}
A nested complex isomorphism~$\Phi : \nestedComplex(\graphG) \to \nestedComplex(\graphG')$ restricts to nested complex isomorphisms~$\nestedComplex(\graphG[H]) \to \nestedComplex(\graphG[H]')$ between maximal connected subgraphs~$\graphG[H]$ of~$\graphG$ and~$\graphG[H]'$ of~$\graphG'$.
\end{proposition}

\begin{theorem}
\label{theo:nestedComplexIsomorphisms}
Let~$\graphG$ and~$\graphG'$ be two connected graphs and~$\Phi : \nestedComplex(\graphG) \to \nestedComplex(\graphG')$ be a non-trivial nested complex isomorphism. Then~$\graphG$ and~$\graphG'$ are spiders and there exists a graph isomorphism~$\psi : \graphG \to \graphG'$ which induces a nested complex isomorphism~$\Psi : \nestedComplex(\graphG) \to \nestedComplex(\graphG')$ (defined by~$\Psi(\tube) \eqdef \set{\psi(v)}{v \in \tube}$) such that the composition~$\Psi^{-1} \circ \Phi$ coincides with the non-trivial nested complex automorphism~$\Omega$ on~$\nestedComplex(\graphG)$.
\end{theorem}

\begin{corollary}
For connected graphs~$\graphG$, $\graphG'$, any nested complex isomorphism~$\Phi : \nestedComplex(\graphG) \to \nestedComplex(\graphG')$ either preserves or dualizes the compatibility degree: either~$\compatibilityDegree{\Phi(\tube)}{\Phi(\tube')} = \compatibilityDegree{\tube}{\tube'}$ for all~$\tube, \tube'$ of~$\graphG$, or~$\compatibilityDegree{\Phi(\tube)}{\Phi(\tube')} = \compatibilityDegree{\tube'}{\tube}$ for all~$\tube, \tube'$ of~$\graphG$.
\end{corollary}

To finish our classification of primal and dual compatibility fans up to linear isomorphisms, it remains to understand when the primal and the dual compatibility fans of~$\graphG$ with respect to the same initial maximal tubing are linearly isomorphic. For example, we already observed that
\begin{enumerate}[(i)]
\item $\compatibilityFan{\pathG_{n+1}}{\tubing^\circ} = \dualCompatibilityFan{\pathG_{n+1}}{\tubing^\circ}$ for any initial tubing~$\tubing^\circ$ on a path~$\pathG_{n+1}$, see Remark~\ref{rem:primalDualPath},
\item $\compatibilityFan{\starG_{n+1}}{\tubing^\circ} = \dualCompatibilityFan{\starG_{n+1}}{\tubing^\circ}$ for the initial tubing~$\tubing^\circ \eqdef \big\{ \{\ell_1\}, \dots, \{\ell_n\} \big\}$ of the star~$\starG_{n+1}$ whose tubes are the~$n$ leaves, see Remark~\ref{rem:primalDualStar}.
\end{enumerate}
These examples extend to all subdivisions of stars. Namely, for~$\ninf \eqdef \{n_1, \dots, n_\ell\} \in \N^\ell$ with ${n = \sum_{i \in [\ell]} (n_i + 1)}$, define the \defn{octopus}~$\octopusG_{\ninf}$ as the graph with vertices~$\big\{ * \big\} \cup \set{v^i_j}{i \in [\ell], 0 \le j \le n_i}$ and edges~${\bigset{\big\{v^i_{j-1}, v^i_j\big\}}{i \in [\ell], j \in [n_i]} \cup \bigset{\big\{*, v^i_0 \big\}}{i \in [\ell]}}$. Informally, the octopus~$\octopusG_{\ninf}$ consists in~$\ell$ paths~$[v^i_0, v^i_{n_i}]$ called \defn{legs}, each attached to a~\defn{head}~$*$. These graphs are often called starlike graphs in the literature, we use the term octopus to stay in the wildlife lexical field. Note that the path~$\pathG_{n+1}$ is a degenerate octopus where the head can be chosen at any vertex. As for stars, a tube of an octopus~$\octopusG_{\ninf}$ that does not contain its head~$*$ is either compatible or exchangeable with any other tube of~$\octopusG_{\ninf}$. Therefore, if~$\tubing^\circ$ is a maximal tubing on~$\octopusG_{\ninf}$ whose tubes do not contain the head~$*$, then the compatibility fan~$\compatibilityFan{\octopusG_{\ninf}}{\tubing^\circ}$ and the dual compatibility fan~$\dualCompatibilityFan{\octopusG_{\ninf}}{\tubing^\circ}$ coincide. The following lemma states that this only happens in this situation.

\begin{lemma}
\label{lem:comparisonPrimalDual}
Let~$\graphG$ be a connected graph and~$\tubing^\circ$ be an initial maximal tubing on~$\graphG$. If~$\compatibilityFan{\graphG}{\tubing^\circ}$ and~$\dualCompatibilityFan{\graphG}{\tubing^\circ}$ are linearly isomorphic, then~$\graphG$ is an octopus whose head is contained in no~tube~of~$\tubing^\circ$.
\end{lemma}

We finally obtain our classification of primal and dual compatibility fans. We first focus on primal compatibility fans and then conclude with both primal and dual compatibility fans together.

\begin{corollary}
\label{coro:many}
The number of linear isomorphism classes of compatibility fans of a connected graph~$\graphG$ is:
\begin{enumerate}[(i)]
\item the number of triangulations of the regular $(n+3)$-gon up to the action of the dihedral group if $\graphG = \pathG_{n+1}$ is a path,
\item the number of orbits of maximal tubings on~$\graphG$ under graph automorphisms of~$\graphG$ otherwise.
\end{enumerate}
\end{corollary}

\begin{corollary}
The number of linear isomorphism classes of primal and dual compatibility fans of a connected graph~$\graphG$ is:
\begin{enumerate}[(i)]
\item the number of triangulations of the regular $(n+3)$-gon up to the action of the dihedral group if $\graphG = \pathG_{n+1}$ is a path,
\item the number of $\graphG$-automorphism orbits of maximal tubings on~$\graphG$ if~$\graphG$ is a spider but not~a~path,
\item the number of $\graphG$-automorphism orbits of maximal tubings on~$\graphG$, counted twice if not rooted at~$*$, if~$\graphG$ is an octopus with head~$*$ but not a path,
\item twice the number of $\graphG$-automorphism orbits of maximal tubings on~$\graphG$ otherwise.
\end{enumerate}
\end{corollary}

\begin{remark}
Not only we obtain many non-isomorphic complete simplicial fan realizations for graphical nested complexes (as stated in Corollary~\ref{coro:many}), but these realizations cannot be derived from the existing geometric constructions for graph associahedra. Indeed, all previous polytopal realizations of graph associahedra can be obtained by successive face truncations of a simplex~\cite{CarrDevadoss} or of a cube~\cite{Volodin, DevadossForceyReisdorfShowers}. Not all compatibility fans can be constructed in this way. For example, the leftmost compatibility fan of \fref{fig:allLabels} is not linearly isomorphic to the normal fan of a polytope obtained by face truncations of the $3$-dimensional simplex or cube.
\end{remark}

\subsection{Polytopality}
\label{subsec:polytopality}

In this section, we briefly discuss the polytopality of our compatibility fans for graphical nested complexes. A complete polyhedral fan is said to be \defn{polytopal} (or \defn{regular}) if it is the normal fan of a polytope. It is well known that not all complete polyhedral fans (even simplicial) are polytopal. Examples are easily constructed from non-regular triangulations, see \eg the discussion in~\cite[Chapter~2]{DeLoeraRambauSantos}.

\para{Polytopality of cluster fans}
The polytopality of cluster fans has been studied since the foundations of finite type cluster algebras. For compatibility fans, polytopality was shown for particular initial clusters by F.~Chapoton, S.~Fomin and A.~Zelevinsky~\cite{ChapotonFominZelevinsky} and in type~$A$ by F.~Santos~\cite[Section~5]{CeballosSantosZiegler}.

\begin{theorem}[\cite{ChapotonFominZelevinsky, CeballosSantosZiegler}]
\label{theo:polytopalityCompatibilityFanCluster}
The $\b{d}$-vector fan (or compatibility fan) is polytopal for
\begin{itemize}
\item any initial cluster in any type~$A$ cluster algebra~\cite[Section~5]{CeballosSantosZiegler}, and
\item the bipartite initial cluster in any finite type cluster algebra~\cite{ChapotonFominZelevinsky}.
\end{itemize}
\end{theorem}

The polytopality of the $\b{g}$-vector fan was studied by C.~Hohlweg, C.~Lange and H.~Thomas~\cite{HohlwegLangeThomas}. See also recent alternative proofs by S.~Stella~\cite{Stella} and V.~Pilaud and C.~Stump~\cite{PilaudStump-brickPolytope}. 

\begin{theorem}[\cite{HohlwegLangeThomas, Stella, PilaudStump-brickPolytope}]
The $\b{g}$-vector fan (or Cambrian fan) is polytopal for any acyclic initial cluster in any finite type cluster algebra.
\end{theorem}

To our knowledge, the polytopality of the $\b{d}$- and $\b{g}$-vector fans remains open in all other cases. All these results rely on the following characterization of polytopality for complete simplicial fans, in connection to regular triangulations of vector configurations and to the theory of secondary polytopes~\cite{GelfandKapranovZelevinsky}, see also~\cite{DeLoeraRambauSantos}. Equivalent formulations of this characterization appear \eg in~\cite[Lemma~2.1]{ChapotonFominZelevinsky}, \cite[Proposition~6.3]{Zelevinsky}, \cite[Theorem~4.1]{HohlwegLangeThomas}, or~\cite[Lemma~5.4]{CeballosSantosZiegler}. Here, we follow the presentation of the first two which fits our previous notations.

\begin{proposition}
\label{prop:polytopalityFan}
Let~$\cF$ be a complete simplicial fan in~$\R^n$ and let~$\b{R}$ denote a set of vectors generating its rays ($1$-dimensional cones). Then the following are equivalent:
\begin{enumerate}
\item $\cF$ is the normal fan of a simple polytope in~$(\R^n)^*$;
\item There exists a map~$\omega: \b{R} \to \R_{> 0}$ such that for any two maximal adjacent cones~$\R_{\ge 0} \b{S}$ and~$\R_{\ge 0} \b{S}'$ of~$\cF$ with $\b{S}, \b{S}' \subseteq \b{R}$ and~$\b{S} \ssm \{\b{s}\} = \b{S}' \ssm \{\b{s'}\}$, we have
\[
\alpha \, \omega(\b{s}) + \alpha' \, \omega(\b{s}') + \sum_{\b{r} \in \b{S} \cap \b{S}'} \beta_{\b{r}} \, \omega(\b{r}) > 0,
\]
where
\[
\alpha \, \b{s} + \alpha' \, \b{s}' + \sum_{\b{r} \in \b{S} \cap \b{S}'} \beta_{\b{r}} \, \b{r} = 0
\]
is the unique (up to rescaling) linear dependence with~$\alpha, \alpha' > 0$ between the rays of~$\b{S} \cup \b{S}'$.
\end{enumerate}
Under these conditions, $\cF$ is the normal fan of the polytope defined by
\[
\set{\phi \in (\R^n)^*}{\dotprod{\phi}{\b{r}} \le \omega(\b{r}) \text{ for all } \b{r} \in \b{R}}.
\]
\end{proposition}

\para{Polytopality of compatibility fans}
We have seen in Section~\ref{subsec:graphAssociahedra} that the nested fan is the normal fan of the graph associahedron of~\cite{CarrDevadoss, Devadoss, Postnikov, Zelevinsky}. For the compatibility fan, the question of the polytopality remains open:

\begin{conjecture}
\label{conj:polytopality}
All primal and dual compatibility fans of graphical nested complexes are polytopal.
\end{conjecture}

To settle this conjecture, the hope would be to apply the characterization of polytopality for complete simplicial fans presented in Proposition~\ref{prop:polytopalityFan}. Besides finding an explicit function~$\omega$ on the compatibility vectors of the tubes of a graph, our main issue is that we do not control the details of the linear dependence between the compatibility vectors of the tubes involved in a flip. See the proof of Theorem~\ref{theo:compatibilityFan} in Section~\ref{subsec:proofCompatibilityFan}.

To support Conjecture~\ref{conj:polytopality}, we have studied the polytopality of the compatibility fans of the specific families of graphs discussed in Section~\ref{sec:specificGraphs}. We show in Section~\ref{subsec:proofPolytopality} that Conjecture~\ref{conj:polytopality} holds for paths and cycles.

\begin{theorem}
\label{theo:polytopalityPathsCycles}
All compatibility and dual compatibility fans of paths and cycles are polytopal.
\end{theorem}

Note that the case of paths is covered by the results of~\cite[Section~5]{CeballosSantosZiegler} presented in Theorem~\ref{theo:polytopalityCompatibilityFanCluster}. For cycles, the result was unknown except for the bipartite initial tubing by the results of~\cite{ChapotonFominZelevinsky} on type~$B$ and~$C$ cluster algebras. Via the correspondences given in Propositions~\ref{prop:typeA} and~\ref{prop:typeBC}, Theorem~\ref{theo:polytopalityPathsCycles} translates to the following relevant property of $\b{d}$-vector fans.

\begin{corollary}
In types~$A$, $B$ and~$C$ cluster algebras, the $\b{d}$-vector fan with respect to any initial cluster (acyclic or not) is polytopal.
\end{corollary}

We were not able to settle Conjecture~\ref{conj:polytopality} for arbitrary graphs. We believe that this question is worth investigating. As already mentioned, it requires a better understanding of all linear dependences between the compatibility vectors of the tubes involved in a flip.

\medskip
In another direction, we checked empirically that all $3$-dimensional compatibility and dual compatibility fans of Section~\ref{subsec:fewVertices} and \fref{fig:4vertices} are polytopal. Using the characterization given in Proposition~\ref{prop:polytopalityFan}, it boils down to check the feasibility of (many) linear programs.

\medskip
Finally, as a curiosity and to conclude this polytopality section on a recreative note, we provide a polytopal realization of the compatibility fan for the star~$\starG_{n+1}$ with respect to the initial tubing~$\tubing^\circ \eqdef \big\{ \{\ell_1\}, \dots, \{\ell_n\} \big\}$ whose tubes are the~$n$ leaves~of~$\starG_{n+1}$. We first observe that this fan is linearly isomophic to the fan~$\cG(\graphG)$ of Theorem~\ref{theo:gFan}. Therefore, it can be realized by an affine transformation of the graph associahedra constructed in Theorem~\ref{theo:graphAssociahedron}.

Here, we prefer to give a direct construction with integer coordinates. We provide both the vertex and the facet descriptions of this realization. On the one hand, for each maximal tubing~$\tubing$ on~$\starG_{n+1}$ we define a point~$\b{x}(\tubing) \in \R^n$ whose~$i$th coordinate is the cardinality of the inclusion minimal tube of~$\tubing \cup \{\ground\}$ containing the leaf~$\ell_i$ minus~$1$. The set~$\set{\b{x}(\tubing)}{\tubing \text{ maximal tubing on } \starG_{n+1}}$ is the orbit under permutation coordinates of the set~$\set{\sum_{i > k} i \, \b{e}_i}{0 \le k \le n}$. On the other hand, for a tube~$\tube$ of~$\starG_{n+1}$ containing the central vertex~$*$, we observed earlier that the compatibility vector of~$\tube$ with respect to~$\tubing^\circ$ is the characteristic vector of the leaves of~$\starG_{n+1}$ not contained in~$\tube$. Let~${f(k) \eqdef \sum_{j = k}^{n} j = \frac{1}{2}(n+k)(n+1-k)}$ and define a half-space~$\HS(\tube)$ of~$\R^n$ by
\[
\HS(\tube) \eqdef \biggset{\b{x} \in \R^n}{\dotprod{\compatibilityVector{\tubing^\circ}{\tube}}{\b{x}} \le f(|\tube|)} = \biggset{\b{x} \in \R^n}{\sum_{\substack{i \in [n] \\ \ell_i \in \tube}} x_i \le f(|\tube|)}.
\]
Finally, for the tubes of the initial tubing~$\tubing^\circ$, we define
\[
\HS(\{\ell_i\}) \eqdef \set{\b{x} \in \R^n}{\dotprod{\compatibilityVector{\tubing^\circ}{\{\ell_i\}}}{\b{x}} \le 0} = \set{\b{x} \in \R^n}{x_i \ge 0}.
\]

\begin{proposition}
\label{prop:polytopeStar}
The compatibility fan~$\compatibilityFan{\starG_{n+1}}{\tubing^\circ}$ for the initial tubing~$\tubing^\circ \eqdef \big\{ \{\ell_1\}, \dots, \{\ell_n\} \big\}$ whose tubes are the $n$ leaves~of~$\starG_{n+1}$ is the normal fan of the $n$-dimensional simple polytope defined equivalently as
\begin{itemize}
\item the convex hull of the points~$\b{x}(\tubing)$ for all maximal tubings~$\tubing$ on~$\starG_{n+1}$, or
\item the intersection of the half-spaces~$\HS(\tube)$ for all tubes~$\tube$ of~$\starG_{n+1}$.
\end{itemize}
\end{proposition}

\begin{figure}
  \capstart
  \centerline{\includegraphics[scale=.36]{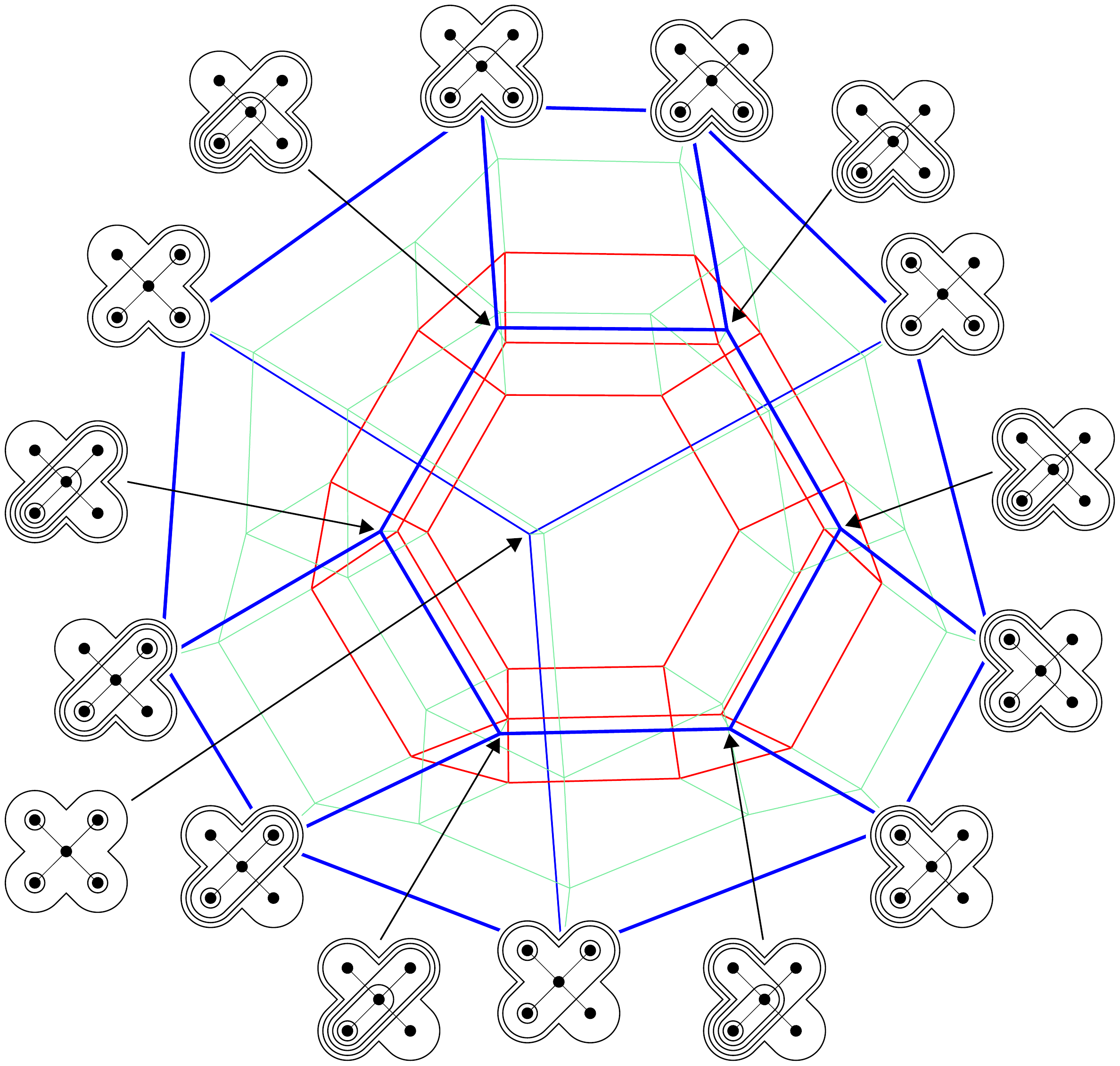}\quad\includegraphics[scale=.36]{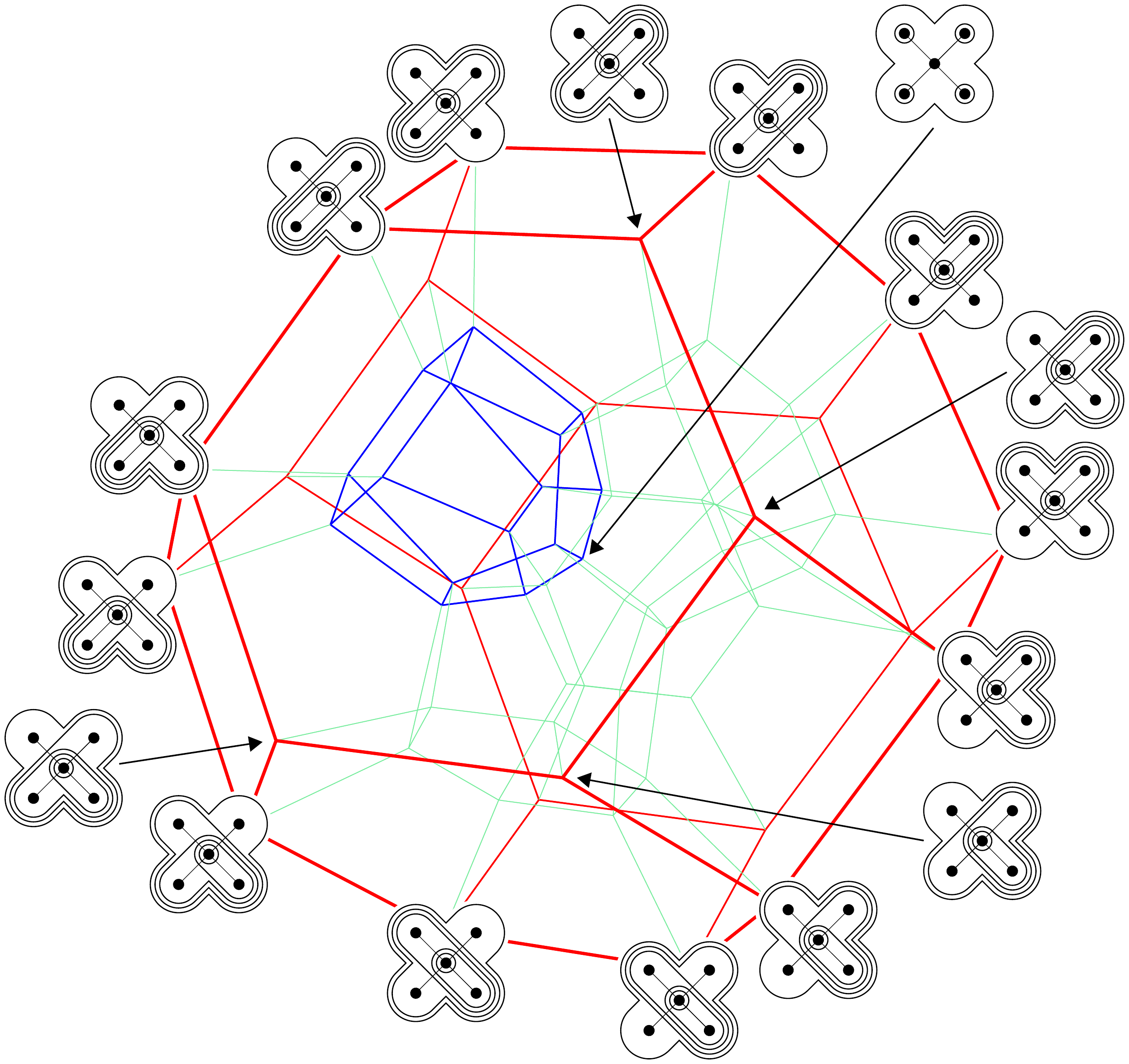}}
  \caption{Two Schlegel diagrams for the $4$-dimensional stellohedron defined in Proposition~\ref{prop:polytopeStar}. The red facet corresponds to all tubings containing the tube~$\{*\}$ while the blue facet corresponds to all tubings containing the tube~$\{\ell\}$, where~$\ell$ is the bottom left leaf of~$\starG_5$.}
  \label{fig:stellohedron4}
\end{figure}

The proof of this statement is given in Section~\ref{subsec:proofPolytopality}. As an illustration, the $3$-dimensional stellohedron defined in Proposition~\ref{prop:polytopeStar} is represented in \fref{fig:stellohedron3}\,(right). \fref{fig:stellohedron4} represents two Schlegel diagrams (see~\cite[Lecture~5]{Ziegler} for definition) for the $4$-dimensional stellohedron defined in Proposition~\ref{prop:polytopeStar}. In both pictures, we have distinguished two particular facets:
\begin{itemize}
\item The red facet corresponds to all tubings containing the tube~$\{*\}$. Since the reconnected complement of~$\{*\}$ in~$\starG_5$ is the complete graph~$\completeG_4$, it has the combinatorics of the permutahedron. In fact, by definition of our polytopal realization, this facet is the classical permutahedron, obtained as the convex hull of the orbit of~$\sum_{i \in [4]} i \, \b{e}_i$ under permutation of the coordinates.
\item The blue facet corresponds to all tubings containing the tube~$\{\ell\}$, where~$\ell$ is the bottom left leaf of~$\starG_5$. This facet contains the initial tubing~$\tubing^\circ$ at the back. Note that there are $4$ isometric facets to this blue facet, corresponding to the four leaves of~$\starG_5$. This is visible in \fref{fig:stellohedron4}\,(right).
\end{itemize}
The blue (resp.~red) facet is the projection facet on the left (resp.~right) picture.

\begin{remark}
To conclude, observe that we could have replaced the function~$f$ in the definition of the half-spaces~$\HS(\tube)$ by any concave function. This follows from Proposition~\ref{prop:polytopalityFan} since the linear dependence between the compatibility vectors of the tubes of~$\tubing \cup \tubing'$ is given for any adjacent maximal tubings~$\tubing, \tubing'$ on~$\starG_{n+1}$ distinct from~$\tubing^\circ$ such that~$\tubing \ssm \{\tube\} = \tubing' \ssm \{\tube'\}$ by
\[
\compatibilityVector{\tubing^\circ}{\tube} + \compatibilityVector{\tubing^\circ}{\tube'} = \compatibilityVector{\tubing^\circ}{\tinf} + \compatibilityVector{\tubing^\circ}{\tsup}
\]
where~$\tsup \eqdef \tube \cap \tube'$ and~$\tinf \eqdef \tube \cap \tube'$ (which are tubes of~$\starG_{n+1}$). Details are left to the reader.
\end{remark}

\subsection{Design nested complex}
\label{subsec:designNestedComplex}

Generalizing graphical nested complexes, S.~Devadoss, T.~Heath and C.~Vipismakul introduced design nested complexes in~\cite[Section~5]{DevadossHeathVipismakul}. To define these complexes, one considers \defn{design tubes} of~$\graphG$, which are of two types: 
\begin{itemize}
\item the \defn{round tubes} are usual tubes of~$\graphG$ (including the connecting components of~$\graphG$),
\item the \defn{square tubes} are just single nodes of~$\graphG$.
\end{itemize}
We denote by~$\squareTube{v}$ the square tube containing~$v$, and still denote round tubes as sets. Two design tubes are \defn{compatible} if 
\begin{itemize}
\item they are both round tubes and they are either nested, or disjoint and non-adjacent,
\item or at least one of them is a square tube and they are not nested.
\end{itemize}
The \defn{design nested complex} of~$\graphG$ is the simplicial complex~$\designNestedComplex(\graphG)$ of sets of pairwise compatible design tubes of~$\graphG$. Examples are given in \fref{fig:designGraphAssociahedra}. By definition, the nested complex~$\nestedComplex(\graphG)$ is (isomorphic to) the subcomplex of the design nested complex~$\designNestedComplex(\graphG)$ involving none of the square tubes, or equivalently containing all improper round tubes.

For a design tube~$\tube$ of~$\graphG$, set
\[
\b{g}\design(\tube) \eqdef
\begin{cases}
\sum_{v \in \tube} \b{e}_v & \text{if } \tube \text{ is a round tube,} \\
-\b{e}_v & \text{if } \tube \text{ is the square tube } \{v\}.
\end{cases}
\]
For a tubing~$\tubing$ on~$\graphG$, define~$\b{g}\design(\tubing) \eqdef \set{\b{g}\design(\tube)}{\tube \in \tubing}$. These vectors again support a complete simplicial fan realization of the design nested complex~$\designNestedComplex(\graphG)$.

\begin{theorem}[\cite{DevadossHeathVipismakul}]
For any graph~$\graphG$, the collection of cones
\[\cG\design(\graphG) \eqdef \set{\R_{\ge 0} \, \b{g}\design(\tubing)}{\tubing \text{ tubing on } \graphG}\]
is a complete simplicial fan of~$\R^\ground$, called \defn{design nested fan} of~$\graphG$, which realizes~$\designNestedComplex(\graphG)$.
\end{theorem}

By definition, the $\b{g}\design$ vector of a round tube~$\tube$ is just the characteristic vector of~$\tube$. Therefore, the non-negative part of the design nested fan~$\cG\design(\graphG)$ projects to the nested fan~$\cG(\graphG)$ defined in Section~\ref{subsec:graphAssociahedra}. Using ideas similar to the construction of the graph associahedra~$\Asso(\graphG)$ realizing~$\cG(\graphG)$, S.~Devadoss, T.~Heath and C.~Vipismakul prove that~$\cG\design(\graphG)$ is as well polytopal.

\begin{theorem}[\cite{DevadossHeathVipismakul}]
The design nested fan is the normal fan of a polytope obtained from the cube by iterated face truncations.
\end{theorem}

\begin{figure}
  \capstart
  \centerline{\includegraphics[scale=1.3]{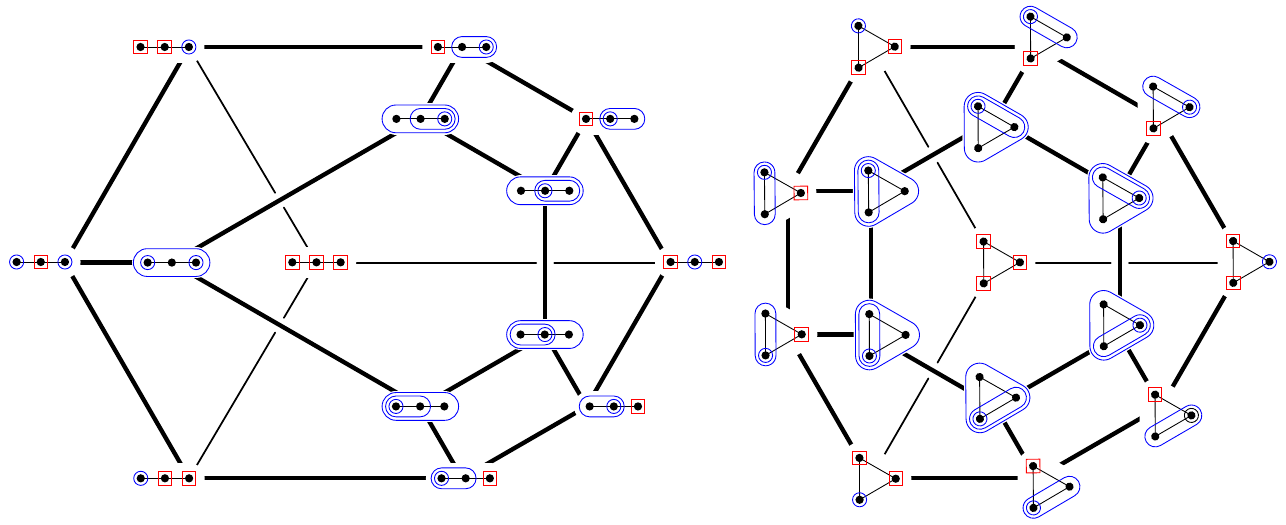}}
  \caption{Two design graph associahedra: the design $\pathG_3$-associahedron (left) and the design $\completeG_3$-associahedron (right). For readability, round tubes are colored blue while square tubes are colored red.}
  \label{fig:designGraphAssociahedra}
\end{figure}

We denote this polytope by~$\Asso\design(\graphG)$ and call it \defn{design graph associahedron} (although it is called \emph{graph cubeahedron} in~\cite{DevadossHeathVipismakul}). Observe that the face of~$\Asso\design(\graphG)$ corresponding to the tubing formed by the connected components of~$\graphG$ coincides with the graph associahedron~$\Asso(\graphG)$. \fref{fig:designGraphAssociahedra} illustrates the design $\pathG_3$-associahedron and the design $\completeG_3$-associahedron. In this figure, the attentive reader should recognize different standard graph associahedra: on the one hand, the front $2$-dimensional faces of~$\Asso\design(\pathG_3)$ and~$\Asso\design(\completeG_3)$ are respectively~$\Asso(\pathG_3)$ and~$\Asso(\completeG_3)$, and on the other hand, the design graph associahedra~$\Asso\design(\pathG_3)$ and~$\Asso\design(\completeG_3)$ themselves turn out to coincide respectively with~$\Asso(\pathG_4)$ and~$\Asso(\starG_4)$ (see Example~\ref{exm:isomorphismDesignNormal}). Starting from dimension~$4$, most design graph associahedra are not standard graph associahedra (see Proposition~\ref{prop:isomorphismDesignNormal}). Figures~\ref{fig:designCycleAssociahedron} and~\ref{fig:designStarAssociahedron} represent two Schlegel diagrams (see~\cite[Lecture~5]{Ziegler} for definition) for the $4$-dimensional design cycle associahedron~$\Asso\design(\cycleG_4)$ and design star associahedron~$\Asso\design(\starG_4)$.

\begin{figure}[p]
  \capstart
  \centerline{\includegraphics[scale=.36]{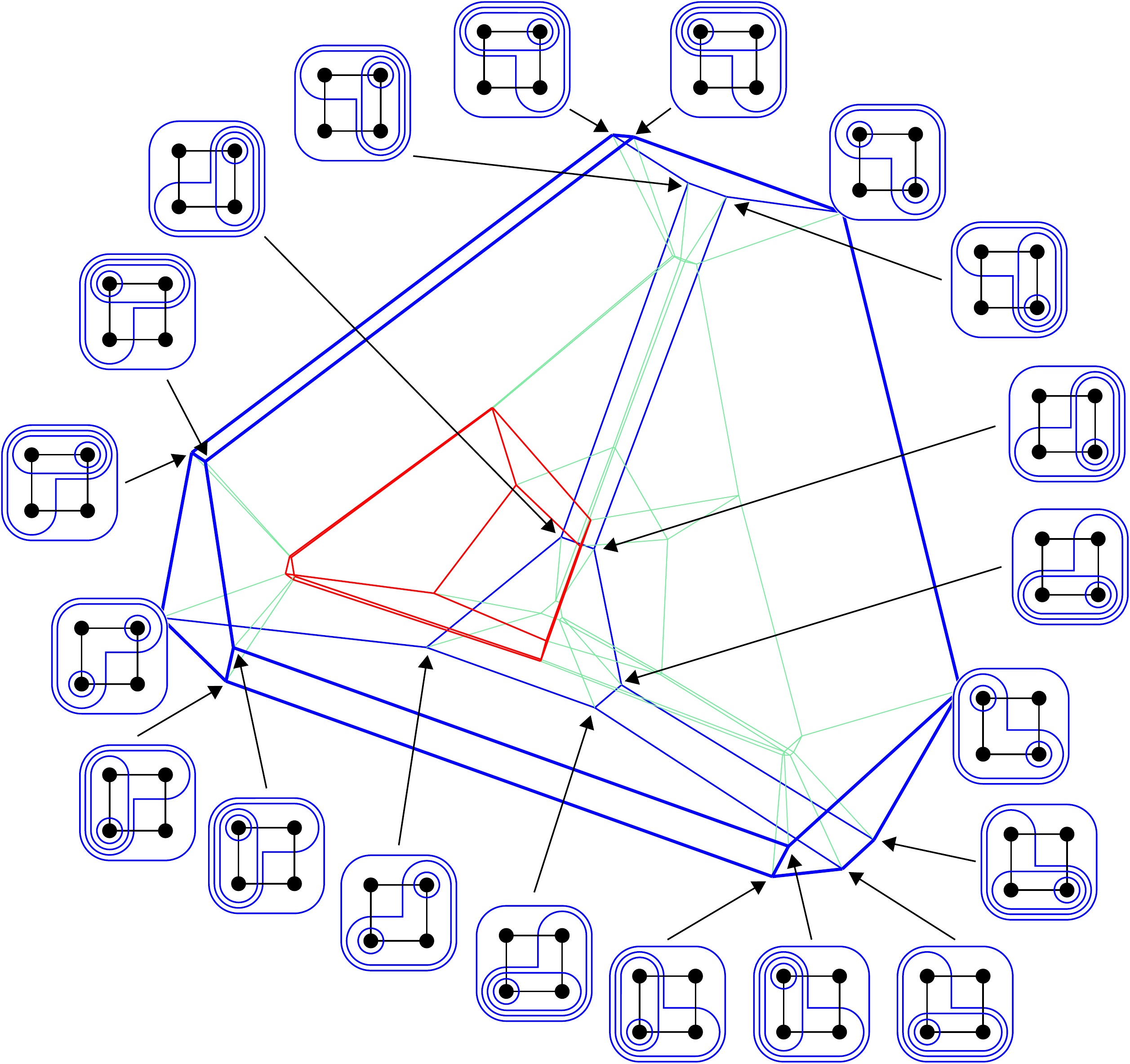}\quad\raisebox{.3cm}{\includegraphics[scale=.36]{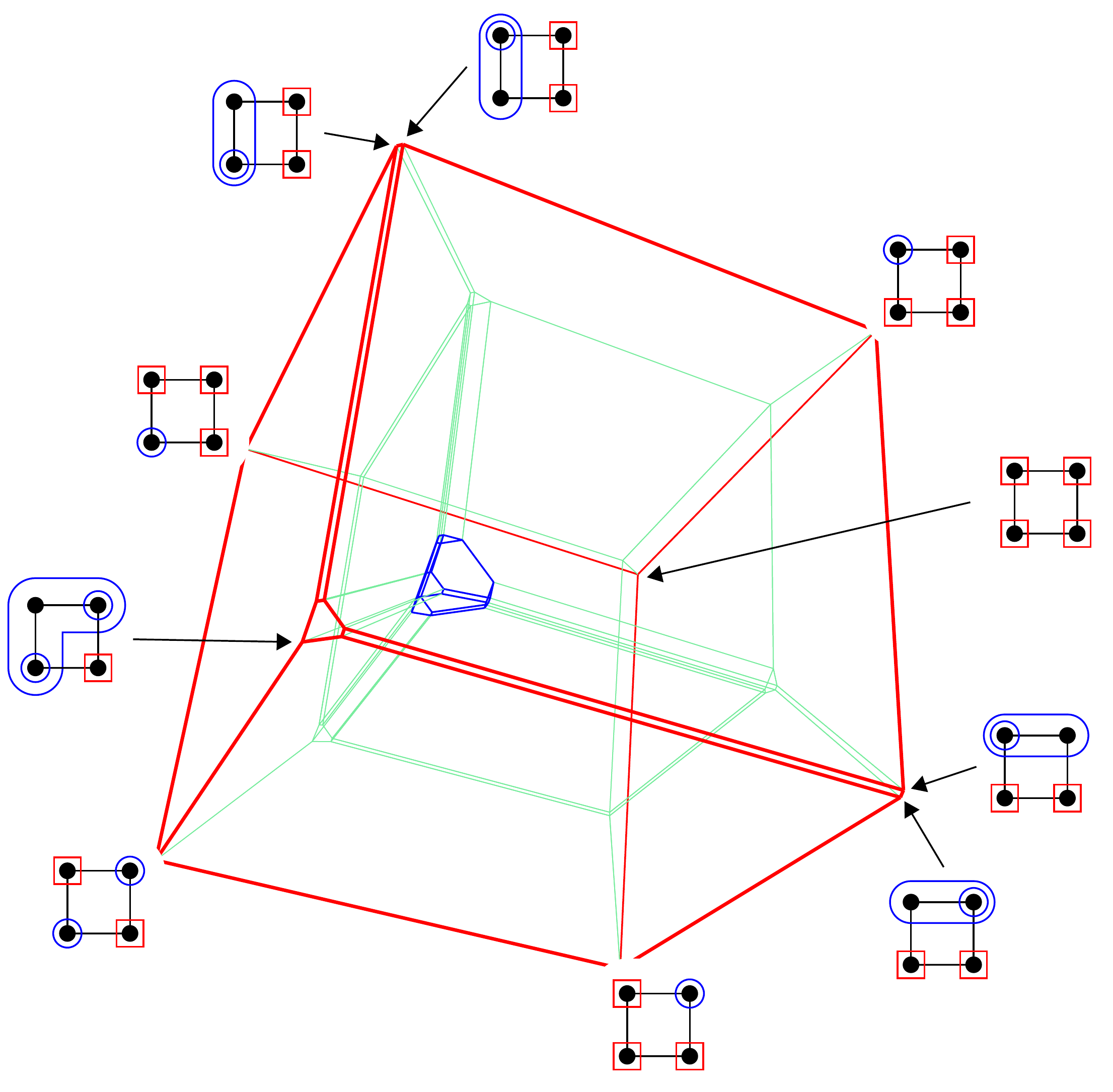}}}
  \caption{Two Schlegel diagrams for the $4$-dimensional design cycle associahedron. The blue facet corresponds to all round tubings while the red facet corresponds to all tubings containing the bottom right square tube of~$\cycleG_4$.}
  \label{fig:designCycleAssociahedron}
  \vspace{.5cm}
\end{figure}

\begin{figure}[p]
  \capstart
  \centerline{\raisebox{.4cm}{\includegraphics[scale=.36]{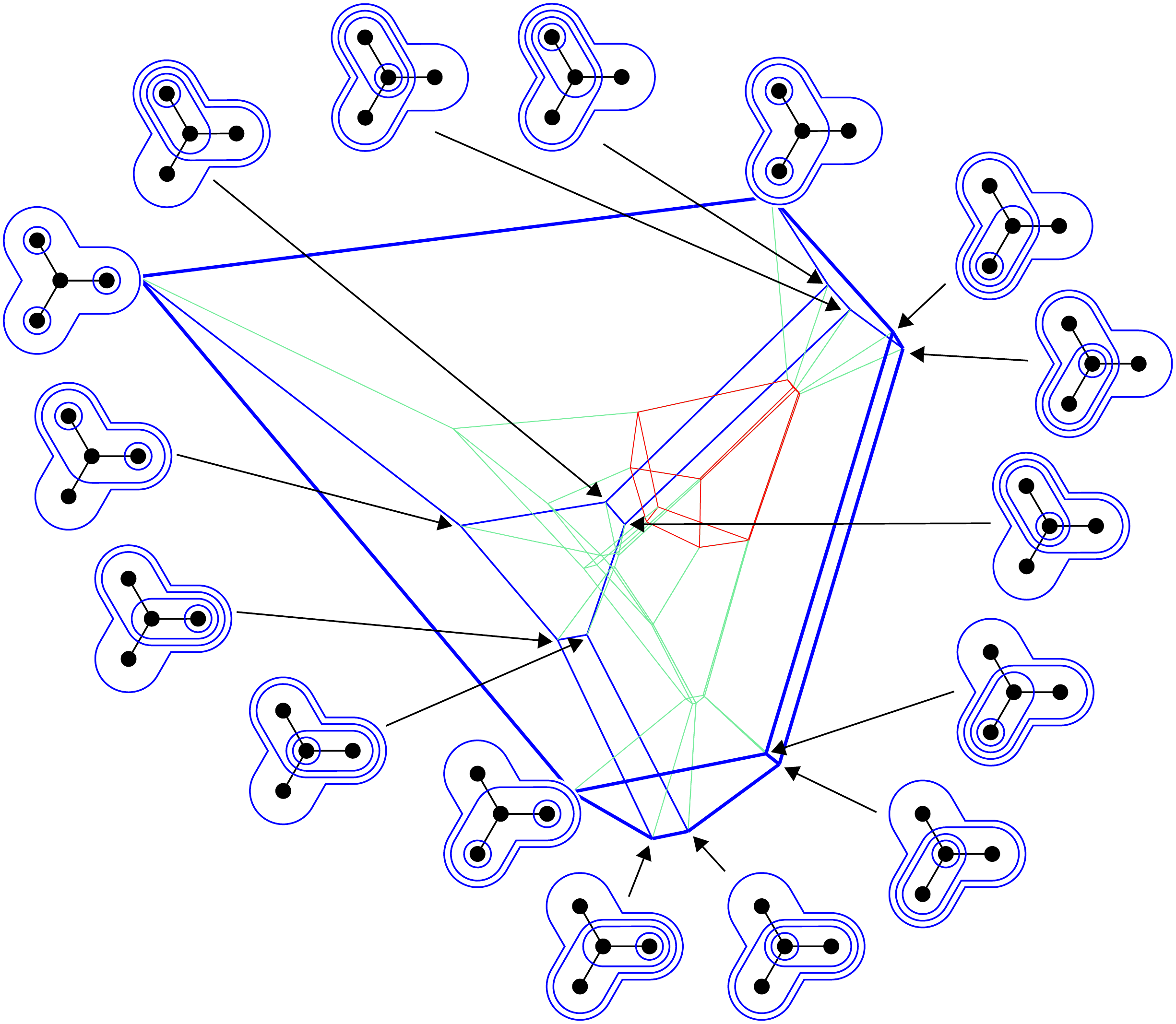}}\quad\includegraphics[scale=.36]{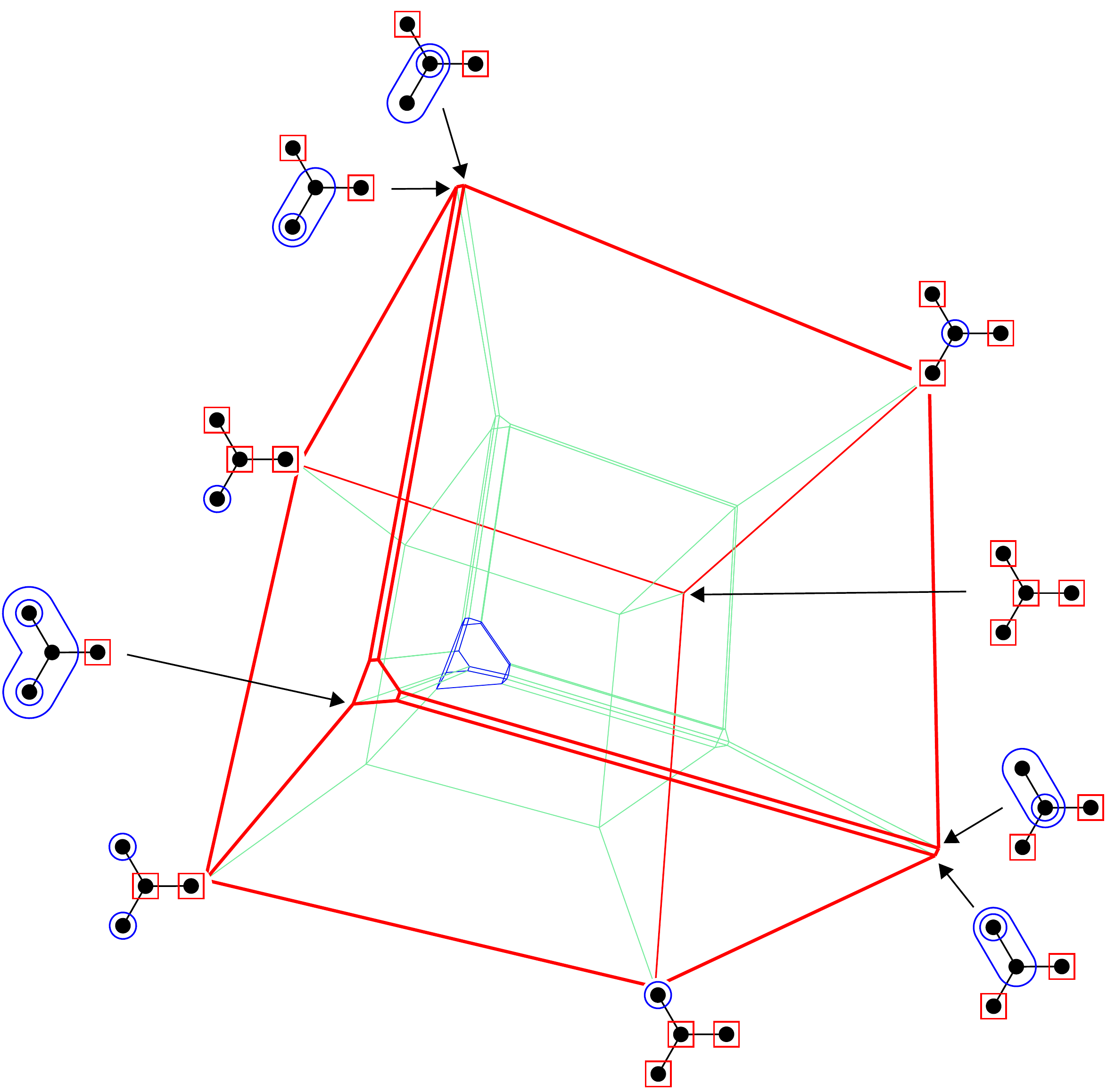}}
  \caption{Two Schlegel diagrams for the $4$-dimensional design star associahedron. The blue facet corresponds to all round tubings while the red facet corresponds to all tubings containing the right square tube of~$\starG_4$.}
  \label{fig:designStarAssociahedron}
\end{figure}

This section aims at showing that our compatibility fan construction extends to the design nested complex. That is, we produce a complete simplicial fan realizing the design nested complex from any initial maximal design tubing. Interestingly, we will see in Remark~\ref{rem:connectionConstructions} that the compatibility fan associated to the specific maximal design tubing consisting of all square tubes coincides with the design nested fan. This provides a relevant connection between our construction and the classical constructions~\cite{CarrDevadoss, FeichtnerSturmfels, Devadoss, Zelevinsky} for graph associahedra.

Observe that a square tube is compatible with all design tubes not containing it and exchangeable with all design tubes containing it. Definition~\ref{def:compatibilityDegree} of compatibility degree thus naturally extends on all pairs of design tubes as follows.

\begin{definition}
For two design tubes~$\tube, \tube'$ of~$\graphG$, the \defn{compatibility degree} of~$\tube$ with~$\tube'$ is
\[
\compatibilityDegree{\tube}{\tube'} =
\begin{cases}
-1 & \text{if } \tube = \tube', \\
1 & \text{if } \tube \text{ and } \tube' \text{ are nested and exactly one of them is square,} \\
|\{\text{neighbors of $\tube$ in } \tube' \ssm \tube\}| & \text{if } \tube \text{ and } \tube' \text{ are round and } \tube \not\subseteq \tube', \\
0 & \text{otherwise.}
\end{cases}
\]
\end{definition}

By construction, this compatibility degree still satisfies the conclusions of Proposition~\ref{prop:compatibilityDegree} and thus measures the incompatibility between design tubes. We then define as usual the \defn{compatibility vector} of a design tube~$\tube$ of~$\graphG$ with respect to an initial maximal design tubing~$\tubing^\circ \eqdef \{\tube_1^\circ, \dots, \tube_n^\circ\}$ as the integer vector $\compatibilityVector{\tubing^\circ}{\tube} \eqdef [\compatibilityDegree{\tube_1^\circ}{\tube}, \dots, \compatibilityDegree{\tube_n^\circ}{\tube}]$ and the \defn{compatibility matrix} of a design tubing~$\tubing \eqdef \{\tube_1, \dots, \tube_m\}$ on~$\graphG$ with respect to~$\tubing^\circ$ as the matrix $\compatibilityVector{\tubing^\circ}{\tubing} \eqdef [\compatibilityDegree{\tube_i^\circ}{\tube_j}]_{i \in [n], j \in [m]}$. We extend Theorem~\ref{theo:compatibilityFan} in the following statement, whose proof is sketched in Section~\ref{subsec:proofDesignCompatibilityFan}.

\begin{theorem}
\label{theo:compatibilityFanDesign}
For any graph~$\graphG$ and any maximal design tubing~$\tubing^\circ$ on~$\graphG$, the collection of cones
\[
\designCompatibilityFan{\graphG}{\tubing^\circ} \eqdef \set{\R_{\ge 0} \, \compatibilityVector{\tubing^\circ}{\tubing}}{\tubing \text{ design tubing on } \graphG}
\]
is a complete simplicial fan which realizes the design nested complex~$\designNestedComplex(\graphG)$. We call it the \defn{design compatibility fan} of~$\graphG$ with respect to~$\tubing^\circ$.
\end{theorem}

\enlargethispage{-.5cm}
Using the same duality trick as in the proof of Theorem~\ref{theo:dualCompatibilityFan}, the reader can obtain as well dual design compatibility fans.

\bigskip
Concerning isomorphisms, we have similar results as in Section~\ref{subsec:many}. Notice first that the conclusions of Proposition~\ref{prop:nestedComplexIsomorphismDisconnected} still hold for design nested complexes. We can thus restrict our discussion to design nested complexes of connected graphs. We first compare design and standard nested complexes, starting with the following examples.

\begin{example}
\label{exm:isomorphismDesignNormal}
The reader can check that:
\begin{enumerate}[(i)]
\item The design nested complex~$\designNestedComplex(\completeG_n)$ is isomorphic to the standard nested complex~$\nestedComplex(\starG_{n+1})$. A natural isomorphism sends a square tube~$\squareTube{v}$ of~$\completeG_n$ to the tube~$\{v\}$ of~$\starG_{n+1}$, and a round tube~$\tube$ of~$\completeG_n$ to the tube~$\{*\} \cup ([n] \ssm \tube)$ of~$\starG_{n+1}$ (where~$*$ is the central vertex). See \fref{fig:isomorphismDesignCompleteStar}.

\begin{figure}
  \capstart
  \centerline{\includegraphics[width=\textwidth]{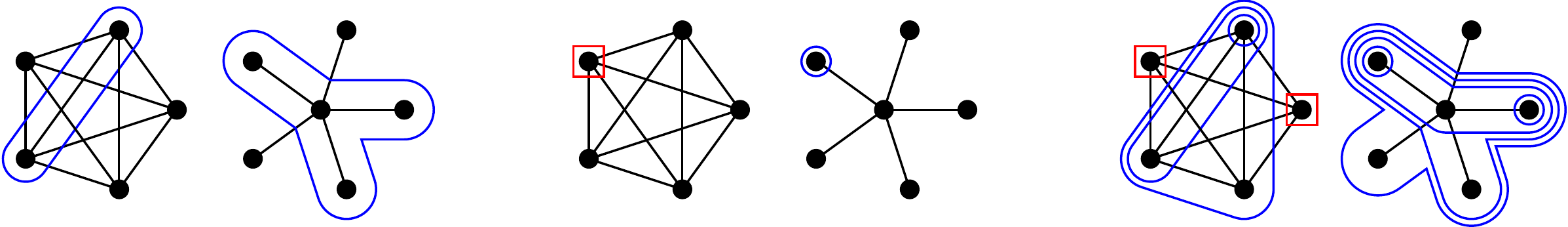}}
  \caption{An isomorphism from the design nested complex of a complete graph with~$5$ vertices (left) to the nested complex of a star with~$6$ vertices~(right).}
  \label{fig:isomorphismDesignCompleteStar}
\end{figure}

\item The design nested complex~$\designNestedComplex(\pathG_n)$ is isomorphic to the standard nested complex~$\nestedComplex(\pathG_{n+1})$. A natural isomorphism sends a square tube~$\squareTube{v}$ of~$\pathG_n$ to the tube~$\{v+1, \dots, n+1\}$ of~$\pathG_{n+1}$, and a round tube~$\tube$ of~$\pathG_n$ to the tube~$\tube$ of~$\pathG_{n+1}$. See \fref{fig:isomorphismDesignPath}. We denote this isomorphism by~$\Pi : \designNestedComplex(\pathG_n) \to \nestedComplex(\pathG_{n+1})$.

\begin{figure}
  \capstart
  \centerline{\includegraphics[width=\textwidth]{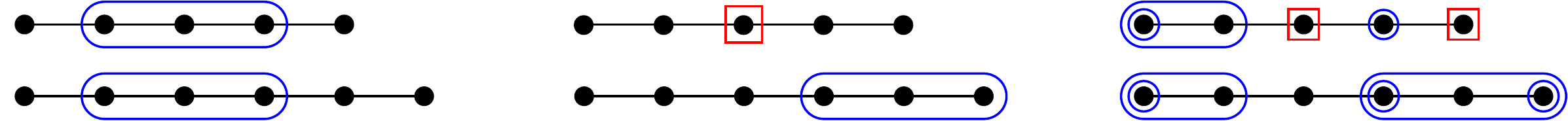}}
  \caption{An isomorphism from the design nested complex of a path with~$5$ vertices (top) to the nested complex of a path with~$6$ vertices (bottom).}
  \label{fig:isomorphismDesignPath}
\end{figure}

\end{enumerate}
\end{example}

We now describe a generalization of both cases of Example~\ref{exm:isomorphismDesignNormal}. We consider the spiders and octopuses defined in Section~\ref{subsec:many}: for~${\ninf \eqdef \{n_1, \dots, n_\ell\}} \in \N^\ell$ with~$n = \sum_{i \in [\ell]} (n_i + 1)$,
\begin{itemize}
\item the spider~$\spiderG_{\ninf}$ has~$\ell$ legs~$[v_1^i, v_{n_i}^i]$ attached to its body formed by a complete graph~on~$\{v_0^i\}_{i \in [\ell]}$,
\item the octopus~$\octopusG_{\ninf}$ has~$\ell$ legs~$[v_0^i, v_{n_i}^i]$ attached to its head formed by a single vertex~$*$.
\end{itemize}

We now define an isomorphism~$\bar\Omega$ from the design nested complex~$\designNestedComplex(\spiderG_{\ninf})$ of the spider~$\spiderG_{\ninf}$ to the nested complex~$\nestedComplex(\octopusG_{\ninf})$ of the octopus~$\octopusG_{\ninf}$. We distinguish three kinds of tubes of~$\designNestedComplex(\spiderG_{\ninf})$:
\begin{description}
\item[Square tubes] for~$ i \in [\ell]$ and~$j \in[0,n_i]$, \;\;\;\, $\bar\Omega(\squareTube{v^i_j}) \eqdef \big[ v^i_0, v^i_{n_i-j} \big]$,
\item[Leg tubes] for~$i \in [\ell]$ and~$1 \le j \le k \le n_i$, \; $\bar\Omega\big( \big[ v^i_j, v^i_k \big] \big) \eqdef \big[ v^i_{n_i+1-k}, v^i_{n_i+1-j} \big]$,
\item[Body tubes] for~$-1 \le k_i \le n_i$ ($i \in [\ell]$), \quad\; $\bar\Omega\big( \bigcup\limits_{i \in [\ell]} \big[ v^i_0, v^i_{k_i} \big] \big) \eqdef \{\ast\} \cup \bigcup\limits_{i \in [\ell]} \big[ v^i_0, v^i_{n_i-1-k_i} \big]$.
\end{description}
\fref{fig:isomorphismDesignSpiderOctopus} illustrates the map~$\bar\Omega$ on different tubes of the spider~$\spiderG_{\{0,3,2,3,0,3,2,3\}}$.
Observe that~$\bar\Omega$ indeed generalizes both isomorphisms of Example~\ref{exm:isomorphismDesignNormal}:
\begin{enumerate}[(i)]
\item We have~$\completeG_n = \spiderG_{\{0\}^n}$ while~$\starG_{n+1} = \octopusG_{\{0\}^n}$, and the isomorphism~$\bar\Omega : \designNestedComplex(\spiderG_{\{0\}^n}) \to \nestedComplex(\octopusG_{\{0\}^n})$ coincides with the isomorphism of Example~\ref{exm:isomorphismDesignNormal}\,(i). 
\item We have~$\pathG_n = \spiderG_{\{n\}}$ while~$\pathG_{n+1} = \octopusG_{\{n\}}$, and the isomorphism~$\bar\Omega : \designNestedComplex(\spiderG_{\{n\}}) \to \nestedComplex(\octopusG_{\{n\}})$ coincides with the isomorphism~$\Pi$ of Example~\ref{exm:isomorphismDesignNormal}\,(ii) up to the automorphisms~$\rot$ and~$\rev$ of the nested complex~$\nestedComplex(\octopusG_{\{n\}})$ described in Example~\ref{exm:nestedComplexIsomorphisms}\,(ii). More precisely if we consider the leftmost vertex of~$\pathG_n$ as the body of~$\spiderG_{\{n\}}$ and the leftmost vertex of~$\pathG_{n+1}$ as the head of~$\octopusG_{\{n\}}$, then one can check that~${\rot} \circ {\bar\Omega} = {\rev} \circ {\Pi}$.
\end{enumerate}
\begin{figure}
  \capstart
  \centerline{\includegraphics[width=\textwidth]{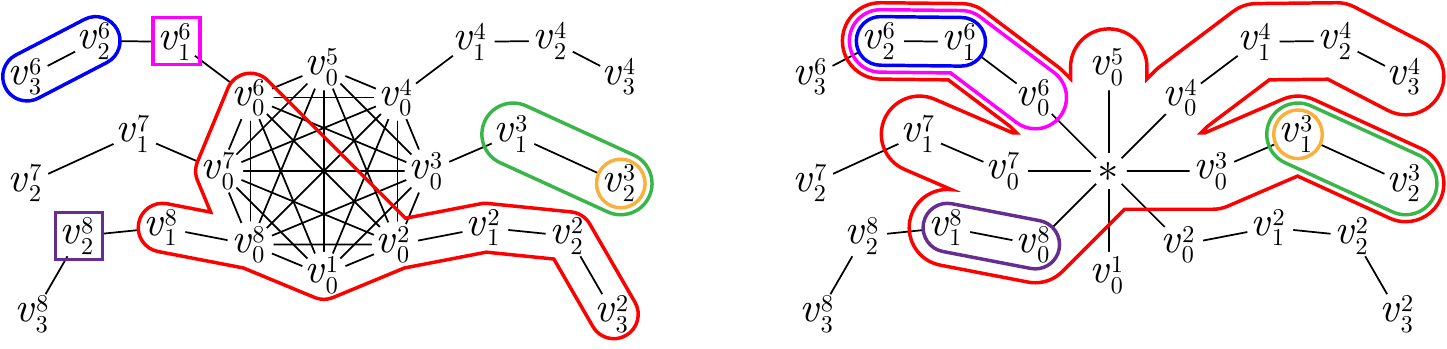}}
  \caption{An isomorphism from the design nested complex of the spider $\spiderG_{\{0,3,2,3,0,3,2,3\}}$ (left) to the nested complex of the octopus~$\octopusG_{\{0,3,2,3,0,3,2,3\}}$~(right).}
  \label{fig:isomorphismDesignSpiderOctopus}
\end{figure}
The following statement is left to the reader. The proof is similar to that of Proposition~\ref{prop:exmAutomorphism}.

\begin{proposition}
The map~$\bar\Omega$ is an isomorphism from the design nested complex~$\nestedComplex(\spiderG_{\ninf})$ of the spider~$\spiderG_{\ninf}$ to the nested complex~$\nestedComplex(\octopusG_{\ninf})$ of the octopus~$\octopusG_{\ninf}$ which dualizes the compatibility degree: $\compatibilityDegree{\bar\Omega(\tube)}{\bar\Omega(\tube')} = \compatibilityDegree{\tube'}{\tube}$.
\end{proposition}

The following proposition, proved in Section~\ref{subsec:proofDesignIsomorphisms}, states that~$\bar\Omega$ is essentially the only isomorphism between design and standard nested complexes.

\begin{proposition}
\label{prop:isomorphismDesignNormal}
Let~$\bar\graphG$ and~$\graphG$ be two connected graphs and~$\Phi : \designNestedComplex(\bar\graphG) \to \nestedComplex(\graphG)$ be a simplicial complex isomorphism. Then~$\bar\graphG$ is a spider~$\spiderG_{\ninf}$ while~$\graphG$ is an octopus~$\octopusG_{\ninf}$ and~$\Phi$ coincides with~$\bar\Omega$ up to composition with a nested complex automorphism of~$\nestedComplex(\graphG)$ (described in Theorem~\ref{theo:nestedComplexIsomorphisms}).
\end{proposition}

We now classify combinatorial isomorphisms of design nested complexes. As for nested complexes, a graph isomorphism~$\phi : \graphG \to \graphG'$ induces a \defn{trivial} design nested complex isomorphism ${\Phi : \designNestedComplex(\graphG) \to \designNestedComplex(\graphG')}$, defined by~$\Psi(\tube) \eqdef \set{\psi(v)}{v \in \tube}$ and~$\Psi(\squareTube{v}) \eqdef \squareTube{\psi(v)}$, which preserves compatibility degrees. We again focus on non-trivial design nested complex isomorphisms. We first underline two examples.

\begin{example}
\label{exm:isomorphismDesignNestedComplexes}
The reader can check that:
\begin{enumerate}[(i)]
\item For the star~$\starG_n$ with central vertex~$*$, the map which preserves the square tube~$\squareTube{*}$, exchanges the square tube~$\squareTube{v}$ with the round tube~$\{v\}$, and exchanges any other round tube~$\tube$ with the round tube~$([n] \ssm \tube) \cup \{*\}$, is a non-trivial design nested complex isomorphism of~$\designNestedComplex(\starG_n)$.
\item For~$p \in [n+3]$, the conjugation~${\Pi \, \circ \! \rot^p \! \circ \, \Pi^{-1}}$ of the automorphism ${\rot^p \; : \nestedComplex(\pathG_{n+1}) \to \nestedComplex(\pathG_{n+1})}$ of Example~\ref{exm:nestedComplexIsomorphisms}\,(ii) by the isomorphism~$\Pi : \designNestedComplex(\pathG_n) \to \nestedComplex(\pathG_{n+1})$ of Example~\ref{exm:isomorphismDesignNormal}\,(ii) is a non-trivial design nested complex isomorphism of~$\designNestedComplex(\pathG_n)$.
\end{enumerate}
Note that these two isomorphisms send some square tubes to round tubes and thus are non-trivial.
\end{example}

We now describe a generalization of both cases of Example~\ref{exm:isomorphismDesignNestedComplexes}. Namely, we define an automorphism~$\Omega\design$ of the design nested complex~$\designNestedComplex(\octopusG_{\ninf})$ of the octopus~$\octopusG_{\ninf}$ for any~$\ninf \eqdef \{n1, \dots, n_\ell\} \in \N^\ell$ with~$n = \sum_{i \in [\ell]} (n_i + 1)$ as follows:
\begin{description}
\item[Square tubes] for~$i \in [\ell]$ and~$j \in [0, n_i]$, \qquad $\Omega\design(\squareTube{v^i_j}) \eqdef \big[ v^i_0, v^i_{n_i-j} \big]$ \quad and \quad $\Omega\design(\squareTube{\ast}) \eqdef \squareTube{\ast}$,
\item[Leg tubes] 
\begin{tabular}[t]{@{}ll}
for~$i \in [\ell]$ and~$0 \le k \le n_i$, & \; $\Omega\design \big( \big[ v^i_0, v^i_k \big] \big) \eqdef \squareTube{v^i_{n_i-k}}$, \\
for~$i \in [\ell]$ and~$1 \le j \le k \le n_i$, & \; $\Omega\design \big( \big[ v^i_j, v^i_k \big] \big) \eqdef \big[ v^i_{n_i+1-k}, v^i_{n_i+1-j} \big]$,
\end{tabular}
\item[Head tubes] for~$-1 \le k_i \le n_i$ ($i \in [\ell]$), \quad $\Omega\design\big( \{\ast\} \cup \bigcup\limits_{i \in [\ell]} \big[ v^i_0, v^i_{k_i} \big] \big) \eqdef \{\ast\} \cup \bigcup\limits_{i \in [\ell]} \big[ v^i_0, v^i_{n_i-1-k_i} \big]$.
\end{description}
\fref{fig:isomorphismDesignOctopus} illustrates the map~$\Omega\design$ on different tubes of the octopus~$\octopusG_{\{0,3,2,3,0,3,2,3\}}$.
\begin{figure}[t]
  \capstart
  \centerline{\includegraphics[scale=.9]{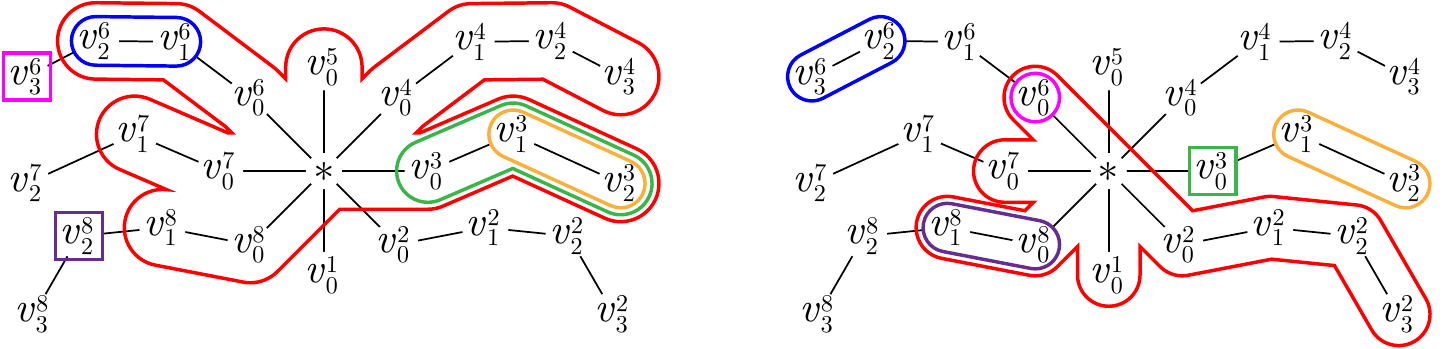}}
  \caption{The octopus~$\spiderG_{\{0,3,2,3,0,3,2,3\}}$ and examples of the action of the non-trivial design nested complex isomorphism~$\Omega\design$: the design tubing~$\tubing$ on the left is sent to the design tubing~$\Omega\design(\tubing)$ on the right.}
  \label{fig:isomorphismDesignOctopus}
\end{figure}
The reader is invited to check that~$\Omega\design$ indeed generalizes the non-trivial automorphisms of Example~\ref{exm:isomorphismDesignNestedComplexes}. The following statement is also left to the reader. The proof is similar to that of Proposition~\ref{prop:exmAutomorphism}.

\begin{proposition}
\label{prop:exmAutomorphismDesign}
The map~$\Omega\design$ is a non-trivial involutive automorphism of the design nested complex~$\nestedComplex(\octopusG_{\ninf})$ of the octopus~$\octopusG_{\ninf}$ which dualizes the compatibility degree: $\compatibilityDegree{\Omega\design(\tube)}{\Omega\design(\tube')} = \compatibilityDegree{\tube'}{\tube}$.
\end{proposition}

\enlargethispage{.5cm}
The following theorem, proved in Section~\ref{subsec:proofDesignIsomorphisms}, states that~$\Omega\design$ is essentially the only non-trivial design nested complex isomorphism.

\begin{theorem}
\label{theo:designNestedComplexIsomorphisms}
Let~$\graphG$ and~$\graphG'$ be two connected graphs and~$\Phi : \designNestedComplex(\graphG) \to \designNestedComplex(\graphG')$ be a non-trivial design nested complex isomorphism. Then~$\graphG$ and~$\graphG'$ are octopuses and there exists a graph isomorphism~$\psi : \graphG \to \graphG'$ which induces a design nested complex isomorphism~${\Psi : \designNestedComplex(\graphG) \to \designNestedComplex(\graphG')}$ (defined by~$\Psi(\tube) \eqdef \set{\psi(v)}{v \in \tube}$ and~$\Psi(\squareTube{v}) \eqdef \squareTube{\psi(v)}$) such that the composition~$\Psi^{-1} \circ \Phi$ coincides with the non-trivial design nested complex automorphism~$\Omega\design$ on~$\designNestedComplex(\graphG)$.
\end{theorem}

We conclude that we get many design compatibility fans (up to linear isomorphism).

\begin{corollary}
\label{coro:classificationDesignPrimal}
If a connected graph~$\graphG$ is not an octopus, then the number of linear isomorphism classes of primal (resp.~dual) design compatibility fans of~$\graphG$ is the number of orbits of maximal design tubings on~$\graphG$ under graph automorphisms of~$\graphG$.
\end{corollary}

Finally, as for the compatibility fan, we do not know in general whether the design compatibility fan is or not polytopal. Nevertheless, we know it for the following initial tubing.

\begin{remark}
\label{rem:connectionConstructions}
Consider the maximal tubing~$\tubing\design$ of a graph~$\graphG$ consisting of all square tubes of~$\graphG$. By definition of the compatibility degree with square tubes, the compatibility vector of a square tube~$\{v\}$ is just~$-\b{e}_v$ while the compatibility vector of a round tube~$\tube$ with respect to~$\tubing\design$ is just the characteristic vector of~$\tube$. Therefore, the design compatibility fan~$\designCompatibilityFan{\graphG}{\tubing\design}$ coincides with the design nested fan~$\cG\design(\graphG)$. It follows that~$\designCompatibilityFan{\graphG}{\tubing\design}$ is the normal fan of the design graph associahedron~$\Asso\design(\graphG)$ for any graph~$\graphG$. In fact, as the linear dependencies among the compatibility vectors with respect to the initial tubing~$\tubing\design$ are simply described, a direct computation shows that the map~$\omega$ defined by~$\omega(\tube) = 3^{|\ground|}|\tube| - 3^{|\tube|}$ for a round tube and~$\omega(\squareTube{v}) = C$ for a sufficiently large constant~$C$ provides a suitable weight function for Condition~(2) of Proposition~\ref{prop:polytopalityFan}.
\end{remark}

\subsection{Laurent Phenomenon algebras}
\label{subsec:LPA}

To conclude, we discuss our construction with respect to the framework of \defn{Laurent Phenomenon algebras} developed by T.~Lam and P.~Pylyavskyy in~\cite{LamPylyavskyy-LaurentPhenomenonAlgebras, LamPylyavskyy-LinearLaurentPhenomenonAlgebras}. The definition of these algebras is similar to that of cluster algebras but allows more flexibility for the variable mutation. Namely, the cluster variables of these algebras are still obtained by a mutation process from an initial cluster~$X^\circ$, but the mutated variable depends upon some mutation polynomials which are themselves updated along the mutations. We refer to~\cite{LamPylyavskyy-LaurentPhenomenonAlgebras} for the precise rules directing these variable and polynomial mutations. 

Since the construction is more flexible, it yields more general algebras and cluster complexes. To our knowledge, there is no classification of the finite type Laurent Phenomenon algebras, in contrast to the case of classical cluster algebras. However, an important subclass is the class of linear Laurent Phenomenon algebras discussed in~\cite{LamPylyavskyy-LinearLaurentPhenomenonAlgebras}, where the mutation polynomials of a given cluster are all linear. A striking fact is that all graphical nested complexes appear as cluster complexes for some linear Laurent Phenomenon algebras. Namely, T.~Lam and P.~Pylyavskyy associate to each graph~$\graphG$ a Laurent Phenomenon algebra~$\cA(\graphG)$ whose cluster complex is isomorphic to the design nested complex~$\designNestedComplex(\graphG)$ discussed in Section~\ref{subsec:designNestedComplex}. The nested complex~$\nestedComplex(\graphG)$ is thus isomorphic to the cluster complex of the linear Laurent Phenomenon algebra obtained by freezing in~$\cA(\graphG)$ the variables corresponding to square tubes.

As the name suggests, Laurent Phenomenon algebras still exhibits the Laurent Phenomenon. Each cluster variable~$x$ can be expressed as a Laurent polynomial in terms of the cluster variables~$x_1^\circ, \dots, x_n^\circ$ of the initial cluster~$X^\circ$. Define the \defn{$\b{d}$-vector} of~$x$ with respect to~$X^\circ$ as the vector~$\b{d}(X^\circ, x)$ whose $i$th coordinate is the exponent of the initial variable~$x_i^\circ$ in the denominator of~$x$. It is then tempting to extend the construction of the $\b{d}$-vector fan of S.~Fomin and A.~Zelevinsky~\cite{FominZelevinsky-ClusterAlgebrasII} to all Laurent Phenomenon algebras: 

\begin{question}
\label{qu:LPA}
Given a Laurent Phenomenon algebra~$\cA$ and an initial cluster seed~$X^\circ$ of~$\cA$, do the cones generated by the $\b{d}$-vectors with respect to~$X^\circ$ of all collections of compatible cluster variables in~$\cA$ always form a complete simplicial fan realizing the cluster complex of~$\cA$?
\end{question}

When studying this question, we encountered several difficulties that we want to underline here:
\begin{itemize}
\item Although similar to the problem addressed in this paper, we have not been able to prove or disprove Question~\ref{qu:LPA} even for the specific case of linear Laurent Phenomenon algebras arising from nested complexes of graphs. In contrast to our combinatorial definition of compatibility degree among tubes, it is difficult to track the denominator vectors along the mutation process. In particular, expressing explicitly  the linear dependence among the denominator vectors of the cluster variables involved in a mutation seems out of reach at the moment.
\item One of the main difficulties is that, even in some finite cases, an entry of the denominator vector of a variable can depend on the whole initial cluster, not only on the variable corresponding to that entry. Namely, in certain linear Laurent Phenomenon algebras, one can find cluster variables~$x,x^\circ$ and initial clusters~$X^\circ, X'^\circ$ both containing~$x^\circ$ such that the exponent of~$x^\circ$ in the denominator of~$x$ is different when computed with respect to~$X^\circ$ or with respect to~$X'^\circ$. In this sense, the denominator vectors are not anymore compatibility vectors arising from a compatibility degree between cluster variables.
\item Even worst, given an initial cluster~$X^\circ$, an initial variable~$x^\circ$ and a cluster variable~$x$ not in~$X^\circ$ but compatible with~$x^\circ$, the exponent of~$x^\circ$ in the denominator~$x$, when expanded in the variables of~$X^\circ$ can be nonzero.
\item In particular, the denominator vectors in the linear Laurent Phenomenon algebra~$\cA(\graphG)$ do not always coincide with our compatibility vectors on tubes of~$\graphG$.
\end{itemize}

Despite all these difficulties, we hope that the study of the linear Laurent Phenomenon algebras arising from (design) nested complexes of graphs can help to answer Question~\ref{qu:LPA}. Since non-isomorphic Laurent Phenomenon algebra may have isomorphic cluster complexes (the cyclohedron is an example), another simpler question would be to look for other Laurent Phenomenon algebras whose denominators interpret our compatibility degrees.

\begin{question}
For which graph~$\graphG$ does there exist a Laurent Phenomenon algebra~$\cA(\graphG)$ whose cluster complex is isomorphic to the nested complex~$\Asso(\graphG)$ (resp.~to the design nested complex~$\Asso\design(\graphG)$) and whose denominators are given by the (primal or dual) compatibility degrees defined in this paper?
\end{question}


\section{Proofs}
\label{sec:proofs}

This section contains all proofs of our results. Some of them require additional technical steps, which motivated us to separate them from the rest of the paper. We also hope that the many examples treated in Section~\ref{sec:specificGraphs} help the reader's intuition throughout these proofs.

\subsection{Compatibility degree (Proposition~\ref{prop:compatibilityDegree})}
\label{subsec:proofCompatibilityDegree}

We start with the proof that our graphical compatibility degree encodes compatibility and exchangeability between tubes. We show the three points of Proposition~\ref{prop:compatibilityDegree}:

\medskip
\noindent$\bullet$ $\compatibilityDegree{\tube}{\tube'} < 0 \iff \compatibilityDegree{\tube'}{\tube} < 0 \iff \tube = \tube'$. \\
This is immediate from the definition since the compatibility degree between two distinct tubes is either a cardinal or~$0$.

\medskip
\noindent$\bullet$ $\compatibilityDegree{\tube}{\tube'} = 0 \iff \compatibilityDegree{\tube'}{\tube} = 0 \iff \tube$ and~$\tube'$ are compatible. \\
Consider two distinct tubes~$\tube, \tube'$ of~$\graphG$. If they are compatible, then either~$\tube \subseteq \tube'$, or~$\tube' \subseteq \tube$, or $\tube$ and~$\tube'$ are non-adjacent. In the first case,~$\compatibilityDegree{\tube}{\tube'} = 0$ by the last line of Definition~\ref{def:compatibilityDegree} of the compatibility degree. In the last two cases, $\compatibilityDegree{\tube}{\tube'} = |\{\text{neighbors of $\tube$ in } \tube' \ssm \tube\}| = |\varnothing| = 0$. Conversely, if~$\compatibilityDegree{\tube}{\tube'} = 0$, then either~$\tube$ has no neighbor in~$\tube'$, or~$\tube' \subseteq \tube$, or~$\tube \subseteq \tube'$, so that the two tubes are compatible.

\medskip
\noindent$\bullet$ $\compatibilityDegree{\tube}{\tube'} = 1 = \compatibilityDegree{\tube'}{\tube} \iff \tube$ and~$\tube'$ are exchangeable. \\
The $\Leftarrow$ part follows from the explicit flip description in Proposition~\ref{prop:flip}. Indeed, assume that~$\tube$ and~$\tube'$ are exchangeable, let~$\tsup \eqdef \tube \cup \tube'$, and let~$\tubing, \tubing'$ be two adjacent maximal tubings on~$\graphG$ such that~${\tubing \ssm \{\tube\} = \tubing' \ssm \{\tube'\}}$. Since $\tube'$ is the connected component of~$\graphG{}[\tsup \ssm \lab(\tube, \tubing)]$ containing~$\lab(\tsup, \tubing)$, the root~$\lab(\tsup, \tubing)$ is the unique neighbor of~$\tube$ in~$\tube' \ssm \tube$. Therefore, $\compatibilityDegree{\tube}{\tube'} = 1$, and~$\compatibilityDegree{\tube'}{\tube} = 1$ by symmetry. 

Assume conversely that~$\compatibilityDegree{\tube}{\tube'} = 1 = \compatibilityDegree{\tube'}{\tube}$. Since~$\compatibilityDegree{\tube'}{\tube} = 1$, there exists a unique neighbor~$r$ of~$\tube'$ in~$\tube \ssm \tube'$. Similarly, there exists a unique neighbor~$r'$ of~$\tube$ in~$\tube' \ssm \tube$. We want to find two adjacent maximal tubings~$\tubing, \tubing'$ on~$\graphG$ such that~$\tubing \ssm \{\tube\} = \tubing' \ssm \{\tube'\}$. We start with the forced tubes (see the end of Section~\ref{subsec:nestedComplex}): we define~$\tsup \eqdef \tube \cup \tube'$ and we let~$\tube[s]_1, \dots, \tube[s]_\ell$ be the connected components of~$\tsup \ssm \{r,r'\}$. We choose an arbitrary maximal tubing~$\tubing[S]_i$ on~$\graphG{}[\tube[s]_i]$ for each~$i$, and an arbitrary maximal tubing~$\tubing[S]$ on~$\graphG$ containing~$\tsup$. The set of tubes
\[
\tubing[R] \eqdef \{\tsup, \tube[s]_1, \dots, \tube[s]_\ell\} \; \sqcup \; \tubing[S]_1 \! \sqcup \dots \sqcup \! \tubing[S]_\ell \; \sqcup \; \set{\tube[s]}{\tube[s] \in \tubing[S], \tube[s] \not\subseteq \tsup}.
\]
is clearly a tubing, and is compatible with both~$\tube$ and~$\tube'$. We now compute the cardinality of~$\tubing[R]$. Observe first that $|\tubing[S]| = |\ground|-|\connectedComponents(\graphG)|$ so that~$|\set{\tube[s]}{\tube[s] \in \tubing[S], \tube[s] \not\subseteq \tsup}| = |\ground|-|\connectedComponents(\graphG)|-|\tsup|$ since~$\tsup$ is a tube of~$\graphG$. Moreover, $|\tubing[S]_i| = |\tube[s]_i| - 1$ since~$\tube[s]_i$ is a tube, and~$\sum_i |\tube[s]_i| = |\bigcup_i \tube[s]_i| = |\tsup \ssm \{r,r'\}| = |\tsup| - 2$. We conclude that
\[
|\tubing[R]| = (1 + \ell) + (|\tsup| - 2 - \ell) + (|\ground|-|\connectedComponents(\graphG)|-|\tsup|) = |\ground| - |\connectedComponents(\graphG)| - 1
\]
Therefore,~$\tubing[R]$ is a ridge of the nested complex~$\nestedComplex(\graphG)$, so that~$\tubing \eqdef \tubing[R] \cup \{\tube\}$ and~$\tubing' \eqdef \tubing[R] \cup \{\tube'\}$ are maximal tubings related by the flip of~$\tube$ into~$\tube'$. \qed 

\subsection{Restriction on coordinate hyperplanes (Proposition~\ref{prop:restriction})}
\label{subsec:proofLemRestriction}

We now state two lemmas needed in the proofs of Theorem~\ref{theo:compatibilityFan}. They have essentially the same content as Proposition~\ref{prop:restriction}, except that they focus on compatibility vectors (rays) and not on the other cones of the compatibility fan since Theorem~\ref{theo:compatibilityFan} is not proved yet. In particular, they will imply Proposition~\ref{prop:restriction} once Theorem~\ref{theo:compatibilityFan} will be established.

Remember that for a tube~$\tube^\circ$ of~$\graphG$, we denote by~$\graphG{}[\tube^\circ]$ the restriction of~$\graphG$ to~$\tube^\circ$ and by~$\graphG{}^\star\tube^\circ$ the reconnected complement of~$\tube^\circ$ in~$\graphG$, defined as the graph with vertex set~$\ground \ssm \tube^\circ$ and edge set ${\bigset{e \in \binom{\ground \ssm \tube^\circ}{2}}{\text{$e$ or~$e \cup \tube^\circ$ is connected in~$\graphG$}}}$.

\begin{lemma}[\cite{CarrDevadoss}]
\label{lem:restriction}
For a tube~$\tube^\circ$ of~$\graphG$, the map
\[
\tube[s] \longmapsto \widetilde{\tube[s]} \eqdef
\begin{cases}
\tube[s] & \text{if $\tube[s] \subsetneq \tube^\circ$} \\
\tube[s] \ssm \tube^\circ & \text{if $\tube[s] \supsetneq \tube^\circ$ or~$\tube[s] \cap \tube^\circ = \varnothing$}
\end{cases}
\]
between the tubes of~$\graphG$ compatible with~$\tube^\circ$ and the tubes of~$\widetilde\graphG \eqdef \graphG{}[\tube^\circ] \sqcup \graphG^\star\tube^\circ$ defines an isomorphism between the link of~$\tube^\circ$ in the nested complex~$\nestedComplex(\graphG)$ and the nested complex~$\nestedComplex(\widetilde\graphG)$.
\end{lemma}

We denote by~$\widetilde\tubing \eqdef \set{\widetilde\tube}{\tube \in \tubing}$ the image of a tubing~$\tubing$ on~$\graphG$. This map actually preserves the compatibility degrees between tubes and the compatibility vectors.

\begin{lemma}
\label{lem:restrictionCompatibility}
Let~$\tube^\circ$ be a tube of~$\graphG$. The map~$\tube[s] \mapsto \widetilde{\tube[s]}$ between the link of~$\tube^\circ$ in the nested complex~$\nestedComplex(\graphG)$ and the nested complex~$\nestedComplex(\widetilde\graphG)$ of the graph~$\widetilde\graphG \eqdef \graphG{}[\tube^\circ] \sqcup \graphG^\star\tube^\circ$ defined in Lemma~\ref{lem:restriction} preserves the compatibility degree: $\compatibilityDegree{\tube}{\tube'} = \compatibilityDegree{\,\widetilde\tube}{\widetilde\tube'}$ for any tubes~$\tube, \tube'$ of~$\graphG$ compatible with~$\tube^\circ$. Therefore, for any maximal tubing~$\tubing^\circ$ on~$\graphG$ containing~$\tube^\circ$ and any tube~$\tube$ of~$\graphG$ compatible with~$\tube^\circ$, the compatibility vector~$\compatibilityVector{\widetilde{\tubing}^\circ}{\widetilde\tube}$ is obtained from the compatibility vector~$\compatibilityVector{\tubing^\circ}{\tube}$ by deletion of its vanishing $\tube^\circ$-coordinate.
\end{lemma}

\begin{proof}
If~$\tube$ and~$\tube'$ are compatible, so are~$\widetilde\tube$ and~$\widetilde\tube'$ by Lemma~\ref{lem:restriction}, thus the result follows from Proposition~\ref{prop:compatibilityDegree}. We can therefore assume that~$\tube$ and~$\tube'$ are incompatible, so as~$\widetilde\tube$ and~$\widetilde\tube'$. Therefore, the compatibility degrees~$\compatibilityDegree{\tube}{\tube'}$ and~$\compatibilityDegree{\,\widetilde\tube}{\widetilde\tube'}$ actually count neighbors. However, it follows immediately from the definitions of the graph~$\widetilde\graphG$ and of the map~$\tube[s] \mapsto \widetilde{\tube[s]}$ that the neighbors of~$\tube$ in~$\tube' \ssm \tube$ are precisely the neighbors of~$\widetilde\tube$ in~$\widetilde\tube' \ssm \widetilde\tube$. This proves the equality between the compatibility degrees. The equality between the compatibility vectors follows coordinate by coordinate.
\end{proof}

\subsection{Compatibility fan (Theorem~\ref{theo:compatibilityFan})}
\label{subsec:proofCompatibilityFan}

In order to show that the cones of the compatibility matrices of all tubings on~$\graphG$ form a complete simplicial fan, we need the following refinement.

\begin{theorem}
\label{theo:compatibilityFanRefined}
For any graph~$\graphG$ and any maximal tubing~$\tubing^\circ$ on~$\graphG$, the compatibility vectors with respect to~$\tubing^\circ$ have the following properties.
\begin{description}
\item[Span Property] For any tube~$\tube[u]$ of~$\graphG$, the span of~$\set{\compatibilityVector{\tubing^\circ}{\tube[s]}}{\tube[s] \in \tubing, \tube[s] \subseteq \tube[u]}$, for a maximal tubing~$\tubing$ on~$\graphG$ containing~$\tube[u]$, is independent of~$\tubing$. \\[-.3cm]
\item[Flip Property] For any two adjacent maximal tubings~$\tubing, \tubing'$ on~$\graphG$ with~$\tubing \ssm \{\tube\} = \tubing' \ssm \{\tube'\}$, there exists a linear dependence
\[
\alpha \, \compatibilityVector{\tubing^\circ}{\tube} + \alpha' \, \compatibilityVector{\tubing^\circ}{\tube'} + \sum_{\tube[s] \in \tubing \cap \tubing'} \beta_{\tube[s]} \, \compatibilityVector{\tubing^\circ}{\tube[s]} = 0
\]
between the compatibility vectors of~$\tubing \cup \tubing'$ with respect to~$\tubing^\circ$ which is:
\begin{description}
\item[Separating] the hyperplane spanned by~$\set{\compatibilityVector{\tubing^\circ}{\tube[s]}}{\tube[s] \in \tubing \cap \tubing'}$ separates~$\compatibilityVector{\tubing^\circ}{\tube}$ and $\compatibilityVector{\tubing^\circ}{\tube'}$, \ie the coefficients~$\alpha$ and~$\alpha'$ have the same sign different from~$0$.
\item[Local] the dependence is supported by tubes included in~$\tsup$, \ie $\beta_{\tube[s]} = 0$ for all~$\tube[s] \not\subseteq \tsup$.
\end{description}
\end{description}
\end{theorem}

Theorem~\ref{theo:compatibilityFan} follows from the \textbf{Separating Flip Property} and the characterization of complete simplicial fans in Proposition~\ref{prop:characterizationFan}. The \textbf{Span Property} and \textbf{Local Flip Property} are not required to get Theorem~\ref{theo:compatibilityFan} but we use them to obtain the proof of the \textbf{Separating Flip Property}. Observe also that we do not need to prove that the linear dependence between the compatibility vectors of~$\tubing \cup \tubing'$ is unique: it is a consequence of Proposition~\ref{prop:characterizationFan} once we know the \textbf{Separating Flip Property}. 

\medskip
Before entering details, let us sketch the general idea of the proof of Theorem~\ref{theo:compatibilityFanRefined}. We seek for a linear dependence between the compatibility vectors of the tubes of~$\tubing \cup \tubing'$ with respect to the initial tubing~$\tubing^\circ$, that is, for a linear relation satisfied by their compatibility degrees with any tube of~$\tubing^\circ$. There are simple combinatorial relations between the compatibility degrees of the tubes of~$\tubing \cup \tubing'$ with all tubes of~$\tubing^\circ$ not contained in~$\tsup$. Our strategy is to start from such a relation and adapt it iteratively such that it holds for the other tubes of~$\tubing^\circ$ as well. This transformation is done in two steps:
\begin{itemize}
\item We first deal with the tubes of~$\tubing^\circ$ contained in~$\tsup$ and maximal for this property. They determine the coefficients of the linear dependence on the forced tubes
\item For the remaining tubes of~$\tubing^\circ$, we need to make successive corrections to the linear dependence. We first get an explicit linear dependence assuming that~$\tubing \cap \tubing'$ contains certain suitable tubes included in~$\tsup$. We then use inductively the \textbf{Span Property} and the \textbf{Local Flip Property} to get an implicit linear dependence in general.
\end{itemize}
The key of the proof is that our transformation increases the set of tubes of~$\tubing^\circ$ for which the relation between the compatibility degrees of the tubes of~$\tubing \cup \tubing'$ is valid.

\medskip
We now start the formal proof. We proceed by induction on the dimension of the nested complex~$\nestedComplex(\graphG)$. It is immediate when this dimension is~$0$. We now consider an arbitrary graph~$\graphG$ and assume that we have shown Theorem~\ref{theo:compatibilityFanRefined} and thus Corollary~\ref{coro:fullRank} for any graph~$\graphG[H]$ such that~$\dim(\nestedComplex(\graphG[H])) < \dim(\nestedComplex(\graphG))$.
Given a exchangeable pair of tubes~$\tube, \tube'$ of~$\graphG$, our first objective is to exhibit \textbf{separating} and \textbf{local} linear dependences for some adjacent maximal tubings~$\tubing, \tubing'$ on~$\graphG$ such that~$\tubing \ssm \{\tube\} = \tubing' \ssm \{\tube'\}$. We will show later the \textbf{Span Property} and use it to prove that the linear dependence is \textbf{separating} and \textbf{local} for all adjacent maximal tubings~$\tubing, \tubing'$ on~$\graphG$ such that~$\tubing \ssm \{\tube\} = \tubing' \ssm \{\tube'\}$.

\begin{lemma}
\label{lem:flipProperty}
For any two exchangeable tubes~$\tube, \tube'$ of~$\graphG$, there exists adjacent maximal tubings~$\tubing, \tubing'$ on~$\graphG$ such that~$\tubing \ssm \{\tube\} = \tubing' \ssm \{\tube'\}$ and a linear dependence between the compatibility vectors of the tubes of~$\tubing \cup \tubing'$ which is both \textbf{separating} and \textbf{local}.
\end{lemma}

\begin{proof}
We fix some notations for the forced tubes of the exchangeable pair~$\{\tube,\tube'\}$. First recall that since~$\tube$ and~$\tube'$ are exchangeable, the tube~$\tube$ (resp.~$\tube'$) has a single neighbor in~$\tube'\ssm\tube$ (resp. in~$\tube\ssm\tube'$) that we denote by~$r$ (resp.~$r'$). We set~${\tsup \eqdef \tube \cup \tube'}$, and we denote by~$\tinf_1, \dots, \tinf_k$ the connected components of~$\graphG{}[\tube \cap \tube']$, by~$\tube[a]_1, \dots, \tube[a]_\ell$ the connected components of~$\graphG{}[\tube \ssm (\tube' \cup \{r\})]$, and by~$\tube[a]'_1, \dots, \tube[a]'_{\ell'}$ the connected components of~$\graphG{}[\tube' \ssm (\tube \cup \{r'\})]$. Although it is not a tube, we set~$\tinf \eqdef \bigsqcup_{i \in [k]} \tinf_i$ and we abuse notation to write~$\compatibilityDegree{\tube^\circ}{\tinf}$ for~$\sum_{i \in [k]} \compatibilityDegree{\tube^\circ}{\tinf_i}$ and similarly~$\compatibilityVector{\tubing^\circ}{\tinf}$ for~$\sum_{i \in [k]} \compatibilityVector{\tubing^\circ}{\tinf_i}$. We will use in the same way the notations~$\tube[a]$ and~$\tube[a]'$.

We need to distinguish different cases, for which the linear dependences are slightly different, while the proofs are essentially identical. To simplify the discussion, we assume in Cases~(A), (B) and~(C) below that~$\tube,\tube' \notin \tubing^\circ$ and that no tube of~$\tubing^\circ$ is compatible with both~$\tube$ and~$\tube'$. At the end of the proof, Cases~(D) and~(E) show how to restrict to these hypotheses.

\para{(A) A first relation}
Consider a tube~$\tube^\circ$ of~$\tubing^\circ$ not contained in~$\tsup$. We claim that
\[
\compatibilityDegree{\tube^\circ}{\tube} + \compatibilityDegree{\tube^\circ}{\tube'} = \compatibilityDegree{\tube^\circ}{\tsup} + \compatibilityDegree{\tube^\circ}{\tinf}.
\]
Indeed, since~$\tube^\circ \not\subseteq \tsup$, these four compatibility degrees actually count neighbors of~$\tube^\circ$. The formula thus follows from inclusion-exclusion principle since~$\tsup = \tube \cup \tube'$ and~$\tinf = \tube \cap \tube'$. 

There are other tubes of~$\tubing^\circ$ satisfying this relation. Indeed, consider a tube~$\tube^\circ$ in~$\tubing^\circ$ included in~$\tsup$ which contains~$r$ but does not contain nor is adjacent to~$r'$. Then~$\tube^\circ \subsetneq \tube \subsetneq \tsup$ so that~$\compatibilityDegree{\tube^\circ}{\tube} = \compatibilityDegree{\tube^\circ}{\tsup} = 0$. Moreover, $\tube^\circ$ is incompatible with~$\tube'$ and all~$\tinf_1, \dots, \tinf_k$. Since $r'$ is not adjacent to~$\tube^\circ$, all neighbors of~$\tube^\circ$ in~$\tube' \ssm \tube^\circ$ are in~$\tinf \ssm \tube^\circ$. Therefore~$\compatibilityDegree{\tube^\circ}{\tube'} = \compatibilityDegree{\tube^\circ}{\tinf}$. The relation follows. Similarly, the relation follows if~$\tube^\circ$ contains~$r'$ and does not contain nor is adjacent to~$r$.

If all tubes of~$\tubing^\circ$ included in~$\tsup$ satisfy the previous conditions, we have obtained a \textbf{separating} and \textbf{local} linear dependence:
\begin{equation}
\label{eq:linearDependence1}
\compatibilityVector{\tubing^\circ}{\tube} + \compatibilityVector{\tubing^\circ}{\tube'} = \compatibilityVector{\tubing^\circ}{\tinf} + \compatibilityVector{\tubing^\circ}{\tsup}.
\end{equation}
This linear dependence would be valid for any adjacent maximal tubings~$\tubing, \tubing'$ on~$\graphG$ such that ${\tubing \ssm \{\tube\} = \tubing' \ssm \{\tube'\}}$ since the tube~$\tsup$ and the tubes~$\tinf$ are forced in any such pair. Unfortunately, these conditions do not always hold for all tubes of~$\tubing^\circ$. In this case, we will therefore adapt the linear dependence~\eqref{eq:linearDependence1} to cover all possible configurations for the tubes of~$\tubing^\circ$.

\para{(B) No tube of~$\tubing^\circ$ contained in~$\tsup$ contains both~$r$ and~$r'$}
Except if the linear dependence~\eqref{eq:linearDependence1} is valid, there must exist w.l.o.g.~a tube~$\tube^\circ \in \tubing^\circ$ contained in~$\tsup$, containing~$r$ and adjacent to~$r'$. Choose~$\tube^\circ$ maximal for these properties. Since we have assumed that no tube of~$\tubing^\circ$ is compatible with both~$\tube$ and~$\tube'$, all tubes of~$\tubing^\circ$ included in~$\tube^\circ$ contain~$r$. These tubes thus form a nested chain~$\tube^\circ = \tube^\circ_0 \supsetneq  \tube^\circ_1 \supsetneq \dots \supsetneq \tube^\circ_p = \{r\}$. For~$i \in [p]$, define~$\tube^\star_i$ to be the connected component of~$\graphG{}[\tube^\circ_{i-1} \ssm \{r\}]$ containing the singleton~$\tube^\circ_{i-1} \ssm \tube^\circ_{i}$.

Set
\[
\alpha \eqdef \compatibilityDegree{\tube^\circ}{\tube[a]} + \compatibilityDegree{\tube^\circ}{\tube'}
\qquad\text{and}\qquad
\alpha' \eqdef \compatibilityDegree{\tube^\circ}{\tinf} + \compatibilityDegree{\tube^\circ}{\tube[a]},
\]
and define inductively~$\beta_1, \dots, \beta_p$ by
\[
\beta_i = \alpha' \, \compatibilityDegree{\tube^\circ_i}{\tube'} - \alpha \, \compatibilityDegree{\tube^\circ_i}{\tinf} - (\alpha - \alpha')\compatibilityDegree{\tube^\circ_i}{\tube[a]} - \sum_{j \in [i-1]} \beta_j \, \compatibilityDegree{\tube^\circ_i}{\tube^\star_j}.
\]

We claim that
\begin{equation}
\label{eq:linearDependence2}
\alpha \, \compatibilityVector{\tubing^\circ}{\tube} + \alpha' \, \compatibilityVector{\tubing^\circ}{\tube'} = \alpha \, \compatibilityVector{\tubing^\circ}{\tinf} + \alpha' \, \compatibilityVector{\tubing^\circ}{\tsup} + (\alpha - \alpha') \, \compatibilityVector{\tubing^\circ}{\tube[a]} + \sum_{i \in [p]} \beta_i \, \compatibilityVector{\tubing^\circ}{\tube^\star_i}.
\end{equation}
To prove it, we check this linear dependence coordinate by coordinate.

Observe first that~$\compatibilityDegree{\tube^\circ_i}{\tube} = \compatibilityDegree{\tube^\circ_i}{\tsup} = 0$ for all~$0 \le i \le p$ since~$\tube^\circ_i \subsetneq \tube \subseteq \tsup$. Moreover, for all~$i < j$, we have by definition~$\tube^\star_j \subsetneq \tube^\circ_i$, so that~$\compatibilityDegree{\tube^\circ_i}{\tube^\star_j} = 0$. Finally, we have~$\compatibilityDegree{\tube^\circ_i}{\tube^\star_i} = 1$ since $\tube^\circ_i$ and~$\tube^\star_i$ are incompatible and the only neighbor of~$\tube^\circ_i$ in~$\tube^\star_i \ssm \tube^\circ_i$ is the singleton~$\tube^\circ_{i-1} \ssm \tube^\circ_{i}$. Therefore, Relation~\eqref{eq:linearDependence2} holds for~$\tube^\circ$ by definition of~$\alpha$ and~$\alpha'$ and for~$\tube^\circ_i$ by definition of~$\beta_i$.

Consider now a tube~$\tube[s]^\circ$ of~$\tubing^\circ$ not included in~$\tube^\circ$. Suppose that~$\tube[s]^\circ$ is included in~$\tsup$. Then~$\tube[s]^\circ$ contains precisely one of~$r$ and~$r'$ (it cannot contain both by assumption~(B), and it cannot avoid both as it would be compatible with both~$\tube$ and~$\tube'$). If $\tube[s]^\circ$ contains~$r$, it contains~$\tube^\circ$ and therefore equals~$\tube^\circ$ by maximality of the latter. Otherwise, $\tube[s]^\circ$ contains~$r'$, thus is adjacent to~$\tube^\circ$, thus contains it (by compatibility), and thus contains~$r$, a contradiction. We therefore obtained that~$\tube[s]^\circ$ is not included in~$\tsup$, so that all compatibility degrees~$\compatibilityDegree{\tube[s]^\circ}{\tube}$,  $\compatibilityDegree{\tube[s]^\circ}{\tsup}$, $\compatibilityDegree{\tube[s]^\circ}{\tinf}$ and~$\compatibilityDegree{\tube[s]^\circ}{\tube[a]}$ actually count neighbors of~$\tube[s]^\circ$.
Observe now that~$r$ cannot be adjacent to~$\tube[s]^\circ$ (except if it belongs to~$\tube[s]^\circ$), since $r$ belongs to~$\tube^\circ$ which is compatible with~$\tube[s]^\circ$. Therefore, we~have
\[
\compatibilityDegree{\tube[s]^\circ}{\tube} = \compatibilityDegree{\tube[s]^\circ}{\tinf} + \compatibilityDegree{\tube[s]^\circ}{\tube[a]}
\qquad\text{and}\qquad
\compatibilityDegree{\tube[s]^\circ}{\tube'} = \compatibilityDegree{\tube[s]^\circ}{\tsup} - \compatibilityDegree{\tube[s]^\circ}{\tube[a]},
\]
since~$\tube = \tinf \sqcup \tube[a] \sqcup \{r\}$ and~$\tube' = \tsup \ssm \tube[a] \ssm \{r\}$. Finally, since any~$\tube^\star_i$ is contained in~$\tube^\circ$, it is compatible with~$\tube[s]^\circ$, so that~$\compatibilityDegree{\tube[s]^\circ}{\tube^\star_i} = 0$ by Proposition~\ref{prop:compatibilityDegree}. Combining these equalities, we obtain that Relation~\eqref{eq:linearDependence2} holds for any~$\tube[s]^\circ \in \tubing^\circ$.

We found a linear dependence between the compatibility vectors of~${\{\tube, \tube', \tsup\} \cup \tinf \cup \tube[a] \cup \set{\tube^\star_i}{i \in [p]}}$ with respect to the initial tubing~$\tubing^\circ$. Any two of these tubes except~$\tube, \tube'$ are compatible:
\begin{itemize}
\item the forced tubes are pairwise compatible;
\item each~$\tube^\star_i$ is a connected component of~$\tube^\circ_i \ssm \{r\}$, thus is contained in~$\tsup \ssm \{r,r'\}$, and thus is compatible with the connected components of~$\tsup \ssm \{r,r'\}$ (\ie all forced tubes);
\item for~$\tube^\circ_i \supseteq \tube^\circ_j$, any connected component of~$\tube^\circ_j \ssm \{r\}$ is contained in a connected component of~$\tube^\circ_i \ssm \{r\}$ and thus is compatible with all connected components of~$\tube^\circ_i \ssm \{r\}$. In particular~$\tube^\star_i$ and~$\tube^\star_j$ are compatible.
\end{itemize}
Therefore, there exists adjacent maximal tubings~$\tubing, \tubing'$ on~$\graphG$ such that~$\tubing \ssm \{\tube\} = \tubing' \ssm \{\tube'\}$ and $\tubing \cup \tubing' \supseteq \{\tube, \tube', \tsup\} \cup \tinf \cup \tube[a] \cup \set{\tube^\star_i}{i \in [p]}$. For this choice of~$\tubing, \tubing'$, we thus obtained a \textbf{separating} and \textbf{local} linear dependence~\eqref{eq:linearDependence2} between the compatibility vectors of the tubes of~$\tubing \cup \tubing'$.

\para{(C) A tube of~$\tubing^\circ$ contained in~$\tsup$ contains both~$r$ and~$r'$}
We distinguish again two cases.

\para{\quad (C.1) No tube of $\tubing^\circ$ contained in~$\tsup$ contains~$r$ and is adjacent to~$r'$ or conversely}
There must exist a tube~$\tube^\circ \in \tubing^\circ$ included in~$\tube$ and containing~$r$. Choose~$\tube^\circ$ maximal for these properties. With the same arguments as in Case~(B), the tubes of~$\tubing^\circ$ included in~$\tube^\circ$ form a nested chain ${\tube^\circ = \tube^\circ_0 \supsetneq \tube^\circ_1 \supsetneq \dots \supsetneq \tube^\circ_p = \{r\}}$. For~$i \in [p]$, define~$\tube^\star_i$ to be the connected component of~${\graphG{}[\tube^\circ_{i-1} \ssm \{r\}]}$ containing the singleton~$\tube^\circ_{i-1} \ssm \tube^\circ_{i}$. We define the tube~$\tube'^\circ$, the chain~${\tube'^\circ = \tube'^\circ_0 \supsetneq \tube'^\circ_1 \supsetneq \dots \supsetneq \tube'^\circ_{p'} = \{r'\}}$ and the tubes~$\tube'^\star_i$ for~$i \in [p']$ similarly.

Consider now the inclusion minimal tube~$\tsup^\circ$ of~$\tubing^\circ$ contained in~$\tsup$ and containing both~$r$ and~$r'$. Since we assumed that no tube of~$\tubing^\circ$ is compatible with both~$\tube$ and~$\tube'$, we have~$\tsup^\circ = \tube^\circ \sqcup \tube'^\circ \sqcup \{\rsup\}$ where~$\rsup \in \tinf$. Let~$\tsup^\star$ be the connected component of~$\graphG{}[\tsup^\circ \ssm \{r, r'\}]$ containing~$\rsup$.

Set
\[
\alpha \eqdef \compatibilityDegree{\tube'^\circ}{\tinf} + \compatibilityDegree{\tube'^\circ}{\tube[a]'}
\qquad\text{and}\qquad
\alpha' \eqdef \compatibilityDegree{\tube^\circ}{\tinf} + \compatibilityDegree{\tube^\circ}{\tube[a]},
\]
and define inductively~$\beta_1, \dots, \beta_p$ and~$\beta'_1, \dots, \beta'_{p'}$ by
\begin{align*}
\beta_i & = \alpha' \, \compatibilityDegree{\tube^\circ_i}{\tube'} - (\alpha + \alpha') \, \compatibilityDegree{\tube^\circ_i}{\tinf} - \alpha \, \compatibilityDegree{\tube^\circ_i}{\tube[a]} + \alpha\alpha' \, \compatibilityDegree{\tube^\circ_i}{\tsup^\star} - \sum_{j \in [i-1]} \beta_j \, \compatibilityDegree{\tube^\circ_i}{\tube^\star_j}, \\
\beta'_i & = \alpha \, \compatibilityDegree{\tube'^\circ_i}{\tube} - (\alpha + \alpha') \, \compatibilityDegree{\tube'^\circ_i}{\tinf} - \alpha' \, \compatibilityDegree{\tube'^\circ_i}{\tube[a]'} + \alpha\alpha' \, \compatibilityDegree{\tube'^\circ_i}{\tsup^\star} - \sum_{j \in [i-1]} \beta'_j \, \compatibilityDegree{\tube'^\circ_i}{\tube'^\star_j}.
\end{align*}

We claim that
\begin{align}
\label{eq:linearDependence3}
\alpha \, \compatibilityVector{\tubing^\circ}{\tube} + \alpha' \, \compatibilityVector{\tubing^\circ}{\tube'} = & \; (\alpha + \alpha') \, \compatibilityVector{\tubing^\circ}{\tinf} + \alpha \, \compatibilityVector{\tubing^\circ}{\tube[a]} + \alpha' \, \compatibilityVector{\tubing^\circ}{\tube[a]'} - \alpha\alpha' \, \compatibilityVector{\tubing^\circ}{\tsup^\star} \\
\nonumber & + \sum_{i \in [p]} \beta_i \, \compatibilityVector{\tubing^\circ}{\tube^\star_i} + \sum_{i \in [p']} \beta'_i \, \compatibilityVector{\tubing^\circ}{\tube'^\star_i}.
\end{align}
To prove it, we check this linear dependence coordinate by coordinate.

We start with~$\tube^\circ$ and~$\tube'^\circ$. We have~$\compatibilityDegree{\tube^\circ}{\tube} = 0$ since $\tube^\circ \subsetneq \tube$. Moreover, $\compatibilityDegree{\tube^\circ}{\tsup^\star} = |\{\rsup\}| = 1$. Finally, we have by definition~$\tube^\star_j \subsetneq \tube^\circ$ so that~$\compatibilityDegree{\tube^\circ}{\tube^\star_j} = 0$ for all~$j \in [p]$. Combining these equalities, Relation~\eqref{eq:linearDependence3} follows for~$\tube^\circ$ from the definition of~$\alpha$ and~$\alpha'$. The argument is identical~for~$\tube'^\circ$.

We now consider the tubes~$\tube^\circ_i$ and~$\tube'^\circ_i$. Observe first that~$\compatibilityDegree{\tube^\circ_i}{\tube} = 0$ for all~$0 \le i \le p$ since~$\tube^\circ_i \subsetneq \tube$. Moreover, for all~$i < j$, we have by definition~$\tube^\star_j \subsetneq \tube^\circ_i$, so that~$\compatibilityDegree{\tube^\circ_i}{\tube^\star_j} = 0$. In addition, we have~$\compatibilityDegree{\tube^\circ_i}{\tube'^\star_j} = 0$ for all~$i \in [p]$ and~$j \in [p']$ since~$\tube^\circ_i$ and~$\tube'^\star_j$ are compatible. Finally, we have~$\compatibilityDegree{\tube^\circ_i}{\tube^\star_i} = 1$ since $\tube^\circ_i$ and~$\tube^\star_i$ are incompatible and the only neighbor of~$\tube^\circ_i$ in~$\tube^\star_i \ssm \tube^\circ_i$ is the singleton~$\tube^\circ_{i-1} \ssm \tube^\circ_{i}$. Therefore, Relation~\eqref{eq:linearDependence3} holds for~$\tube^\circ_i$ by definition of~$\beta_i$. The argument is identical for~$\tube'^\circ_i$.

Finally, we consider a tube~$\tube[s]^\circ$ of~$\tubing^\circ$ not strictly contained in~$\tsup^\circ$. With similar arguments as in Case~(B), the compatibility degrees~$\compatibilityDegree{\tube[s]^\circ}{\tube}$,  $\compatibilityDegree{\tube[s]^\circ}{\tube'}$, $\compatibilityDegree{\tube[s]^\circ}{\tinf}$, $\compatibilityDegree{\tube[s]^\circ}{\tube[a]}$ and~$\compatibilityDegree{\tube[s]^\circ}{\tube[a]'}$ actually count neighbors of~$\tube[s]^\circ$, and (by assumption~$C$) satisfy
\[
\compatibilityDegree{\tube[s]^\circ}{\tube} = \compatibilityDegree{\tube[s]^\circ}{\tinf} + \compatibilityDegree{\tube[s]^\circ}{\tube[a]}
\qquad\text{and}\qquad
\compatibilityDegree{\tube[s]^\circ}{\tube'} = \compatibilityDegree{\tube[s]^\circ}{\tinf} + \compatibilityDegree{\tube[s]^\circ}{\tube[a]'}.
\]
Moreover, since all~$\tube^\star_i$, $\tube'^\star_i$ and~$\tsup^\star$ are contained in~$\tsup^\circ$, they are all compatible with~$\tube[s]^\circ$, so that $\compatibilityDegree{\tube[s]^\circ}{\tube^\star_i} = \compatibilityDegree{\tube[s]^\circ}{\tube'^\star_i} = \compatibilityDegree{\tube[s]^\circ}{\tsup^\star} = 0$ by Proposition~\ref{prop:compatibilityDegree}. Combining these equalities, we obtain that Relation~\eqref{eq:linearDependence3} holds for any~$\tube[s]^\circ \in \tubing^\circ$.

With the same arguments as in Case~(B), there exists adjacent maximal tubings~$\tubing, \tubing'$ on~$\graphG$ such that~$\tubing \ssm \{\tube\} = \tubing' \ssm \{\tube'\}$ and~$\tubing \cup \tubing' \supseteq \{\tube, \tube', \tsup^\star\} \cup \tinf \cup \tube[a] \cup \tube[a]' \cup \set{\tube^\star_i}{i \in [p]} \cup \set{\tube'^\star_i}{i \in [p]}$. For this choice of~$\tubing, \tubing'$, we thus obtained a \textbf{separating} and \textbf{local} linear dependence~\eqref{eq:linearDependence3} between the compatibility vectors of the tubes of~$\tubing \cup \tubing'$.

\para{\quad (C.2) A tube of $\tubing^\circ$ contained in~$\tsup$ contains~$r$ and is adjacent to~$r'$ or conversely}
W.l.o.g., consider an inclusion maximal tube~$\tube^\circ \in \tubing^\circ$ contained in~$\tsup$, containing~$r$ and adjacent to~$r'$. Since we have assumed that no tube of~$\tubing^\circ$ is compatible with both~$\tube$ and~$\tube'$, all tubes of~$\tubing^\circ$ included in~$\tube^\circ$ contain~$r$. These tubes thus form a nested chain~$\tube^\circ = \tube^\circ_0 \supsetneq  \tube^\circ_1 \supsetneq \dots \supsetneq \tube^\circ_p = \{r\}$. For~$i \in [p]$, define~$\tube^\star_i$ to be the connected component of~$\graphG{}[\tube^\circ_{i-1} \ssm \{r\}]$ containing the singleton~$\tube^\circ_{i-1} \ssm \tube^\circ_{i}$.

Set
\[
\alpha \eqdef \compatibilityDegree{\tube^\circ}{\tube'} - \compatibilityDegree{\tube^\circ}{\tinf} = |\{r'\}| = 1
\qquad\text{and}\qquad
\alpha' \eqdef \compatibilityDegree{\tube^\circ}{\tinf} + \compatibilityDegree{\tube^\circ}{\tube[a]},
\]
and define inductively~$\beta_1, \dots, \beta_p$ by
\[
\beta_i = \alpha' \, \compatibilityDegree{\tube^\circ_i}{\tube'} - (1 + \alpha') \, \compatibilityDegree{\tube^\circ_i}{\tinf} - \compatibilityDegree{\tube^\circ_i}{\tube[a]} - \sum_{j \in [i-1]} \beta_j \, \compatibilityDegree{\tube^\circ_i}{\tube^\star_j}.
\]

We claim that
\begin{equation}
\label{eq:linearDependence4}
\compatibilityVector{\tubing^\circ}{\tube} + \alpha' \, \compatibilityVector{\tubing^\circ}{\tube'} = (1 + \alpha') \, \compatibilityVector{\tubing^\circ}{\tinf} + \compatibilityVector{\tubing^\circ}{\tube[a]} + \alpha' \, \compatibilityVector{\tubing^\circ}{\tube[a]'} + \sum_{i \in [p]} \beta_i \, \compatibilityVector{\tubing^\circ}{\tube^\star_i}.
\end{equation}
To prove it, we check this linear dependence coordinate by coordinate.

Observe first that~$\compatibilityDegree{\tube^\circ_i}{\tube} = 0$ for all~$0 \le i \le p$ since~$\tube^\circ_i \subsetneq \tube$. Moreover, for all~$i < j$, we have by definition~$\tube^\star_j \subsetneq \tube^\circ_i$, so that~$\compatibilityDegree{\tube^\circ_i}{\tube^\star_j} = 0$. Finally, we have~$\compatibilityDegree{\tube^\circ_i}{\tube^\star_i} = 1$ since $\tube^\circ_i$ and~$\tube^\star_i$ are incompatible and the only neighbor of~$\tube^\circ_i$ in~$\tube^\star_i \ssm \tube^\circ_i$ is the singleton~$\tube^\circ_{i-1} \ssm \tube^\circ_{i}$. Therefore, Relation~\eqref{eq:linearDependence2} holds for~$\tube^\circ$ by definition of~$\alpha$ and~$\alpha'$ and for~$\tube^\circ_i$ by definition of~$\beta_i$.

Consider now a tube~$\tube[s]^\circ$ of~$\tubing^\circ$ not strictly contained in~$\tube^\circ$. With similar arguments as in Case~(B), the compatibility degrees~$\compatibilityDegree{\tube[s]^\circ}{\tube}$,  $\compatibilityDegree{\tube[s]^\circ}{\tube'}$, $\compatibilityDegree{\tube[s]^\circ}{\tinf}$, $\compatibilityDegree{\tube[s]^\circ}{\tube[a]}$ and~$\compatibilityDegree{\tube[s]^\circ}{\tube[a]'}$ actually count neighbors of~$\tube[s]^\circ$, and (by assumption~$C$) satisfy
\[
\compatibilityDegree{\tube[s]^\circ}{\tube} = \compatibilityDegree{\tube[s]^\circ}{\tinf} + \compatibilityDegree{\tube[s]^\circ}{\tube[a]}
\qquad\text{and}\qquad
\compatibilityDegree{\tube[s]^\circ}{\tube'} = \compatibilityDegree{\tube[s]^\circ}{\tinf} + \compatibilityDegree{\tube[s]^\circ}{\tube[a]'}.
\]
Moreover, since all~$\tube^\star_i$ are contained in~$\tube^\circ$, they are all compatible with~$\tube[s]^\circ$, so that $\compatibilityDegree{\tube[s]^\circ}{\tube^\star_i} = 0$ by Proposition~\ref{prop:compatibilityDegree}. Combining these equalities, we obtain that Relation~\eqref{eq:linearDependence3} holds for any~$\tube[s]^\circ \in \tubing^\circ$.

With the same arguments as in Case~(B), there exists adjacent maximal tubings~$\tubing, \tubing'$ on~$\graphG$ such that~$\tubing \ssm \{\tube\} = \tubing' \ssm \{\tube'\}$ and~$\tubing \cup \tubing' \supseteq \{\tube, \tube'\} \cup \tinf \cup \tube[a] \cup \tube[a]' \cup \set{\tube^\star_i}{i \in [p]}$. For this choice of~$\tubing, \tubing'$, we thus obtained a \textbf{separating} and \textbf{local} linear dependence~\eqref{eq:linearDependence4} between the compatibility vectors of the tubes of~$\tubing \cup \tubing'$.

\medskip
\enlargethispage{.2cm}
We assumed in Cases~(A), (B) and~(C) above that~$\tube,\tube' \notin \tubing^\circ$ and that no tube of~$\tubing^\circ$ is compatible with both~$\tube$ and~$\tube'$. The remaining two cases show how to force this assumption.

\para{(D) A tube~$\tube^\circ$ of~$\tubing^\circ$ is compatible with both~$\tube$ and~$\tube'$}
We treat this case by induction on the number of tubes of~$\tubing^\circ$ compatible with both~$\tube$ and~$\tube'$. By Lemma~\ref{lem:restrictionCompatibility}, the compatibility vectors with respect to~$\tubing^\circ$ of the tubes of~$\graphG$ compatible with~$\tube^\circ$ correspond to the compatibility vectors of the tubes of~$\widetilde\graphG \eqdef \graphG{}[\tube^\circ] \sqcup \graphG^\star\tube^\circ$ with respect to the maximal tubing~$\widetilde\tubing^\circ$. Since there are strictly less tubes of~$\widetilde\tubing^\circ$ compatible with both~$\widetilde\tube$ and~$\widetilde\tube'$ than tubes of~$\tubing^\circ$ compatible with both~$\tube$ and~$\tube'$, the induction hypothesis ensures that there exists adjacent maximal tubings~$\widetilde\tubing, \widetilde\tubing'$ on~$\widetilde\graphG$ such that~${\widetilde\tubing \ssm \{\widetilde\tube\} = \widetilde\tubing' \ssm \{\widetilde\tube'\}}$ and a \textbf{separating} and \textbf{local} linear dependence between the compatibility vectors of~${\widetilde\tubing \cup \widetilde\tubing'}$ with respect to~$\widetilde\tubing^\circ$. By Lemma~\ref{lem:restriction}, the sets~${\tubing \eqdef \{\tube^\circ\} \cup \bigset{\tube[s]}{\widetilde{\tube[s]} \in \widetilde\tubing}}$ and ${\tubing' \eqdef \{\tube^\circ\} \cup \bigset{\tube[s]'}{\widetilde{\tube[s]}' \in \widetilde\tubing'}}$ are tubings on~$\graphG$ (since~$\tube^\circ$ is compatible with all preimages of tubes of~$\widetilde\graphG$) and they are maximal by cardinality. Moreover, Lemma~\ref{lem:restrictionCompatibility} ensures that the linear dependence between the compatibility vectors of the tubes of~$\tubing \cup \tubing'$ with respect to~$\tubing^\circ$ coincides with the linear dependence between the compatibility vectors of the tubes of~$\widetilde\tubing \cup \widetilde\tubing'$ with respect to~$\widetilde\tubing^\circ$. The linear dependence clearly remains \textbf{separating} and \textbf{local}, which concludes when there is a tube~$\tube^\circ \in \tubing^\circ$ compatible with both~$\tube$ and~$\tube'$.

\para{(E) $\tube$ or~$\tube'$ belongs to~$\tubing^\circ$}
We can assume w.l.o.g.~that~$\tube = \tube^\circ$ belongs to~$\tubing^\circ$. Consider any two adjacent maximal tubings~$\tubing, \tubing'$ on~$\graphG$ such that~${\tubing \ssm \{\tube\} = \tubing' \ssm \{\tube'\}}$. Since any tube~$\tube[s]$ in~$\tubing \cap \tubing'$ is compatible with~$\tube = \tube^\circ$, the $\tube^\circ$-coordinate of the compatibility vector~$\compatibilityVector{\tubing^\circ}{\tube[s]}$ vanishes. The same happens for the vector~${\b{v} \eqdef \compatibilityVector{\tubing^\circ}{\tube} + \compatibilityVector{\tubing^\circ}{\tube'}}$ since $\compatibilityDegree{\tube^\circ}{\tube} = -1$ (as~$\tube = \tube^\circ$) while $\compatibilityDegree{\tube^\circ}{\tube'} = 1$ (by Proposition~\ref{prop:compatibilityDegree}). The set of vectors~$\{\b{v}\} \cup \set{\compatibilityVector{\tubing^\circ}{\tube[s]}}{\tube[s] \in \tubing \cap \tubing'}$ has cardinality~$|\tubing^\circ|$ but is contained in the hyperplane of~$\R^{\tubing^\circ}$ orthogonal to~$\b{e}_{\tube^\circ}$. Therefore, there is a linear dependence between these vectors, which translates into a linear dependence between the compatibility vectors of the tubes of~$\tubing \cup \tubing'$ with the same coefficient on~$\compatibilityVector{\tubing^\circ}{\tube}$ and~$\compatibilityVector{\tubing^\circ}{\tube'}$. This coefficient cannot vanish: otherwise, we would have a linear dependence on the compatibility vectors of~$\widetilde\tubing$ with respect to~$\widetilde\tubing^\circ$. Since~$\widetilde\tubing$ is a maximal tubing on~$\widetilde\graphG$, and~$\dim(\nestedComplex(\widetilde\graphG)) = \dim(\nestedComplex(\graphG)) - 1$, this would contradict Corollary~\ref{coro:fullRank} and thus the induction hypothesis for~$\widetilde\graphG$. This linear dependence is therefore \textbf{separating}. It is automatically \textbf{local} if~$\tsup = \ground$ since all tubes are then subsets of~$\tsup$. Otherwise, we prove that it is \textbf{local} by restriction. We distinguish two cases.

\para{\quad (E.1) A tube~$\tsup^\circ$ of~$\tubing^\circ$ contains~$\tube \cup \{r'\}$ and is contained in~$\tsup$}
Consider the restricted graph~${\hat\graphG = \graphG{}[\tube \cup \{r'\}]}$. Any tube~$\tube[s]$ of~$\graphG$ included in~$\tube$ is also a tube of~$\hat\graphG$. Therefore, the set~$\hat\tubing \eqdef \set{\tube[s]}{\tube[s] \in \tubing, \tube[s] \subseteq \tube}$ is a tubing on~$\hat\graphG$ for any tubing~$\tubing$ containing~$\tube$. Define also the tube~$\hat\tube' \eqdef \tube' \ssm \tube[a]'$ of~$\hat\graphG$. The existence of~$\tsup^\circ$ implies that $\compatibilityDegree{\tube[s]^\circ}{\tube'} = \compatibilityDegree{\tube[s]^\circ}{\tube[a]'}$ for any~$\tube[s]^\circ \in \tubing^\circ$ not included in~$\tube$. It follows that the compatibility vector~$\compatibilityVector{\hat\tubing^\circ}{\hat\tube'}$ is the restriction of~$\compatibilityVector{\tubing^\circ}{\tube'}-\compatibilityVector{\tubing^\circ}{\tube[a]'}$ to the coordinates indexed by~$\hat\tubing^\circ$. Similarly, for any tube~$\tube[s] \in \tubing$ contained in~$\tube$, the compatibility vector~$\compatibilityVector{\hat\tubing^\circ}{\tube[s]}$ is the restriction of~$\compatibilityVector{\tubing^\circ}{\tube[s]}$ to the coordinates indexed by~$\hat\tubing^\circ$. This shows that the linear dependence on~$\bigset{\compatibilityVector{\hat\tubing^\circ}{\tube[s]}}{\tube[s] \in \hat\tubing \cup \hat\tubing'}$ provides a linear dependence on~$\set{\compatibilityVector{\tubing^\circ}{\tube[s]}}{\tube[s] \in \tubing \cup \tubing', \tube[s] \subseteq \tube} \cup \{\compatibilityVector{\tubing^\circ}{\tube'}-\compatibilityVector{\tubing^\circ}{\tube[a]'}\}$. The resulting linear dependence on~$\set{\compatibilityVector{\tubing^\circ}{\tube[s]}}{\tube[s] \in \tubing \cup \tubing'}$ is~\textbf{local}.

\para{\quad (E.2) No tube of~$\tubing^\circ$ contains~$\tube \cup \{r'\}$ and is contained in~$\tsup$}
The proof is identical to Case~(E.1), replacing~$\compatibilityVector{\tubing^\circ}{\tube'}-\compatibilityVector{\tubing^\circ}{\tube[a]'}$ by~$\compatibilityVector{\tubing^\circ}{\tube'}-\compatibilityVector{\tubing^\circ}{\tsup}$.  Details are left to the reader.
\end{proof}

We can now prove the \textbf{Span Property}.

\begin{lemma}
\label{lem:spanProperty}
For any tube~$\tube[u]$ of~$\graphG$, the span of~$\set{\compatibilityVector{\tubing^\circ}{\tube[s]}}{\tube[s] \in \tubing, \tube[s] \subseteq \tube[u]}$, for a maximal tubing~$\tubing$ on~$\graphG$ containing~$\tube[u]$, is independent of~$\tubing$.
\end{lemma}

\begin{proof}
We proceed by induction on the size of~$\tube[u]$. The result is immediate if~$\tube[u]$ is a singleton. Consider now two adjacent maximal tubings~$\tubing, \tubing'$ on~$\graphG$ containing~$\tube[u]$ such that~$\tubing \ssm \{\tube\} = \tubing' \ssm \{\tube'\}$. Assume first that~$\tube$ and~$\tube'$ are contained in~$\tube[u]$. By Lemma~\ref{lem:flipProperty}, there exists adjacent maximal tubings~$\tubing[S], \tubing[S]'$ on~$\graphG$ containing~$\tube[u]$ such that~$\tubing[S] \ssm \{\tube\} = \tubing[S]' \ssm \{\tube'\}$ and a linear dependence between the compatibility vectors of the tubes of~$\tubing[S] \cup \tubing[S]'$ with respect to~$\tubing^\circ$ which is both \textbf{separating} and \textbf{local}. By definition, this implies that there exists~$\alpha > 0$ and~$\alpha' > 0$ such that the vector~${\alpha \, \compatibilityVector{\tubing^\circ}{\tube} + \alpha' \, \compatibilityVector{\tubing^\circ}{\tube'}}$ belongs to~$\vect(\set{\compatibilityVector{\tubing^\circ}{\tube[s]}}{\tube[s] \in \tubing[S] \cap \tubing[S]', \tube[s] \subseteq \tsup})$. However,
\[
\set{\tube[s] \in \tubing[S] \cap \tubing[S]'}{\tube[s] \subseteq \tsup} = \{\tsup\} \, \cup \, \set{\tube[s] \in \tubing[S] \cap \tubing[S]'}{\tube[s] \subseteq \tinf} \, \cup \, \set{\tube[s] \in \tubing[S] \cap \tubing[S]'}{\tube[s] \subseteq \tube[a]} \, \cup \, \set{\tube[s] \in \tubing[S] \cap \tubing[S]'}{\tube[s] \subseteq \tube[a]'}.
\]
By induction hypothesis applied to each tube of~$\tinf$, we have
\[
\begin{array}{@{}c@{\;}c@{\;}l@{\quad}r}
\vect(\set{\compatibilityVector{\tubing^\circ}{\tube[s]}}{\tube[s] \in \tubing[S] \cap \tubing[S]', \tube[s] \subseteq \tinf}) & = &\vect(\set{\compatibilityVector{\tubing^\circ}{\tube[s]}}{\tube[s] \in \tubing[S], \tube[s] \subseteq \tinf}) & \text{(as~$\tubing[S] \ssm (\tubing[S] \cap \tubing[S]') = \tube \not\subseteq \tinf$)} \\
& = & \vect(\set{\compatibilityVector{\tubing^\circ}{\tube[s]}}{\tube[s] \in \tubing, \tube[s] \subseteq \tinf}) & \text{(induction hypothesis)} \\
& = & \vect(\set{\compatibilityVector{\tubing^\circ}{\tube[s]}}{\tube[s] \in \tubing \cap \tubing', \tube[s] \subseteq \tinf}) & \text{(as~$\tubing \ssm (\tubing \cap \tubing') = \tube \not\subseteq \tinf$)}
\end{array}
\]
and similarly replacing~$\tinf$ by~$\tube[a]$ or~$\tube[a]'$. It follows that the vector~${\alpha \, \compatibilityVector{\tubing^\circ}{\tube} + \alpha' \, \compatibilityVector{\tubing^\circ}{\tube'}}$ also belongs to~$\vect(\set{\compatibilityVector{\tubing^\circ}{\tube[s]}}{\tube[s] \in \tubing \cap \tubing', \tube[s] \subseteq \tube[u]})$. Since~$\alpha \ne 0$, this implies that~$\compatibilityVector{\tubing^\circ}{\tube}$ belongs to~$\vect(\set{\compatibilityVector{\tubing^\circ}{\tube[s]}}{\tube[s] \in \tubing', \tube[s] \subseteq \tube[u]})$. Similarly, $\compatibilityVector{\tubing^\circ}{\tube'}$ belongs to~$\vect(\set{\compatibilityVector{\tubing^\circ}{\tube[s]}}{\tube[s] \in \tubing, \tube[s] \subseteq \tube[u]})$. We therefore obtained that
\[
\vect(\set{\compatibilityVector{\tubing^\circ}{\tube[s]}}{\tube[s] \in \tubing, \tube[s] \subseteq \tube[u]}) = \vect(\set{\compatibilityVector{\tubing^\circ}{\tube[s]}}{\tube[s] \in \tubing', \tube[s] \subseteq \tube[u]}).
\]
This also clearly holds when~$\tube$ and~$\tube'$ are not contained in~$\tube[u]$. This concludes the proof since the graph of flips on the maximal tubings on~$\graphG$ containing~$\tube[u]$ is connected.
\end{proof}

We can finally conclude the proof of Theorem~\ref{theo:compatibilityFanRefined} using both Lemmas~\ref{lem:flipProperty} and~\ref{lem:spanProperty}.

\begin{proof}[Proof of Theorem~\ref{theo:compatibilityFanRefined}]
The \textbf{Span Property} is proved in Lemma~\ref{lem:spanProperty}. It only remains to show the \textbf{Flip Property} for arbitrary adjacent maximal tubings. Consider two adjacent maximal tubings~$\tubing, \tubing'$ on~$\graphG$ with~$\tubing \ssm \{\tube\} = \tubing' \ssm \{\tube'\}$. By Lemma~\ref{lem:flipProperty}, there exists adjacent maximal tubings~$\tubing[S], \tubing[S]'$ on~$\graphG$ such that~$\tubing[S] \ssm \{\tube\} = \tubing[S]' \ssm \{\tube'\}$ and a linear dependence between the compatibility vectors of the tubes of~$\tubing[S] \cup \tubing[S]'$ with respect to~$\tubing^\circ$ which is both \textbf{separating} and \textbf{local}. By definition, this implies that there exists~$\alpha > 0$ and~$\alpha' > 0$ such that the vector~${\alpha \, \compatibilityVector{\tubing^\circ}{\tube} + \alpha' \, \compatibilityVector{\tubing^\circ}{\tube'}}$ belongs to~$\vect(\set{\compatibilityVector{\tubing^\circ}{\tube[s]}}{\tube[s] \in \tubing[S] \cap \tubing[S]', \tube[s] \subseteq \tsup})$. Lemma~\ref{lem:spanProperty} applied to~$\tinf$ ensures that
\[
\vect(\set{\compatibilityVector{\tubing^\circ}{\tube[s]}}{\tube[s] \in \tubing \cap \tubing', \tube[s] \subseteq \tinf}) = \vect(\set{\compatibilityVector{\tubing^\circ}{\tube[s]}}{\tube[s] \in \tubing[S] \cap \tubing[S]', \tube[s] \subseteq \tinf}).
\]
and similarly replacing~$\tinf$ by~$\tube[a]$ or~$\tube[a]'$. We thus conclude that
\[
\vect(\set{\compatibilityVector{\tubing^\circ}{\tube[s]}}{\tube[s] \in \tubing \cap \tubing', \tube[s] \subseteq \tsup}) = \vect(\set{\compatibilityVector{\tubing^\circ}{\tube[s]}}{\tube[s] \in \tubing[S] \cap \tubing[S]', \tube[s] \subseteq \tsup})
\]
contains the vector~${\alpha \, \compatibilityVector{\tubing^\circ}{\tube} + \alpha' \, \compatibilityVector{\tubing^\circ}{\tube'}}$ with~$\alpha > 0$ and~$\alpha' > 0$. In other words, we obtained a \textbf{separating} and \textbf{local} linear dependence on~$\set{\compatibilityVector{\tubing^\circ}{\tube[s]}}{\tube[s] \in \tubing \cup \tubing'}$.
\end{proof}

\subsection{Dual compatibility fan (Theorem~\ref{theo:dualCompatibilityFan})}
\label{subsec:proofDualCompatibilityFan}

The fact that dual compatibility vectors support a complete simplicial fan realizing the nested complex is a direct consequence of Theorem~\ref{theo:compatibilityFan}, using the following duality trick. Observe first that given $n+1$ vectors~$\{u, v, w_1, \dots, w_{n-1}\}$ in~$\R^n$, the hyperplane spanned by~$\{w_1, \dots, w_{n-1}\}$ separates~$u$ and~$v$ if and only if
\[
\det([u | w_1 | \dots | w_{n-1}]) \cdot \det([v | w_1 | \dots | w_{n-1}]) < 0.
\]
Theorem~\ref{theo:compatibilityFan} shows this condition for primal compatibility matrices and we need to show it for dual compatibility matrices. For this, we notice that the primal and dual compatibility matrices are related by
\begin{equation}
\label{eq:dualityTrick}
\dualCompatibilityMatrix{\tubing}{\tubing^\circ} \eqdef [\compatibilityDegree{\tube_j}{\tube_i^\circ}]_{i, j \in [n]} = \transpose{[\compatibilityDegree{\tube_i}{\tube_j^\circ}]_{i,j \in [n]}} \defeq \transpose{\compatibilityMatrix{\tubing}{\tubing^\circ}},
\end{equation}
where~$\transpose{M}$ denotes the transpose of the matrix~$M$.
Consider now two pairs of adjacent maximal tubings~$\tubing^\circ, {\tubing^\circ}'$ and~$\tubing, \tubing'$ on~$\graphG$. From the \textbf{Separation Flip Property} of Theorem~\ref{theo:compatibilityFan} applied to the initial maximal tubing~$\tubing$, we obtain that ${\det(\compatibilityMatrix{\tubing}{\tubing^\circ}) \cdot \det(\compatibilityMatrix{\tubing}{{\tubing^\circ}'}) < 0}$. Similarly, for the initial maximal tubing~$\tubing'$, we obtain that ${\det(\compatibilityMatrix{\tubing'}{\tubing^\circ}) \cdot \det(\compatibilityMatrix{\tubing'}{{\tubing^\circ}'}) < 0}$. Multiplying these two inequalities, we get
\[
\bigg( \det(\compatibilityMatrix{\tubing}{\tubing^\circ}) \cdot \det(\compatibilityMatrix{\tubing'}{\tubing^\circ}) \bigg) \cdot \bigg( \det(\compatibilityMatrix{\tubing}{{\tubing^\circ}'}) \cdot \det(\compatibilityMatrix{\tubing'}{{\tubing^\circ}'}) \bigg) > 0.
\]
Since the transposition preserves the determinant, we obtain by Equation~\ref{eq:dualityTrick} that
\[
\bigg( \det(\dualCompatibilityMatrix{\tubing}{\tubing^\circ}) \cdot \det(\dualCompatibilityMatrix{\tubing'}{\tubing^\circ}) \bigg) \cdot \bigg( \det(\dualCompatibilityMatrix{\tubing}{{\tubing^\circ}'}) \cdot \det(\dualCompatibilityMatrix{\tubing'}{{\tubing^\circ}'}) \bigg) > 0.
\]
This implies that the products~$\det(\dualCompatibilityMatrix{\tubing}{\tubing^\circ}) \cdot \det(\dualCompatibilityMatrix{\tubing'}{\tubing^\circ})$ and~$\det(\dualCompatibilityMatrix{\tubing}{{\tubing^\circ}'}) \cdot \det(\dualCompatibilityMatrix{\tubing'}{{\tubing^\circ}'})$ have the same sign. Since we know that~$\det(\dualCompatibilityMatrix{\tubing}{{\tubing^\circ}}) \cdot \det(\dualCompatibilityMatrix{\tubing'}{{\tubing^\circ}}) < 0$ for the tubing~$\tubing^\circ = \tubing$, we obtain by repeated flips in~$\tubing^\circ$ that~$\det(\dualCompatibilityMatrix{\tubing}{{\tubing^\circ}}) \cdot \det(\dualCompatibilityMatrix{\tubing'}{{\tubing^\circ}}) < 0$ for any initial tubing~$\tubing^\circ$ and pair of adjacent maximal tubings~$\tubing, \tubing'$ on~$\graphG$. This shows the \textbf{Separation Flip Property} for dual compatibility vectors. We conclude again by Proposition~\ref{prop:characterizationFan}. \qed

\begin{remark}
Observe that this proof does not provide us with any explicit linear dependence between dual compatibility vectors. In particular, we do not know whether the following analogue properties of Theorem~\ref{theo:compatibilityFanRefined} hold:
\begin{description}
\item[Dual Span Property] For any tube~$\tube[u]$ of~$\graphG$, the span of~$\set{\dualCompatibilityVector{\tube}{\tubing^\circ}}{\tube \in \tubing, \tube \not\subset \tube[u]}$, for a maximal tubing~$\tubing$ on~$\graphG{}$ containing~$\tube[u]$, is independent of~$\tubing$. \\[-.3cm]
\item[Dual Local Flip Property] For any two maximal tubings~$\tubing, \tubing'$ with~${\tubing \ssm \{\tube\} = \tubing' \ssm \{\tube'\}}$, the unique linear dependence between the dual compatibility vectors of~$\tubing \cup \tubing'$ with respect to~$\tubing^\circ$ is supported by tubes not strictly included in a connected component of~$\tinf = \tube \cap \tube'$.
\end{description}
\end{remark}

\subsection{Nested complex isomorphisms (Proposition~\ref{prop:exmAutomorphism}, Proposition~\ref{prop:nestedComplexIsomorphismDisconnected} and Theorem~\ref{theo:nestedComplexIsomorphisms})}
\label{subsec:proofIsomorphisms}

We now prove various results on nested complex isomorphisms presented in Section~\ref{subsec:many}. We first show that the map~$\Omega$ on the tubes of a spider~$\spiderG_{\ninf}$ defines a nested complex automorphism that dualizes the compatibility degree. 

\begin{proof}[Proof of Proposition~\ref{prop:exmAutomorphism}]
First, $\Omega$ clearly sends tubes of~$\spiderG_{\ninf}$ to tubes of~$\spiderG_{\ninf}$. We just have to show that~$\compatibilityDegree{\Omega(\tube)}{\Omega(\tube')} = \compatibilityDegree{\tube'}{\tube}$ for any two tubes~$\tube, \tube'$ of~$\spiderG_{\ninf}$, and Proposition~\ref{prop:compatibilityDegree} will imply that~$\Omega$ is a nested complex automorphism. This follows from the definition of~$\Omega$ and the fact that
\begin{gather*}
\biggCompatibilityDegree{\big[ v^i_j, v^i_k \big]}{\big[ v^{i'}_{j'}, v^{i'}_{k'} \big]} = \delta_{i = i'} \cdot \big( \delta_{j < j' \le k+1 < k'+1} + \delta_{j' < j \le k'+1 < k+1} \big), \\
\biggCompatibilityDegree{\big[ v^i_j, v^i_k \big]}{\bigcup_{h \in \ell} \big[ v^h_0, v^h_{k_h} \big]} = \delta_{j \le k_i + 1 \le k}, \\
\text{and} \qquad \biggCompatibilityDegree{\bigcup_{i \in \ell} \big[ v^i_0, v^i_{k_i} \big]}{\bigcup_{i \in \ell} \big[ v^i_0, v^i_{k'_i} \big]} = |\set{i \in [\ell]}{k_i < k'_i}| \cdot \delta_{\exists i \in [\ell], \, k_i > k'_i}. \qedhere
\end{gather*}
\end{proof}

Our objective is to show that these maps~$\Omega$ on spiders are essentially the only non-trivial nested complex isomorphisms. We fix an isomorphism~$\Phi$ between two nested complexes~$\nestedComplex(\graphG)$ and~$\nestedComplex(\graphG')$. We first show that~$\Phi$ preserves connected components in the following sense.

\begin{lemma}
\label{lem:isomorphismsPreserveConnectedComponents}
Two tubes~$\tube$ and~$\tube'$ of~$\graphG$ belong to the same connected component of~$\graphG$ if and only if their images~$\Phi(\tube)$ and~$\Phi(\tube')$ belong to the same connected component of~$\graphG'$.
\end{lemma}

\begin{proof}
Observe first that two tubes from distinct connected components are automatically compatible. Assume now that~$\tube$ and~$\tube'$ are in the same connected component of~$\graphG$. If~$\tube$ and~$\tube'$ are incompatible, then~$\Phi(\tube)$ and~$\Phi(\tube')$ are also incompatible and therefore in the same connected component of~$\graphG'$. If~$\tube$ and~$\tube'$ are compatible, then there exists a tube~$\tube''$ incompatible with both~$\tube$~and~$\tube'$: 
\begin{itemize}
\item if~$\tube \cap \tube' = \varnothing$, consider a path connecting a neighbor of~$\tube$ to a neighbor of~$\tube'$ in~$\graphG{}[\ground \ssm (\tube \cup \tube')]$;
\item if~$\tube \subseteq \tube'$, consider a path connecting a neighbor of~$\tube$ to a neighbor of~$\tube'$ in~$\graphG{}[\ground \ssm \tube]$.
\end{itemize}
We obtain that~$\Phi(\tube'')$ is incompatible with both~$\Phi(\tube)$ and~$\Phi(\tube')$, so that they all belong to the same connected component of~$\graphG'$. This proves one direction. For the other direction, consider~$\Phi^{-1}$.
\end{proof}

Consequently, the nested complex~$\nestedComplex(\graphG)$ records the sizes of the connected components of the graph~$\graphG$. We call \defn{connected size partition} of~$\graphG$ the partition~$\lambda(\graphG) \eqdef |\ground_1|, |\ground_2|, \dots, |\ground_{\connectedComponents}|$ of~$|\ground|$, where~$\ground_1, \dots, \ground_\connectedComponents$ are the connected components of~$\graphG$ ordered such that~$|\ground_i| \ge |\ground_{i+1}|$.

\begin{corollary}
\label{coro:isomorphismsPreserveConnectedSizePartition}
Two graphs whose nested complexes are isomorphic have the same connected size partitions: $\nestedComplex(\graphG) \simeq \nestedComplex(\graphG') \Longrightarrow \lambda(\graphG) = \lambda(\graphG')$. In particular, they have the same number of vertices.
\end{corollary}

\begin{proof}
Consider a maximal tubing~$\tubing$ on~$\graphG$ and decompose it into subtubings~$\tubing_1, \dots, \tubing_\connectedComponents$ on the connected components~$\ground_1, \dots, \ground_\connectedComponents$ of~$\graphG$. Their images~$\Phi(\tubing_1), \dots, \Phi(\tubing_\connectedComponents)$ decompose the maximal tubing~$\Phi(\tubing)$. Moreover, Lemma~\ref{lem:isomorphismsPreserveConnectedComponents} ensures that two tubes~$\Phi(\tube) \in \Phi(\tubing_i)$ and~$\Phi(\tube') \in \Phi(\tubing_{i'})$ belong  to the same connected component of~$\graphG'$ if and only if~$i = i'$. We therefore obtain that ${\lambda(\graphG) = \{|\ground_1|, \dots, |\ground_{\connectedComponents}|\} = \{|\tubing_1| + 1, \dots, |\tubing_{\connectedComponents}| + 1\} = \{|\Phi(\tubing_1)| + 1, \dots, |\Phi(\tubing_{\connectedComponents})| + 1\} = \lambda(\graphG')}$.
\end{proof}

Proposition~\ref{prop:nestedComplexIsomorphismDisconnected} is another immediate consequence of Lemma~\ref{lem:isomorphismsPreserveConnectedComponents}: since it sends all tubes in a connected component~$\graphG[H]$ of~$\graphG$ to tubes in the same connected component~$\graphG[H]'$ of~$\graphG'$, the map~$\Phi$ induces a nested complex isomorphism between~$\nestedComplex(\graphG[H])$ and~$\nestedComplex(\graphG[H]')$. From now on, we assume without loss of generality that~$\graphG$ is connected. Our next step is a crucial structural property of~$\Phi$.

\begin{lemma}
For any nested complex isomorphism~$\Phi:\nestedComplex(\graphG) \to \nestedComplex(\graphG')$ and any tube~$\tube$ of~$\graphG$, either~$|\Phi(\tube)| = |\tube|$ or~$|\Phi(\tube)| = |\ground| - |\tube|$.
\end{lemma}

\begin{proof}
By Lemma~\ref{lem:restriction}, the join of a tube~$\tube$ in~$\nestedComplex(\graphG)$ is isomorphic to the nested complex~${\nestedComplex(\graphG{}[\tube] \sqcup \graphG^\star\tube)}$ of the union of the restriction~$\graphG{}[\tube]$ with the reconnected complement~$\graphG^\star\tube$. The former has~$|\tube|$ vertices while the latter has~$|\ground| - |\tube|$ vertices. Since~$\Phi$ induces an isomorphism from the join of~$\tube$ in~$\nestedComplex(\graphG)$ to the join of~$\Phi(\tube)$ in~$\nestedComplex(\graphG')$, the result follows from Corollary~\ref{coro:isomorphismsPreserveConnectedSizePartition}.
\end{proof}

We say that~$\Phi$ \defn{maintains} the tube~$\tube$ if~$|\Phi(\tube)| = |\tube|$ and that~$\Phi$ \defn{swaps} the tube~$\tube$ if~${|\Phi(\tube)| = |\ground| - |\tube|}$. 

\begin{proposition}
\label{prop:nonReversing}
If it maintains all tubes of~$\graphG$, then~$\Phi$ is the trivial nested complex isomorphism induced by the graph isomorphism~$\psi : \graphG \to \graphG'$ defined by~$\Phi(\{v\}) = \{\psi(v)\}$.
\end{proposition}

\begin{proof}
Two vertices~$v$ and~$w$ of~$\graphG$ are adjacent if and only if the two tubes~$\{v\}$ and~$\{w\}$ are incompatible. Since~$\Phi$ preserves the compatibility relation, this shows that~$v$ and~$w$ are adjacent if and only if~$\psi(v)$ and~$\psi(w)$ are, \ie that~$\psi$ defines a graph isomorphism. Let~$\Psi$ denote the nested complex isomorphism induced by~$\psi$, \ie defined by~$\Psi(\tube) \eqdef \set{\psi(v)}{v \in \tube}$.

We prove by induction on~$|\tube|$ that~$\Phi(\tube) = \Psi(\tube)$ for any tube~$\tube$ on~$\graphG$. It holds for singletons. For the induction step, consider an arbitrary tube~$\tube$ of~$\graphG$. Let~$v \in \ground \ssm \tube$ be a neighbor of~$\tube$. Since~$\{v\}$ and~$\tube$ are incompatible, so are~$\Phi(\{v\}) = \{\psi(v)\}$ and~$\Phi(\tube)$, and thus~$\psi(v)$ is a neighbor of~$\Phi(\tube)$. Let~$w \in \ground$ be such that~$\psi(w)$ is a neighbor of~$\psi(v)$ in~$\Phi(\tube)$. If~$w \notin \tube$ then it is incompatible with~$\tube \cup \{v\}$, and thus~$\Phi(\{w\}) = \{\psi(w)\}$ is incompatible with~$\Phi(\tube \cup \{v\})$. Therefore, $\Phi(\tube \cup \{v\})$ is adjacent to and does not contain~$\psi(w)$ which is in~$\Phi(\tube)$. Since~${|\Phi(\tube \cup \{v\})| = |\tube| + 1 = |\Phi(\tube)| + 1}$, this implies that~$\Phi(\tube \cup \{v\})$ is incompatible with~$\Phi(\tube)$, a contradiction. Therefore, we know that~$w \in \tube$. Let~$\tube_1, \dots, \tube_k$ denote the connected components of~$\graphG{}[\tube \ssm \{w\}]$. By induction hypothesis, ${\Phi(\tube_i) = \Psi(\tube_i)}$ for all~${i \in [k]}$. Moreover, since~$\Phi(\tube_i)$ is compatible with~$\Phi(\tube)$ and adjacent to~${\psi(w) \in \Phi(\tube)}$, it is included in~$\Phi(\tube)$. We thus obtain~${\Psi(\tube) = \{\psi(w)\} \cup \Psi(\tube_1) \cup \dots \cup \Psi(\tube_k) \subseteq \Phi(\tube)}$ and thus~$\Phi(\tube) = \Psi(\tube)$ since~${|\Phi(\tube)| = |\tube| = |\Psi(\tube)|}$.
\end{proof}

\begin{lemma}
\label{lem:structureMaintenedSwappedTubes}
If~$\Phi$ does not maintain a tube~$\tube$ of~$\graphG$ (\ie~$|\Phi(\tube)| \ne |\tube|$), then
\begin{enumerate}[(i)]
\item $\Phi$ swaps any tube of~$\graphG$ containing~$\tube$,
\item $\Phi$ maintains any tube of~$\graphG$ disjoint from and non-adjacent to~$\tube$, and
\item $\Phi$ swaps at least one singleton included in~$\tube$.
\end{enumerate}
\end{lemma}

\begin{proof}
Consider a tube~$\tube[s]$ of~$\graphG$ strictly containing~$\tube$. The link of~$\{\tube[s], \tube\}$ in~$\nestedComplex(\graphG)$ is isomorphic to the nested complex of the union of the graphs~$\graphG{}[\tube]$, $(\graphG^\star\tube)[\tube[s] \ssm \tube]$, and~$\graphG^\star\tube[s]$ with~$|\tube|$, $|\tube[s]| - |\tube|$ and~${|\ground| - |\tube[s]|}$ vertices respectively. Therefore, Corollary~\ref{coro:isomorphismsPreserveConnectedSizePartition} ensures that the link of~$\{\Phi(\tube[s]), \Phi(\tube)\}$ in~$\nestedComplex(\graphG')$ is isomorphic to the nested complex of a graph with three connected components with~$|\tube|$, $|\tube[s]| - |\tube|$ and~$|\ground| - |\tube[s]|$ vertices respectively.
If~$\Phi(\tube[s])$ is not contained in~$\Phi(\tube)$, then the link of~$\{\Phi(\tube[s]), \Phi(\tube)\}$ in~$\nestedComplex(\graphG')$ would be isomorphic to the nested complex of a graph with one connected component~$\graphG'[\Phi(\tube)]$ having~$|\Phi(\tube)|$ vertices. We reach a contradiction as~$|\Phi(\tube)| = |\ground| - |\tube|$ is neither~$|\tube|$ (by assumption on~$\tube$), nor~$|\ground| - |\tube[s]|$ (since~$|\tube[s]| > |\tube|$), nor~$|\tube[s]| - |\tube|$ (since~$|\tube[s]| < |\ground|$).
Therefore, $\Phi(\tube[s])$ is contained in~$\Phi(\tube)$ and the link of~$\{\Phi(\tube[s]), \Phi(\tube)\}$ in~$\nestedComplex(\graphG')$ is isomorphic to the union of the graphs~$\graphG{}[\Phi(\tube[s])]$, $(\graphG^\star\Phi(\tube[s]))[\Phi(\tube) \ssm \Phi(\tube[s])]$, and~$\graphG^\star\Phi(\tube)$ with~$|\Phi(\tube[s])|$, $|\Phi(\tube)| - |\Phi(\tube[s])|$ and~${|\ground| - |\Phi(\tube)|}$ vertices respectively. If~$|\Phi(\tube[s])| \ne |\ground| - |\tube[s]|$, then it forces~$|\Phi(\tube[s])| = |\tube[s]| - |\tube| = |\tube[s]|$, a contradiction.
This proves~(i).

Consider now a tube~$\tube[s]$ of~$\graphG$ disjoint from and non-adjacent to~$\tube$. Note that~$|\tube[s]| + |\tube| < |\ground|$ as there is at least a vertex separing them. The link of~$\{\tube[s], \tube\}$ in~$\nestedComplex(\graphG)$ is isomorphic to the nested complex of the union of the graphs~$\graphG{}[\tube[s]]$, $\graphG{}[\tube]$, and~$(\graphG^\star\tube[s])^\star\tube$ with~$|\tube[s]|$, $|\tube|$ and~${|\ground| - |\tube[s]| - |\tube|}$ vertices respectively. Again, Corollary~\ref{coro:isomorphismsPreserveConnectedSizePartition} ensures that the link of~$\{\Phi(\tube[s]), \Phi(\tube)\}$ in~$\nestedComplex(\graphG')$ is isomorphic to the nested complex of a graph with three connected components with~$|\tube[s]|$, $|\tube|$ and~${|\ground| - |\tube[s]| - |\tube|}$ vertices~respectively.
If~$\Phi(\tube[s])$ is not contained in~$\Phi(\tube)$, then the link of~$\{\Phi(\tube[s]), \Phi(\tube)\}$ in~$\nestedComplex(\graphG')$ would be isomorphic to the nested complex of a graph with one connected component~$\graphG'[\Phi(\tube)]$ having~$|\Phi(\tube)|$ vertices. We reach a contradiction as~$|\Phi(\tube)| = |\ground| - |\tube|$ is neither~$|\tube|$ (by assumption on~$\tube$), nor~$|\tube[s]|$ (since~$|\tube[s]| + |\tube| < |\ground|$), nor~$|\ground| - |\tube[s]| - |\tube|$ (since~$|\tube[s]| > 0$).
Therefore, $\Phi(\tube[s])$ is contained in~$\Phi(\tube)$ and the link of~$\{\Phi(\tube[s]), \Phi(\tube)\}$ in~$\nestedComplex(\graphG')$ is isomorphic to the union of the graphs~$\graphG{}[\Phi(\tube[s])]$, $(\graphG^\star\Phi(\tube[s]))[\Phi(\tube) \ssm \Phi(\tube[s])]$, and~$\graphG^\star\Phi(\tube)$ with~$|\Phi(\tube[s])|$, $|\Phi(\tube)| - |\Phi(\tube[s])|$ and~${|\ground| - |\Phi(\tube)|}$ vertices respectively. If~$|\Phi(\tube[s])| \ne |\tube[s]|$, then it forces~$|\Phi(\tube[s])| = |\ground| - |\tube[s]| - |\tube| = |\ground| - |\tube[s]|$, a contradiction.
This proves~(ii).

Finally, to prove~(iii) we can assume that~$\tube$ is not a singleton. Thus~$\Phi(\tube)$ is not an inclusion maximal tube. Let~$\tube'$ be a tube of~$\graphG$ such that~$\Phi(\tube')$ is a maximal tube of~$\graphG'$ containing~$\Phi(\tube)$. Since~$\Phi^{-1}$ swaps~$\Phi(\tube)$ and~$\Phi(\tube')$ contains~$\Phi(\tube)$, $\Phi^{-1}$ also swaps~$\Phi(\tube')$ by~(i). Thus, $\tube'$ is a singleton swapped by~$\Phi$ and contained in~$\tube$.
\end{proof}

\begin{lemma}
\label{lem:structureMaintenedSwappedVertices}
Denote by~$M \eqdef \set{v \in V}{|\Phi(\{v\})| = 1}$ the set of vertices maintained by~$\Phi$ and by~$S \eqdef \set{v \in V}{|\Phi(\{v\}|) = |\ground| - 1}$  the set of vertices swapped by~$\Phi$. Then
\begin{enumerate}[(i)]
\item $S$ forms a clique of~$\graphG$,
\item any vertex in~$M$ has at most one neighbor in~$S$, and
\item any vertex in~$S$ has at most one neighbor in~$M$.
\end{enumerate}
\end{lemma}

\begin{proof}
Let~$s,s' \in S$. Since~$|\Phi(\{s\})| = |\Phi(\{s'\})| = |\ground| - 1$, the tubes~$\Phi(\{s\})$ and~$\Phi(\{s'\})$ are incompatible. Therefore~$\{s\}$ and~$\{s'\}$ are incompatible, so that~$s$ and~$s'$ are neighbors. The set~$S$ thus forms a clique.

To prove~(ii) and~(iii), assume that some vertices~$m \in M$ and~$s \in S$ are neighbors. The tubes~$\{m\}$ and~$\{s\}$ are thus incompatible, so that~$\Phi(\{m\})$ and~$\Phi(\{s\})$ are also incompatible. Since ${|\Phi(\{m\})| = 1}$ while~$|\Phi(\{s\})| = |\ground| - 1$, this implies~$\Phi(\{s\}) = \ground' \ssm \Phi(\{m\})$. It follows that~$m$ cannot have another neighbor swapped by~$\Phi$ and~$s$ cannot have another neighbor maintained~by~$\Phi$.
\end{proof}

We are now ready to prove that any non-trivial nested complex isomorphism coincides, up to composition with a trivial nested complex isomorphism, with the isomorphism~$\Omega$ on a spider.

\begin{proof}[Proof of Theorem~\ref{theo:nestedComplexIsomorphisms}]
The proof works by induction on the number~$|\ground|$ of vertices of~$\graphG$. It is clear when~$|\ground| \le 2$. For the induction step, assume that the result holds for all graphs on less than~$|\ground|$ vertices and consider a non-trivial nested complex isomorphism~$\Phi : \nestedComplex(\graphG) \to \nestedComplex(\graphG')$. Then~$\Phi$ does not maintain all tubes of~$\graphG$ by Proposition~\ref{prop:nonReversing}, and thus swaps at least one singleton~$\{s\}$ by Lemma~\ref{lem:structureMaintenedSwappedTubes}\,(iii). Let~$s'$ denote the vertex of~$\graphG'$ such that~$\Phi(\{s\}) = \ground' \ssm \{s'\}$. 

The map~$\Phi$ induces a nested complex isomorphism between the link of~$\{s\}$ in~$\nestedComplex(\graphG)$ and the link of~$\Phi(\{s\})$ in~$\nestedComplex(\graphG')$. The former is isomorphic to the nested complex of the reconnected complement~$\widetilde\graphG \eqdef \graphG^\star\{s\}$ while the latter is isomorphic to the nested complex of the restriction~$\widetilde\graphG' \eqdef \graphG'[\Phi(\{s\})]$. Let~$\widetilde\Phi : \widetilde\tube \mapsto \Phi(\tube)$ denote the resulting nested complex isomorphism between~$\nestedComplex(\widetilde\graphG)$ and~$\nestedComplex(\widetilde\graphG')$. This isomorphism~$\widetilde\Phi$ is non-trivial: by Lemma~\ref{lem:structureMaintenedSwappedTubes}, $\Phi$ swaps any tube~$\tube$ containing~$\{s\}$, so that $\widetilde\Phi$ swaps the tube~${\tube \ssm \{s\}}$. It follows by induction hypothesis that~$\widetilde\graphG$ and~$\widetilde\graphG'$ are spiders and that there exists a graph isomorphism~${\widetilde\psi : \widetilde\graphG \to \widetilde\graphG'}$ inducing a trivial nested complex isomorphism~$\widetilde\Psi : \nestedComplex(\widetilde\graphG) \to \nestedComplex(\widetilde\graphG')$ such that~$\widetilde\Psi^{-1} \circ \widetilde\Phi \defeq \widetilde\Omega$ is the automorphism of~$\nestedComplex(\widetilde\graphG)$ described in Section~\ref{subsec:many}. In other words, we can label by~$\widetilde v^i_j$ the vertices of the spider~$\widetilde\graphG$ and by~$\widetilde v'^i_j$ the vertices of the spider~$\widetilde\graphG'$, with~$i \in [\,\widetilde\ell\,]$ and~$0 \le j \le \widetilde n_i$, such that
\begin{equation}
\label{eq:barPhi}
\widetilde\Phi\big( \big[ \widetilde v^i_j, \widetilde v^i_k \big] \big) = \big[ \widetilde v'^i_{\widetilde n_i+1-k}, \widetilde v'^i_{\widetilde n_i+1-j} \big]
\quad\text{and}\quad
\widetilde\Phi\big( \bigcup_{i \in [\,\widetilde\ell\,]} \big[ \widetilde v^i_0, \widetilde v^i_{k_i} \big] \big) = \bigcup_{i \in [\,\widetilde\ell\,]} \big[ \widetilde v'^i_0, \widetilde v'^i_{\widetilde n_i-1-k_i} \big].
\end{equation}

We now claim that~$\graphG$ and~$\graphG'$ are both spiders. To prove it, we distinguish two cases:
\begin{description}
\item[Body case] all neighbors of~$s$ in~$\graphG$ are swapped by~$\Phi$. Then they form a clique in~$\graphG$ (by Lemma~\ref{lem:structureMaintenedSwappedVertices}\,(i)), so that~$\graphG$ is the spider~$\widetilde\graphG$ where we add one more body vertex~$s$ with no attached leg. Moreover,~$s'$ is necessarily swapped by~$\Phi^{-1}$ (otherwise~$\Phi^{-1}(\{s'\})$ would be a neighbor of~$s$ maintained by~$\Phi$). We conclude by symmetry that~$\graphG'$ is the spider~$\widetilde\graphG'$ where we add one more body vertex~$s'$ with no attached leg.
\item[Leg case] $s$ has a neighbor~$m$ maintained by~$\Phi$. It is unique by Lemma~\ref{lem:structureMaintenedSwappedVertices}\,(iii) and not connected to any other vertex swapped by~$\Phi$ by Lemma~\ref{lem:structureMaintenedSwappedVertices}\,(ii). Therefore,~$\graphG$ is the spider~$\widetilde\graphG$ where we replace the edges connecting~$m$ to all other body vertices of~$\widetilde\graphG$ by a new body vertex~$s$ with an edge to~$m$. Moreover,~$s'$ is necessarily maintained by~$\Phi^{-1}$ (otherwise~$\Phi^{-1}(\{s'\})$ should be~$\ground \ssm \{s\}$ which is not connected). We conclude that~$\graphG'$ is the spider~$\widetilde\graphG'$ where we add one additional leg vertex~$s'$ to the free endpoint of a leg.
\end{description}

We now label by~$v^i_j$ the vertices of~$\graphG$ according to the labels~$\widetilde v^i_j$ of~$\widetilde\graphG$ and by~$v'^i_j$ the vertices of~$\graphG'$ according to the labels~$\widetilde v'^i_j$ of~$\widetilde\graphG'$. We follow the two cases above:
\begin{description}
\item[Body case] We set~$\ell \eqdef \widetilde\ell + 1$, $n_i \eqdef \widetilde n_i$ for~$i \in [\,\widetilde\ell\,]$ and~$n_\ell = 0$. For any~$i \in [\,\widetilde\ell\,]$ and~$0 \le j \le n_i$, we label by~$v^i_j$ the vertex of~$\graphG$ corresponding to the vertex labeled by~$\widetilde v^i_j$ in~$\widetilde\graphG$, and similarly we label by~$v'^i_j$ the vertex of~$\graphG'$ corresponding to the vertex labeled by~$\widetilde v'^i_j$ in~$\widetilde\graphG'$. Finally, we label~$s$ by~$v^\ell_0$ and~$s'$ by~$v'^\ell_0$.
\item[Leg case] Assume that the neighbor of~$s$ maintained by~$\Phi$ corresponds to the vertex labeled by~$\widetilde v^a_0$ in~$\widetilde\graphG$. Then the neighbor of~$s'$ corresponds to the vertex labeled by~$\widetilde v'^a_{\widetilde n_a}$ in~$\widetilde\graphG'$. We set~$\ell \eqdef \widetilde\ell$, $n_i \eqdef \widetilde n_i$ for~$i \in [\ell] \ssm \{a\}$ and~$n_a \eqdef \widetilde n_a + 1$. For any~$i \in [\ell]$ and~$0 \le j \le \widetilde n_i$, we label by~$v^i_j$ if~$i \ne a$ and~$v^i_{j+1}$ if~$i = a$ the vertex of~$\graphG$ corresponding to the vertex labeled by~$\widetilde v^i_j$ in~$\widetilde\graphG$ and by~$v'^i_j$ the vertex of~$\graphG'$ corresponding to the vertex labeled by~$\widetilde v'^i_j$ in~$\widetilde\graphG'$. Finally, we label~$s$ by~$v^a_0$ and~$s'$ by~$v'^a_{n_a}$.
\end{description}
By our previous description of the graphs~$\graphG$ and~$\graphG'$, these labelings are indeed valid labelings of spiders, meaning that the edges of~$\graphG$ are indeed given by~$\bigset{\big\{v^i_{j-1}, v^i_j\big\}}{i \in [\ell], j \in [n_i]} \cup \bigset{\big\{v^i_0, v^{i'}_0\big\}}{i \ne i' \in [\ell]}$, and similarly for~$\graphG'$. We moreover claim that~$\Phi$ is given by
\begin{equation}
\label{eq:Phi}
\Phi\big( \big[ v^i_j, v^i_k \big] \big) = \big[ v^i_{n_i+1-k}, v^i_{n_i+1-j} \big]
\qquad\text{and}\qquad
\Phi\big( \bigcup_{i \in [\ell]} \big[ v^i_0, v^i_{k_i} \big] \big) = \bigcup_{i \in [\ell]} \big[ v^i_0, v^i_{n_i-1-k_i} \big].
\end{equation}
It is immediate for all tubes compatible with~$\{s\} = \{v^a_0\}$ as it is easily transported from~\eqref{eq:barPhi}. Therefore, we only have to check it for the tubes of~$\graphG$ adjacent to~$s = v^a_0$ and not containing~it. Observe first that~$\Phi\big( \big[ v^a_1, v^a_k \big] \big)$ is a tube with~$k$ vertices (by Lemma~\ref{lem:structureMaintenedSwappedTubes}\,(iii)), and that it contains~$s' = v'^a_{n_a}$ since it has to be incompatible with~$\Phi(\{s\}) = \ground' \ssm \{s'\}$. Therefore, $\Phi\big( \big[ v^a_1, v^a_k \big] \big) = \big[ v^a_{n_a+1-k}, v^a_{n_a} \big]$. Consider now a tube~$\tube = \bigcup_{i \in [\ell]} \big[ v^i_0, v^i_{k_i} \big]$ not containing~$s = v^a_0$ (\ie with~$k_a = -1$). Since the nested tubes~$\tube$ and~$\tube \cup \{s\}$ are both swapped, we have~$\Phi(\tube) = \Phi(\tube \cup \{s\}) \cup \{s'\}$. Since~$\Phi(\tube \cup \{s\})$ is given by Equation~\eqref{eq:Phi}, so is~$\Phi(\tube)$. This concludes the proof that~$\Phi$ is given by Equation~\eqref{eq:Phi}, so that it coincides with~$\Omega$ up to the graph automorphism defined by~$v^i_j \mapsto v'^i_j$.
\end{proof}

We now prove that if the primal and dual compatibility fans of~$\graphG$ with respect to the same initial maximal tubing~$\tubing^\circ$ are linearly isomorphic, then~$\graphG$ is an octopus whose head is contained in no tube of~$\tubing^\circ$.

\begin{proof}[Proof of Lemma~\ref{lem:comparisonPrimalDual}]
Consider a graph~$\graphG$ and an initial tubing~$\tubing^\circ$ such that the fans~$\compatibilityFan{\graphG}{\tubing^\circ}$ and~$\dualCompatibilityFan{\graphG}{\tubing^\circ}$ are linearly isomorphic. The fans~$\compatibilityFan{\graphG}{\tubing^\circ}$ and~$\dualCompatibilityFan{\graphG}{\tubing^\circ}$ both contains precisely $n$ pairs of opposite rays, given by the vectors~$\b{e}_i$ of the canonical basis and their opposites~$-\b{e}_i$. Therefore, the fans~$\compatibilityFan{\graphG}{\tubing^\circ}$ and~$\dualCompatibilityFan{\graphG}{\tubing^\circ}$ have the same rays, which implies that the compatibility vector~$\compatibilityVector{\tubing^\circ}{\tube}$ and dual compatibility vector~$\dualCompatibilityVector{\tube}{\tubing^\circ}$ are collinear for any tube~$\tube$ of~$\graphG$. In other words, we have~$\compatibilityDegree{\tube^\circ_1}{\tube} \compatibilityDegree{\tube}{\tube^\circ_2} = \compatibilityDegree{\tube}{\tube^\circ_1} \compatibilityDegree{\tube^\circ_2}{\tube}$ for all tubes~$\tube$ of~$\graphG$ and~$\tube^\circ_1, \tube^\circ_2 \in \tubing^\circ$.

We now prove by induction that this condition implies that~$\graphG$ is an octopus whose head is contained in no tube of~$\tubing^\circ$. The result is clear when~$|\ground| \le 3$. Consider thus a connected graph~$\graphG$ on more than~$4$ vertices and a maximal tubing~$\tubing^\circ$ on~$\graphG$ with root~$u$ (\ie $u$ is the only vertex of~$\ground$ contained in no proper tube of~$\tubing^\circ$). The graph~$\graphG{}[\ground\ssm\{u\}]$ has connected components~$\graphG_1,\dots,\graphG_k$ and~$\tubing^\circ$ induces a maximal tubing~$\tubing^\circ_i$ on each component~$\graphG_i$. We know that~$\graphG_1$ is an octopus whose head~$v$ is contained in no tube of~$\tubing^\circ$. Otherwise, by induction hypothesis, we could find three tubes~$\tube, \tube^\circ_1, \tube^\circ_2$ of~$\graphG_1$ such that~$\compatibilityDegree{\tube^\circ_1}{\tube} \compatibilityDegree{\tube}{\tube^\circ_2} \ne \compatibilityDegree{\tube}{\tube^\circ_1} \compatibilityDegree{\tube^\circ_2}{\tube}$, which would contradict our assumption on~$\tubing^\circ$ since~$\tube, \tube^\circ_1, \tube^\circ_2$ are also tubes of~$\graphG$. Assume now that~$\graphG_1$ is not a path. We distinguish five~cases:
\begin{itemize}
\item Suppose that~$v$ is the unique vertex of~$\graphG_1$ adjacent to~$u$. Let~${\tube = \{u,v\}}$, ${\tube^\circ_1 = \graphG_1}$ and $\tube^\circ_2$ be a leg of~$\graphG_1$. Then~$\compatibilityDegree{\tube}{\tube_1^\circ} \geq 2$ and~$\compatibilityDegree{\tube_1^\circ}{\tube} = \compatibilityDegree{\tube}{\tube_2^\circ} = \compatibilityDegree{\tube_2^\circ}{\tube} = 1$. 
\item Suppose that~$u$ is adjacent to~$v$ and at least another vertex~$w$ of~$\graphG_1$. Let~${\tube = \{u\}}$, ${\tube^\circ_1 = \graphG_1}$ and $\tube^\circ_2$ be the leg of~$\graphG_1$ containing~$w$. Then~$\compatibilityDegree{\tube_1^\circ}{\tube} = \compatibilityDegree{\tube_2^\circ}{\tube} = 1$ and~${\compatibilityDegree{\tube}{\tube_2^\circ} < \compatibilityDegree{\tube}{\tube_1^\circ}}$. 
\item Suppose that~$u$ is adjacent to at least two legs of~$\graphG_1$. Let~${\tube = \{u\}}$, ${\tube^\circ_1 = \graphG_1}$ and $\tube^\circ_2$ be a leg of~$\graphG_1$ adjacent to~$u$. Then~${\compatibilityDegree{\tube_1^\circ}{\tube} = \compatibilityDegree{\tube_2^\circ}{\tube} = 1}$ and~${\compatibilityDegree{\tube}{\tube_2^\circ} < \compatibilityDegree{\tube}{\tube_1^\circ}}$.
\item Suppose that~$u$ is adjacent to a single leg~$\graphG[H]$ of~$\graphG_1 \ssm \{v\}$ but not to~$v$. Let~$\tube = \{u,v\} \cup \graphG[H]$, $\tube^\circ_1 = \graphG_1$ and $\tube^\circ_2$ be a leg of~$\graphG_1$ distinct from~$\graphG[H]$. Then~$\compatibilityDegree{\tube}{\tube_1^\circ} \geq 2$ and~$\compatibilityDegree{\tube_1^\circ}{\tube} = \compatibilityDegree{\tube}{\tube_2^\circ} = \compatibilityDegree{\tube_2^\circ}{\tube} = 1$.
\end{itemize}
In all cases, we have~$\compatibilityDegree{\tube^\circ_1}{\tube} \compatibilityDegree{\tube}{\tube^\circ_2} \ne \compatibilityDegree{\tube}{\tube^\circ_1} \compatibilityDegree{\tube^\circ_2}{\tube}$, contradicting our assumption on~$\tubing^\circ$. Thus,~$\graphG_1$ is a path. A similar case analysis shows that~$\graphG_1$ is attached to~$u$ only by one of its endpoints. By symmetry, all components~$\graphG_1, \dots, \graphG_k$ of~$\graphG \ssm \{u\}$ are paths attached to~$u$ only by an endpoint, so that~$\graphG$ is an octopus whose head is contained in no tube~of~$\tubing^\circ$.
\end{proof}

\subsection{Polytopality of compatibility fans (Theorem~\ref{theo:polytopalityPathsCycles} and Proposition~\ref{prop:polytopeStar})}
\label{subsec:proofPolytopality}

This section provides the proof of the polytopality results presented in Section~\ref{subsec:polytopality}. Using a similar method as~\cite[Section~5]{CeballosSantosZiegler} based on Proposition~\ref{prop:polytopalityFan}, we first prove that all compatibility and dual compatibility fans of paths and cycles are polytopal.

\begin{proof}[Proof of Theorem~\ref{theo:polytopalityPathsCycles}]
We use the characterization of polytopality of complete simplicial fans given in Proposition~\ref{prop:polytopalityFan}. For this, we need to understand better the linear dependences on compatibility vectors for paths and cycles.

\medskip
Consider first the case of the path. When~$(\tubing \cup \tubing') \cap \tubing^\circ = \varnothing$, the linear dependences can only be of the form
\begin{align*}
\compatibilityVector{\tubing^\circ}{\tube} + \compatibilityVector{\tubing^\circ}{\tube'} & = \compatibilityVector{\tubing^\circ}{\tsup} + \compatibilityVector{\tubing^\circ}{\tinf},
\\
\compatibilityVector{\tubing^\circ}{\tube} + \compatibilityVector{\tubing^\circ}{\tube'} & = \compatibilityVector{\tubing^\circ}{\tube[a]} + \compatibilityVector{\tubing^\circ}{\tube[a]'},
\\
2 \, \compatibilityVector{\tubing^\circ}{\tube} + \compatibilityVector{\tubing^\circ}{\tube'} & = \compatibilityVector{\tubing^\circ}{\tsup} + \compatibilityVector{\tubing^\circ}{\tube[a]},
\\
2 \, \compatibilityVector{\tubing^\circ}{\tube} + \compatibilityVector{\tubing^\circ}{\tube'} & = \compatibilityVector{\tubing^\circ}{\tinf} + \compatibilityVector{\tubing^\circ}{\tube[a]'},
\end{align*}
up to exchanging simultaneously~$\tube$ with~$\tube'$ and~$\tube[a]$ with~$\tube[a]'$. If~$\tubing \cap \tubing'$ contains a tube~$\tube^\circ \in \tubing^\circ$, then the compatibility degree of all tubes of~$(\tubing \cup \tubing') \ssm \{\tube^\circ\}$ with~$\tube^\circ$ vanishes, so that the tube~$\tube^\circ$ cannot appear in the linear dependence. When~$\tube, \tube' \notin \tubing^\circ$ but~$(\tubing \cap \tubing') \cap \tubing^\circ \ne \varnothing$, the relations are thus obtained from the ones above by deleting terms in their right hand sides. The dependences when~$\tube$ or~$\tube'$ belong to~$\tubing^\circ$ will be treated separately.

We now define a height function~$\omega$ on tubes on~$\pathG_{n+1}$ by
\[
\omega(\tube) = 
\begin{cases}
f(|\tube|) & \text{if } \tube \notin \tubing^\circ, \\
\Omega & \text{otherwise,}
\end{cases}
\]
where~$f: \R \to \R_{>0}$ is any strictly concave increasing positive function and~$\Omega \in \R$ is a large enough constant. When~$(\tubing \cup \tubing') \cap \tubing^\circ = \varnothing$, we obtain by definition of~$\tsup$, $\tinf$, $\tube[a]$ and~$\tube[a]'$, and using that~$f$ is concave and increasing, that
\begin{align*}
\omega(\tube) + \omega(\tube') & > \omega(\tsup) + \omega(\tinf),
\\
\omega(\tube) + \omega(\tube') & > \omega(\tube[a]) + \omega(\tube[a]'),
\\
2 \, \omega(\tube) + \omega(\tube') & > \omega(\tsup) + \omega(\tube[a]),
\\
2 \, \omega(\tube) + \omega(\tube') & > \omega(\tinf) + \omega(\tube[a]').
\end{align*}
Moreover, the inequalities still hold when we delete terms in their right hand sides since~$\omega$ is positive. Therefore, $\omega$ satisfies Condition~(2) of Proposition~\ref{prop:polytopalityFan} when we do not flip an initial tube. Finally, initial tubes only appear in the left hand sides of linear dependences, so choosing~$\omega(\tube^\circ) = \Omega$ large enough ensures that $\omega$ satisfies Condition~(2) of Proposition~\ref{prop:polytopalityFan} for any flip. Observe that this is essentially the same proof as in~\cite[Section~5]{CeballosSantosZiegler}.

\medskip
We now adapt this proof for the cycle~$\cycleG_{n+1}$. Clearly, the dependences described above for the path also appear for the cycle (as cycles contain paths). Beside those, when~$(\tubing \cup \tubing') \cap \tubing^\circ = \varnothing$, a straightforward case analysis shows that the linear dependences can only be of the form
\[
\compatibilityVector{\tubing^\circ}{\tube} + \compatibilityVector{\tubing^\circ}{\tube'} = 2 \, \compatibilityVector{\tubing^\circ}{\tinf_1}, \qquad\text{where } \tinf_1 \in \tinf.
\]
Again, no tube of~$\tubing^\circ$ can appear in their right hand sides of the linear dependences. Therefore, when~$\tube, \tube' \notin \tubing^\circ$ but $(\tubing \cap \tubing') \cap \tubing^\circ \ne \varnothing$, the linear dependences are obtained from the generic ones above by deleting terms in their right hand sides. The dependences when~$\tube$ or~$\tube'$ belong to~$\tubing^\circ$ will again be treated separately.

We choose the same height function~$\omega$ as before. For the same reasons, the linear dependences for the path are again transformed to strict inequalities on~$\omega$. Moreover, as~$\tube_1 \subseteq \tube \cap \tube'$ and~$f$ is increasing, we have
\[
\omega(\tube) + \omega(\tube') > 2 \, \omega(\tinf_1).
\]
We conclude as before by choosing~$\Omega$ large enough that~$\omega$ satisfies the Condition~(2) of Proposition~\ref{prop:polytopalityFan} for any flip.

\medskip
Finally, for dual compatibility vectors, a straightforward case analysis shows that the linear dependences are all of the form
\begin{align*}
\dualCompatibilityVector{\tube}{\tubing^\circ} + \dualCompatibilityVector{\tube'}{\tubing^\circ} & = \dualCompatibilityVector{\tsup}{\tubing^\circ} + \dualCompatibilityVector{\tinf}{\tubing^\circ},
\\
\dualCompatibilityVector{\tube}{\tubing^\circ} + \dualCompatibilityVector{\tube'}{\tubing^\circ} & = \dualCompatibilityVector{\tube[a]}{\tubing^\circ} + \dualCompatibilityVector{\tube[a]'}{\tubing^\circ},
\\
2 \, \dualCompatibilityVector{\tube}{\tubing^\circ} + \dualCompatibilityVector{\tube'}{\tubing^\circ} & = \dualCompatibilityVector{\tsup}{\tubing^\circ} + \dualCompatibilityVector{\tube[a]}{\tubing^\circ},
\\
2 \, \dualCompatibilityVector{\tube}{\tubing^\circ} + \dualCompatibilityVector{\tube'}{\tubing^\circ} & = \dualCompatibilityVector{\tinf}{\tubing^\circ} + \dualCompatibilityVector{\tube[a]'}{\tubing^\circ},
\end{align*}
when none of~$\tube$, $\tube'$ and~$\tsup$ have~$n$ vertices. We can also have the linear dependences
\begin{align*}
\dualCompatibilityVector{\tube}{\tubing^\circ} + \dualCompatibilityVector{\tube'}{\tubing^\circ} & = 2 \, \dualCompatibilityVector{\tsup}{\tubing^\circ} + \dualCompatibilityVector{\tinf}{\tubing^\circ},
\\
2 \, \dualCompatibilityVector{\tube}{\tubing^\circ} + \dualCompatibilityVector{\tube'}{\tubing^\circ} & = 2 \, \dualCompatibilityVector{\tsup}{\tubing^\circ} + \dualCompatibilityVector{\tube[a]}{\tubing^\circ},
\end{align*}
when~$|\tsup| = n$ and
\[
\dualCompatibilityVector{\tube}{\tubing^\circ} + \dualCompatibilityVector{\tube'}{\tubing^\circ} = \dualCompatibilityVector{\tinf_1}{\tubing^\circ}, \qquad\text{where } \tinf_1 \in \tinf.
\]
when~$|\tube| = |\tube'| = n$.
Again, no tube of~$\tubing^\circ$ can appear in the right hand sides of the linear dependences. Therefore, when~$\tube, \tube' \notin \tubing^\circ$ but $(\tubing \cap \tubing') \cap \tubing^\circ \ne \varnothing$, the linear dependences are obtained from the generic ones above by deleting terms in their right hand sides. 

We now define a height function~$\omega$ on tubes on~$\cycleG_{n+1}$ by
\[
\omega(\tube) = 
\begin{cases}
f(|\tube|) & \text{if } \tube \notin \tubing^\circ \text{ and } |\tube| \ne n, \\
\displaystyle{\frac{f(|\tube|)}{2}} & \text{if } \tube \notin \tubing^\circ \text{ and } |\tube| = n , \\
\Omega & \text{otherwise,}
\end{cases}
\]
where~$f: \R \to \R_{>0}$ is any strictly concave increasing positive function and~$\Omega \in \R$ is a large enough constant. By definition of~$\tsup$, $\tinf$, $\tube[a]$ and~$\tube[a]'$, and using that~$f$ is concave and increasing, we obtain that~$\omega$ satisfies a strict inequality for each linear dependence above. We conclude as before by choosing~$\Omega$ large enough that~$\omega$ satisfies the Condition~(2) of Proposition~\ref{prop:polytopalityFan} for any flip.
\end{proof}

Our last proof concerns the polytopality of the compatibility fan for the star, for which we have presented a candidate in Section~\ref{subsec:polytopality}. 

\begin{proof}[Proof of Proposition~\ref{prop:polytopeStar}]
We just have to show that for any tube~$\tube$ and any maximal tubing~$\tubing$ on~$\starG_{n+1}$, the point~$\b{x}(\tubing)$ belongs to the half-space~$\HS(\tube)$ and to the boundary of this half-space if and only if~$\tube \in \tubing$. 

Consider first a tube~$\tube$ not in~$\tubing^\circ$. Let~$\tinf$ denote the inclusion minimal tube of~$\tubing \cup \ground$ containing the central vertex~$*$. Then the other tubes of~$\tubing$ are all leaves of~$\starG_{n+1}$ contained in~$\tinf$ and a nested chain of tubes~$\tinf = \tube_{|\tinf|} \subsetneq \tube_{|\tinf|+1} \subsetneq \dots \subsetneq \tube_{n+1} = \ground$ of~$\starG_{n+1}$. Therefore, we have~$\b{x}(\tubing)_i = 0$ if $\{\ell_i\} \subseteq \tinf$ and~$\b{x}(\tubing)_i = j-1$ if~$\{\ell_i\} = \tube_j \ssm \tube_{j-1}$. We conclude that
\[
\dotprod{\compatibilityVector{\tubing^\circ}{\tube}}{\b{x}(\tubing)} = \sum_{\substack{i \in [n] \\ \ell_i \in \tube}} \b{x}(\tube)_i = \sum_{\substack{|\tinf| \le j \le n+1 \\ \tube_j \ssm \tube_{j-1} \not\subseteq \tube}} (j-1) \le \sum_{j = |\tube|}^{n} j = f(|\tube|),
\]
with equality if and only if~$\tube_j \ssm \tube_{j-1} \not\subseteq \tube$ for all~$|\tube| \le j \le n+1$, \ie if and only if~$\tube = \tube_{|\tube|}$.

Finally, for any~$i \in [n]$, we have~$\dotprod{\compatibilityVector{\tubing^\circ}{\{\ell_i\}}}{\b{x}(\tubing)} = -\b{x}(\tubing)_i \le 0$, with equality if and only if the inclusion minimal tube of~$\tubing \cup \ground$ containing~$i$ is~$\{\ell_i\}$, \ie if and only if~$\{\ell_i\} \in \tubing$.
\end{proof}

\subsection{Design compatibility fan (Theorem~\ref{theo:compatibilityFanDesign})}
\label{subsec:proofDesignCompatibilityFan}

The proof of Theorem~\ref{theo:compatibilityFanDesign} still relies on Proposition~\ref{prop:characterizationFan}, that is on the understanding of the linear dependences of the compatibility vectors of the tubes involved in a flip. We now need to distinguish two kinds of flips: we call \defn{round flips} those exchanging two round tubes, and \defn{square flips} those exchanging a square to a round tube.

We claim that Theorem~\ref{theo:compatibilityFanRefined} still holds for round flips. Indeed, a coordinatewise verification shows that the linear dependences exhibited in Section~\ref{subsec:proofCompatibilityFan} still hold for a initial maximal design tubing~$\tubing^\circ$: the arguments are identical for coordinates corresponding to round tubes and straightforward for coordinates corresponding to square tubes.

It thus remains to show the \textbf{Separating Flip Property} for square flips. It turns out that the proof for square flips is much easier as the linear dependences only involve compatibility vectors of forced tubes. Using the duality trick presented in Section~\ref{subsec:proofDualCompatibilityFan} and the fact that any two maximal design tubings are connected by a sequence of square flips, it is equivalent to prove the \textbf{Separating Flip Property} for the compatibility fan or for the dual compatibility fan. In the sequel, we prefer to work with the dual compatibility vectors.

Fix an initial maximal design tubing~$\tubing^\circ$ on a graph~$\graphG$. Consider a round tube~$\tube$ exchangeable with a square tube~$\squareTube{v}$, that is the vertex~$v$ is contained in~$\tube$. The forced tubes of this square flip (\ie the tubes contained in all flips exchanging~$\tube$ and~$\squareTube{v}$) are the following:
\begin{itemize}
\item the square tubes~$\squareTube{w_1}, \dots \squareTube{w_k}$ for the neighbors~$w_1, \dots, w_k$ of~$\tube$ in~$\graphG$,
\item the round tubes~$\tube[a]_1,\dots,\tube[a]_{\ell}$ given by the connected components of~$\graphG{}[\tube\ssm \{v\}]$.
\end{itemize} 
Let~$p$ be number of~$w_i$'s which are roots of initial round tubes in~$\tubing^\circ$ and let~$q$ be the number of~$\tube[a]_i$'s containing an initial square tube of~$\tubing^\circ$. Suppose that 
\begin{itemize}
\item the~$w_i$'s are ordered such that~$w_1, \dots, w_p$ are roots of tubes~$\tube_1^\circ, \dots, \tube_p^\circ$ in~$\tubing^\circ$ with~$\tube_i^\circ \not\supseteq \tube_j^\circ$ for~$1 \le i < j \le p$, while~$\squareTube{w_{p+1}}, \dots, \squareTube{w_k}$ are square tubes of~$\tubing^\circ$,
\item the~$\tube[a]_i$'s are ordered such that~$\tube[a]_1,\dots,\tube[a]_q$ contain an initial square tube of~$\tubing^\circ$ while~$\tube[a]_{q+1},\dots,\tube[a]_{\ell}$ do not. 
\end{itemize}
With these notations, the reader can check that if~$p = 0$ and~$q > 0$, then the dual compatibility vectors of these forced tubes satisfy the dependence
\begin{equation}
\label{eq:designTrivialDependence}
\dualCompatibilityVector{\tube}{\tubing^\circ} = \sum_{j \in [q]} \big(\dualCompatibilityVector{\tube[a]_j}{\tubing^\circ} - \dualCompatibilityVector{\squareTube{v}}{\tubing^\circ}\big)
\end{equation}
We now adapt this linear dependence to cover the case when~$p \ge 0$ and~$q > 0$. Namely, we claim that
\[
\dualCompatibilityVector{\tube}{\tubing^\circ} = \sum_{j \in [q]} \big(\dualCompatibilityVector{\tube[a]_j}{\tubing^\circ} - \dualCompatibilityVector{\squareTube{v}}{\tubing^\circ}\big) + \sum_{i \in [p]}\beta_i \, \dualCompatibilityVector{\squareTube{w_i}}{\tubing^\circ}
\]
where the coefficients~$\beta_1, \dots, \beta_p$ are recursively defined by
\begin{equation}
\label{eq:designFullDependence}
\beta_i = \compatibilityDegree{\tube}{\tube_i^\circ} - \bigg(\sum_{j \in [q]} \big(\compatibilityDegree{\tube[a]_{j}}{\tube_i^\circ} -\compatibilityDegree{\squareTube{v}}{\tube_i^\circ}\big) + \sum_{r \in [i-1]} \beta_r \, \compatibilityDegree{\squareTube{w_r}}{\tube_i^\circ} \bigg).
\end{equation}
To prove it, we check this linear dependence coordinate by coordinate. It boils down to Equation~\eqref{eq:designTrivialDependence} for initial tubes contained in~$\tube$. It is also clear for the initial square tubes not contained in~$\tube$ as all dual compatibility degrees involved in~\eqref{eq:designFullDependence} vanish. Moreover, the coefficients~$\beta_1, \dots, \beta_p$ are defined in order to compensate for the default of the initial tubes~$\tube_1^\circ, \dots, \tube_p^\circ$ in Equation~\eqref{eq:designTrivialDependence}. Finally, we proceed by induction for the remaining initial tubes of~$\tubing^\circ$, that is the initial round tubes not contained in~$\tube$ and whose root is not one of the~$w_i$'s. Namely, since~$q > 0$ and thanks to our special ordering on the $w_i$'s, the coordinate corresponding to such a tube in all terms of Equality~\eqref{eq:designFullDependence} is a linear combination of the coordinates corresponding to its predecessors in the spine of~$\tubing^\circ$ (\ie the inclusion poset on round tubes of~$\tubing^\circ$).

Finally, we still have to check the case where~$q = 0$, meaning where~$\tube$ is contained in an initial round tube of~$\tubing^\circ$. We denote by~$\tube^\circ$ the inclusion minimal initial round tube of~$\tubing^\circ$ containing~$\tube$ and by~$z^\circ$ its root in~$\tubing^\circ$. We then have to distinguish whether or not~$z^\circ=v$:
\begin{itemize}
\item If~$z^\circ = v$, call~$w_1,\dots,w_r$ ($r \ge 1$) the neighbors of~$\tube$ that also belong to~$\tube^\circ$. We then claim that the following linear dependence holds:
\[
\dualCompatibilityVector{\tube}{\tubing^\circ} + r \, \dualCompatibilityVector{\squareTube{v}}{\tubing^\circ}= \sum_{i \in [r]} \dualCompatibilityVector{\squareTube{w_i}}{\tubing^\circ}.
\]
Indeed an initial tube~$\tube[s]^\circ$ of~$\tubing^\circ$ either contains~$\tube$ and thus all the vertices~$v, w_1, \dots, w_r$, or does not contain~$v$ so that the equality holds by counting the vertices~$w_1, \dots, w_r$ in~$\tube[s]^\circ$.
\item If~$z^\circ \ne v$, then~$z^\circ$ belongs to one of the tubes~$\tube[a]_1, \dots, \tube[a]_\ell$, say~$\tube[a]_1$. Let~$w_1,\dots,w_r$ be the neighbors of~$\tube$ contained in~$\tube^\circ$ that are also neighbors of~$\tube[a]_1$, and~$w_{r+1},\dots,w_s$ be the other neighbors of~$\tube$ contained in~$\tube^\circ$. Observe that~$s \ge 1$ since~$\tube \subset \tube^\circ$. One can then check that the linear dependence we look for is
\[
(s-r+1) \, \dualCompatibilityVector{\tube}{\tubing^\circ} = s \, \bigg( \dualCompatibilityVector{\tube[a]_1}{\tubing^\circ}-\dualCompatibilityVector{\squareTube{v}}{\tubing^\circ}-\sum_{i = r +1}^s \dualCompatibilityVector{\squareTube{w_i}}{\tubing^\circ} \bigg) + (s-r+1)\sum_{i = 1}^s \dualCompatibilityVector{\squareTube{w_i}}{\tubing^\circ}.
\]
It clearly holds for tubes not contained in~$\tube^\circ$ and for~$\tube^\circ$ itself. Consider thus an initial round tube~$\tube[s]^\circ$ of~$\tubing^\circ$ contained in~$\tube^\circ$. Since~$z^\circ \in \tube[a]_1$, $\tube[s]^\circ$ cannot contain~$\tube[a]_1$. Thus the first term of the right hand side vanishes for~$\tube[s]^\circ$, while~$\compatibilityDegree{\tube}{\tube[s]^\circ} = \sum_{i \in [s]} \compatibilityDegree{\squareTube{w_i}}{\tube[s]^\circ}$, concluding the proof. \qed
\end{itemize}

\subsection{Design nested complex isomorphisms (Proposition~\ref{prop:isomorphismDesignNormal} and Theorem~\ref{theo:designNestedComplexIsomorphisms})}
\label{subsec:proofDesignIsomorphisms}

We now prove our characterization of the design nested complex isomorphisms announced in Section~\ref{subsec:designNestedComplex}. In both proofs, we will use the following description of links in design nested complexes, similar to that of~\cite{CarrDevadoss} for links in nested complexes. We leave this proof to the reader.

\begin{lemma}
\label{lem:linksDesign}
The link of a tube~$\tube$ in the design nested complex~$\designNestedComplex(\graphG)$ is isomorphic,
\begin{itemize}
\item for a square tube~$\tube = \squareTube{v}$, to the join~$\designNestedComplex(\graphG_1) \ast \dots \ast \, \designNestedComplex(\graphG_\ell)$ of the design nested complexes of the connected components~$\graphG_1, \dots, \graphG_\ell$ of~$\graphG{}[\ground \ssm \{v\}]$.
\item for a round tube~$\tube$, to the join of the nested complex~$\nestedComplex(\graphG{}[\tube])$ of the restriction of~$\graphG$ to~$\tube$ with the design nested complex~$\designNestedComplex(\graphG^\star\tube)$ of the reconnected complement of~$\tube$ in~$\graphG$.
\end{itemize}
\end{lemma}

Before proving Proposition~\ref{prop:isomorphismDesignNormal} and Theorem~\ref{theo:designNestedComplexIsomorphisms}, we need a technical result.

\begin{lemma}
\label{lem:groundImage}
Let~$\bar\graphG$ and~$\graphG$ be connected graphs with vertex sets~$\bar\ground$ and~$\ground$ respectively. Let~$\Phi$ be an isomorphism from the design nested complex~$\designNestedComplex(\bar\graphG)$ of~$\bar\graphG$ to the design (resp.~standard) nested complex~$\designNestedComplex(\graphG)$ (resp.~$\nestedComplex(\graphG)$) of~$\graphG$. If there exists a vertex~$v \in \ground$ such that all tubes of~$\graphG$ incompatible with~$\Phi(\bar\ground)$ are round tubes containing~$v$, then~$\bar\graphG$ and~$\graphG$ are paths.
\end{lemma}

\begin{proof}
Since~$\Phi$ is an isomorphism, it induces a bijection between the design tubes of~$\bar\graphG$ incompatible with~$\bar\ground$ and the tubes of~$\graphG$ incompatible with~$\Phi(\bar\ground)$. The former are precisely the square tubes of~$\bar\graphG$ (by definition of the compatibility of design tubes) while the later are some round tubes of~$\graphG$ containing~$v$ (by assumption). Since all the square tubes of~$\bar\graphG$ are compatible, it follows that their images by~$\Phi$ are nested in~$\graphG$. Let~$\bar{u}, \bar{v} \in \bar\ground$ be the only vertices such that~$|\Phi(\squareTube{\bar{u}})| = 1$ and~$|\Phi(\squareTube{\bar{v}})| = |\ground|-1$. For any other vertex~$\bar{w} \in \bar\ground$, the link of the tube~$\Phi(\squareTube{\bar{w}})$ is the join of two non-trivial design or standard nested complexes by Lemma~\ref{lem:linksDesign}. Since~$\Phi$ is an isomorphism, so is the link of the square tube~$\squareTube{\bar{w}}$, so that~$\bar{w}$ disconnects~$\bar\graphG$ again by Lemma~\ref{lem:linksDesign}. We conclude that all but two vertices of~$\bar\graphG$ disconnect~$\bar\graphG$, which implies that~$\bar\graphG$ is a path. In order to show that~$\graphG$ is also a path, we now distinguish two situations:
\begin{enumerate}[(i)]
\item Suppose first that~$\Phi$ is an isomorphism from the design nested complex~$\designNestedComplex(\bar\graphG)$ of~$\bar\graphG$ to the standard nested complex~$\nestedComplex(\graphG)$ of~$\graphG$. Then the composition of~$\Phi^{-1}$ with the isomorphism~$\Pi$ of Example~\ref{exm:isomorphismDesignNormal}\,(ii) is an isomorphism between the standard nested complex~$\nestedComplex(\graphG)$ of~$\graphG$ and the standard nested complex of a path, which implies by Theorem~\ref{theo:nestedComplexIsomorphisms} that~$\graphG$ itself is a path.
\item Suppose now that~$\Phi$ is an isomorphism from the design nested complex~$\designNestedComplex(\bar\graphG)$ of~$\bar\graphG$ to the design nested complex~$\designNestedComplex(\graphG)$ of~$\graphG$. Using Lemma~\ref{lem:linksDesign} and a similar argument as in the proof of Lemma~\ref{lem:isomorphismsPreserveConnectedComponents}, the image~$\Phi(\bar\ground)$ of~$\bar\ground$ is either a square tube~$\squareTube{v}$ of~$\graphG$, or a singleton round tube~$\{v\}$ of~$\graphG$, or the round tube~$\ground$ of~$\graphG$. The last two cases are discarded by our assumption since the square tube~$\squareTube{v}$ is incompatible with~$\{v\}$ and with~$\ground$. Therefore, we obtain that~$\Phi(\bar\ground) = \squareTube{v}$ and the tubes of~$\graphG$ incompatible with~$\Phi(\bar\ground) = \squareTube{v}$ are exactly the round tubes containing~$v$. So the square tubes of~$\bar\graphG$ are in bijection with the tubes of~$\graphG$ containing~$v$. Since~$|\bar\ground|=|\ground|$ (dimensions of isomorphic simplicial complexes), an immediate induction shows that~$\graphG$ is a path. \qedhere
\end{enumerate}
\end{proof}

We are now ready to prove our classification of design nested complex isomorphisms.

\begin{proof}[Proof of Proposition~\ref{prop:isomorphismDesignNormal}]

We show the result by induction on~$|\bar\ground|$, the cases~$|\bar\ground| \le 2$ being trivial. Assume that~$|\bar\ground| \ge 3$ and consider an isomorphism~$\Phi:\designNestedComplex(\bar\graphG)\to\nestedComplex(\graphG)$. 

We first consider the image~$\Phi(\bar\ground)$ of the round tube~$\bar\ground$ of~$\bar\graphG$. It is either a singleton, or the complement of a singleton. Otherwise, its link would be the join of two non-trivial complexes by Lemma~\ref{lem:linksDesign}, which yields a contradiction using a similar argument as in the proof of Lemma~\ref{lem:isomorphismsPreserveConnectedComponents}. We claim that we can assume that~$\Phi(\bar\ground)$ is a singleton adjacent to at least two vertices in~$\graphG$. Indeed, 
\begin{itemize}
\item if~$\Phi(\bar\ground) = \ground \ssm \{v\}$, then all tubes incompatible with~$\Phi(\bar\ground) = \ground \ssm \{v\}$ contain~$v$,
\item if~$\Phi(\bar\ground) = \{v\}$ where~$v$ has a unique neighbor~$w$ in~$\graphG$, then all tubes incompatible with~$\Phi(\bar\ground) = \{v\}$ contain~$w$,
\end{itemize}
In both cases, Lemma~\ref{lem:groundImage} ensures that~$\bar\graphG$ and~$\graphG$ are paths, and we can compose~$\Phi$ with a rotation~$\rot^p$ to ensure that~$\Phi(\bar\ground)$ is a singleton adjacent to at least two vertices in~$\graphG$.

We now consider a vertex~$\bar{w} \in \bar\ground$ which does not disconnect~$\bar\graphG$ (such a vertex always exists). By Lemma~\ref{lem:linksDesign}, the link of the square tube~$\squareTube{\bar{w}}$ is not the join of two non-trivial nested complexes. Since~$\Phi$ is an isomorphism, so is the link of~$\Phi(\squareTube{\bar{w}})$, so that~$\Phi(\squareTube{\bar{w}})$ is either a singleton or the complement of a singleton again by Lemma~\ref{lem:linksDesign}. Moreover, since~$\squareTube{\bar{w}}$ is incompatible with~$\bar\ground$ and~$\Phi$ is an isomorphism, $\Phi(\squareTube{\bar{w}})$ is incompatible with~$\Phi(\bar\ground) = \{v\}$. Therefore, $\Phi(\squareTube{\bar{w}})$ is either a singleton~$\{w\}$ adjacent to~$\{v\}$, or the complement~$\ground \ssm \{v\}$ of the singleton~$\{v\}$. Now, since~$\bar\ground$ contains at least two vertices~$\bar{w}, \bar{w}'$ which do not disconnect~$\graphG$ and since~$\Phi(\squareTube{\bar{w}}) \ne \Phi(\squareTube{\bar{w}'})$, we can assume that~$\Phi(\squareTube{\bar{w}})$ is a singleton~$\{w\}$ adjacent to~$\{v\}$.

It implies that~$\Phi$ induces an isomorphism from the design nested complex~$\designNestedComplex(\bar\graphG{}[\bar\ground\ssm\{\bar{w}\}])$ to the nested complex~$\nestedComplex(\graphG^\star\{w\})$ of the reconnected complement of the tube~$\{w\}$ in~$\graphG$. The induction hypothesis implies that
\begin{itemize}
\item the graph~$\bar\graphG{}[\bar\ground\ssm\{\bar{w}\}]$ is isomorphic to the spider~$\spiderG_{\ninf}$ and the graph~$\graphG^\star\{w\}$ is isomorphic to the octopus~$\octopusG_{\ninf}$ (with head denoted~$v$), for a certain~$\ninf = \{n_1, \dots, n_{\ell}\} \in \N^\ell$.
\item the image~$\Phi(\bar\ground\ssm\{\bar{w}\})$ of~$\bar\ground\ssm\{\bar{w}\}$ is the pair~$\{v,w\}$ containing the central vertex of~$\graphG^\star\{w\}$,
\item the description of the images of the tubes of~$\bar\graphG$ not containing~$\bar{w}$ is given by~$\bar\Omega$, in particular~$\Phi(\squareTube{\bar{w}'})$ is a singleton adjacent to~$v$ in~$\graphG^\star\{w\}$. Since~$\Phi(\squareTube{\bar{w}'})$ cannot be an edge in~$\graphG$, it then has to be also a singleton~$\{w'\}$ in~$\graphG$, non-adjacent to~$w$ for compatibility.
\end{itemize}

Since~$\graphG^\star\{w\}$ is an octopus with head~$v$, it follows by definition of the reconnected complement that~$\graphG$ is either an octopus with head~$v$ or an octopus with head~$v$ with an additional edge of the form~$\{v,v_i^1\}$ for a certain~$i\in[\ell]$. We can now apply the same reasonning to~$\bar{w}'$ and since~$\Phi(\squareTube{\bar{w}})$ and~$\Phi(\squareTube{\bar{w}'})$ are non-adjacent singletons in~$\graphG$, we conclude that~$\graphG$ is an octopus. Since~$\Phi(\bar\ground)=\{v\}$, the graph~$\bar\graphG$ is a spider and the reader can check that the restriction of~$\Phi$ to the link of~$\bar\ground$ in the design nested complex~$\designNestedComplex(\bar\graphG)$ is the non-trivial isomorphism~$\Omega:\nestedComplex(\bar\graphG)\to\nestedComplex(\graphG^\star\{v\})$ defined in Section~\ref{subsec:many}. It follows that~$\Phi$ coincides with~$\bar\Omega$.
\end{proof}

\begin{proof}[Proof of Theorem~\ref{theo:designNestedComplexIsomorphisms}]
If~$\graphG$ is the path~$\pathG_n$, then its design nested complex~$\designNestedComplex(\graphG)$ is isomorphic to the nested complex~$\nestedComplex(\pathG_{n+1})$ by Example~\ref{exm:isomorphismDesignNormal}\,(ii). Therefore, the design nested complex~$\designNestedComplex(\graphG')$ is also isomorphic to the nested complex~$\nestedComplex(\pathG_{n+1})$, which implies that~$G'$ is the path~$\pathG_n$ by Proposition~\ref{prop:isomorphismDesignNormal}. We can therefore assume that~$\graphG$ and~$\graphG'$ are not paths.

Let~$\ground$ and~$\ground'$ denote the vertex sets of~$\graphG$ and~$\graphG'$. By Lemma~\ref{lem:linksDesign}, the link of~$\ground$ in~$\designNestedComplex(\graphG)$ is the nested complex~$\nestedComplex(\graphG)$. With a similar argument as in the proof of Lemma~\ref{lem:isomorphismsPreserveConnectedComponents}, it follows that its image~$\Phi(\ground)$ is either a square tube of~$\graphG'$, or a singleton round tube of~$\graphG'$, or the round tube~$\ground'$ of~$\graphG'$. We treat these situations separately.
\begin{description}
\item[if~$\Phi(\ground) = \squareTube{v'}$] 
All tubes of~$\graphG'$ incompatible with~$\Phi(\ground) = \squareTube{v'}$ are round tubes containing~$v'$. Thus by Lemma~\ref{lem:groundImage}, the graphs~$\graphG$ and~$\graphG'$ are paths, which we already excluded.

\item[if~$\Phi(\ground)=\ground'$]
By Lemma~\ref{lem:linksDesign}, $\Phi$ induces a non-trivial nested complex isomorphism~$\Psi$ from~$\nestedComplex(\graphG)$ to~$\nestedComplex(\graphG')$. It follows from Theorem~\ref{theo:nestedComplexIsomorphisms} that~$\graphG$ and~$\graphG'$ are isomorphic spiders and~$\Psi$ coincides with the non-trivial nested complex isomorphism~$\Omega$ described in Section~\ref{subsec:many}. It thus suffices to show that the nested complex isomorphism~$\Omega$ cannot be extended to design nested complexes. For this observe first that such an extension would send square tubes to square tubes. Now consider a singleton~$\{v\}$ of~$\graphG$ swapped by~$\Omega$. This singleton~$\{v\}$ is incompatible with the square tube~$\squareTube{v}$ and compatible with all other square tubes of~$\graphG$. Yet its image is incompatible with more than one square tube, a contradiction.

\item[if~$\Phi(\ground)=\{v'\}$]
The link of~$\{v'\}$ in the design nested complex~$\designNestedComplex(\graphG')$ is a design nested complex  isomorphic to the nested complex~$\nestedComplex(\graphG)$. By Proposition~\ref{prop:isomorphismDesignNormal}, we obtain that~$\graphG$ is an octopus and~$\Phi^{-1}(\ground')$ is the singleton containing its central vertex. Since~$\Phi^{-1}(\ground')\neq\ground$, the same argument applies to~$\Phi^{-1}$ and shows that~$\graphG'$ is also an octopus with the same legs as~$\graphG$. Moreover Proposition~\ref{prop:isomorphismDesignNormal} describes the images of round tubes of~$\graphG$ by~$\Phi$ and of round tubes of~$\graphG'$ by~$\Phi^{-1}$, which together forces~$\Phi$ to coincide with the map~$\Omega\design$. \qedhere
\end{description}
\end{proof}


\section*{Acknoledgements}

\enlargethispage{.4cm}
We thank Francisco Santos for comments and suggestions on a former version of this paper. We also thank an FPSAC referee for pointing out a misleading over-simplification in a previous version of the proof of Theorem~\ref{theo:dualCompatibilityFan}.


\bibliographystyle{alpha}
\bibliography{compatibilityFans}

\begin{thebibliography}{MHPS12}

\bibitem[BFS90]{BilleraFillimanSturmfels}
Louis~J. Billera, Paul Filliman, and Bernd Sturmfels.
\newblock Constructions and complexity of secondary polytopes.
\newblock {\em Adv.~Math.}, 83(2):155--179, 1990.

\bibitem[BHLT09]{BergeronHohlwegLangeThomas}
Nantel Bergeron, Christophe Hohlweg, Carsten Lange, and Hugh Thomas.
\newblock Isometry classes of generalized associahedra.
\newblock {\em S\'em. Lothar. Combin.}, 61A:Art. B61Aa, 13, 2009.

\bibitem[CD06]{CarrDevadoss}
Michael~P. Carr and Satyan~L. Devadoss.
\newblock Coxeter complexes and graph-associahedra.
\newblock {\em Topology Appl.}, 153(12):2155--2168, 2006.

\bibitem[CFZ02]{ChapotonFominZelevinsky}
Fr{\'e}d{\'e}ric Chapoton, Sergey Fomin, and Andrei Zelevinsky.
\newblock Polytopal realizations of generalized associahedra.
\newblock {\em Canad. Math. Bull.}, 45(4):537--566, 2002.

\bibitem[Cha00]{Chapoton}
Fr{\'e}d{\'e}ric Chapoton.
\newblock Alg\`ebres de {H}opf des permutah\`edres, associah\`edres et
  hypercubes.
\newblock {\em Adv. Math.}, 150(2):264--275, 2000.

\bibitem[CP14]{ChatelPilaud}
Gr\'egory Chatel and Vincent Pilaud.
\newblock {C}ambrian {H}opf {A}lgebras.
\newblock Preprint,
  \href{http://arxiv.org/abs/1411.3704}{\texttt{arXiv:1411.3704}}, 2014.

\bibitem[CP15]{CeballosPilaud}
Cesar Ceballos and Vincent Pilaud.
\newblock Denominator vectors and compatibility degrees in cluster algebras of
  finite types.
\newblock {\em Trans. Amer. Math. Soc.}, 367:1421--1439, 2015.

\bibitem[CSZ15]{CeballosSantosZiegler}
Cesar Ceballos, Francisco Santos, and G\"unter~M. Ziegler.
\newblock Many non-equivalent realizations of the associahedron.
\newblock {\em Combinatorica}, 35(5):513--551, 2015.

\bibitem[DCP95]{DeConciniProcesi}
Conrado De~Concini and Claudio Procesi.
\newblock Wonderful models of subspace arrangements.
\newblock {\em Selecta Math. (N.S.)}, 1(3):459--494, 1995.

\bibitem[Deh10]{Dehornoy}
Patrick Dehornoy.
\newblock On the rotation distance between binary trees.
\newblock {\em Adv. Math.}, 223(4):1316--1355, 2010.

\bibitem[Dev09]{Devadoss}
Satyan~L. Devadoss.
\newblock A realization of graph associahedra.
\newblock {\em Discrete Math.}, 309(1):271--276, 2009.

\bibitem[DFRS15]{DevadossForceyReisdorfShowers}
Satyan~L. Devadoss, Stefan Forcey, Stephen Reisdorf, and Patrick Showers.
\newblock Convex polytopes from nested posets.
\newblock {\em European J. Combin.}, 43:229--248, 2015.

\bibitem[DHV11]{DevadossHeathVipismakul}
Satyan~L. Devadoss, Timothy Heath, and Wasin Vipismakul.
\newblock Deformations of bordered surfaces and convex polytopes.
\newblock {\em Notices Amer. Math. Soc.}, 58(4):530--541, 2011.

\bibitem[DRS10]{DeLoeraRambauSantos}
Jesus~A. {De Loera}, J\"org Rambau, and Francisco Santos.
\newblock {\em Triangulations: Structures for Algorithms and Applications},
  volume~25 of {\em Algorithms and {C}omputation in Mathematics}.
\newblock Springer Verlag, 2010.

\bibitem[FS05]{FeichtnerSturmfels}
Eva~Maria Feichtner and Bernd Sturmfels.
\newblock Matroid polytopes, nested sets and {B}ergman fans.
\newblock {\em Port. Math. (N.S.)}, 62(4):437--468, 2005.

\bibitem[FZ02]{FominZelevinsky-ClusterAlgebrasI}
Sergey Fomin and Andrei Zelevinsky.
\newblock Cluster algebras. {I}. {F}oundations.
\newblock {\em J. Amer. Math. Soc.}, 15(2):497--529 (electronic), 2002.

\bibitem[FZ03a]{FominZelevinsky-ClusterAlgebrasII}
Sergey Fomin and Andrei Zelevinsky.
\newblock Cluster algebras. {II}. {F}inite type classification.
\newblock {\em Invent. Math.}, 154(1):63--121, 2003.

\bibitem[FZ03b]{FominZelevinsky-YSystems}
Sergey Fomin and Andrei Zelevinsky.
\newblock {$Y$}-systems and generalized associahedra.
\newblock {\em Ann. of Math. (2)}, 158(3):977--1018, 2003.

\bibitem[FZ07]{FominZelevinsky-ClusterAlgebrasIV}
Sergey Fomin and Andrei Zelevinsky.
\newblock Cluster algebras. {IV}. {C}oefficients.
\newblock {\em Compos. Math.}, 143(1):112--164, 2007.

\bibitem[GKZ08]{GelfandKapranovZelevinsky}
Israel Gelfand, Mikhail M.~M. Kapranov, and Andrei Zelevinsky.
\newblock {\em Discriminants, resultants and multidimensional determinants}.
\newblock Modern Birkh\"auser Classics. Birkh\"auser Boston Inc., Boston, MA,
  2008.
\newblock Reprint of the 1994 edition.

\bibitem[Hai84]{Haiman}
Mark Haiman.
\newblock Constructing the associahedron.
\newblock Unpublished manuscript, 11 pages, available at
  \url{http://www.math.berkeley.edu/~mhaiman/ftp/assoc/manuscript.pdf}, 1984.

\bibitem[HL07]{HohlwegLange}
Christophe Hohlweg and Carsten Lange.
\newblock Realizations of the associahedron and cyclohedron.
\newblock {\em Discrete Comput.~Geom.}, 37(4):517--543, 2007.

\bibitem[HLR10]{HohlwegLortieRaymond}
Christophe Hohlweg, Jonathan Lortie, and Annie Raymond.
\newblock The centers of gravity of the associahedron and of the permutahedron
  are the same.
\newblock {\em Electron. J. Combin.}, 17(1):Research Paper 72, 14, 2010.

\bibitem[HLT11]{HohlwegLangeThomas}
Christophe Hohlweg, Carsten Lange, and Hugh Thomas.
\newblock Permutahedra and generalized associahedra.
\newblock {\em Adv. Math.}, 226(1):608--640, 2011.

\bibitem[HN99]{HurtadoNoy}
Ferran Hurtado and Marc Noy.
\newblock Graph of triangulations of a convex polygon and tree of
  triangulations.
\newblock {\em Comput. Geom.}, 13(3):179--188, 1999.

\bibitem[HNT05]{HivertNovelliThibon-algebraBinarySearchTrees}
Florent Hivert, Jean-Christophe Novelli, and Jean-Yves Thibon.
\newblock The algebra of binary search trees.
\newblock {\em Theoret. Comput. Sci.}, 339(1):129--165, 2005.

\bibitem[Hoh12]{Hohlweg}
Christophe Hohlweg.
\newblock Permutahedra and associahedra.
\newblock In Folkert M{\"u}ller-Hoissen, Jean Pallo, and Jim Stasheff, editors,
  {\em Associahedra, Tamari Lattices and Related Structures~--~Tamari Memorial
  Festschrift}, volume 299 of {\em Progress in Mathematics}, pages 129--159.
  Birkh{\"a}user, 2012.

\bibitem[Lee89]{Lee}
Carl~W. Lee.
\newblock The associahedron and triangulations of the {$n$}-gon.
\newblock {\em European J.~Combin.}, 10(6):551--560, 1989.

\bibitem[Lod04]{Loday}
Jean-Louis Loday.
\newblock Realization of the {S}tasheff polytope.
\newblock {\em Arch.~Math.~(Basel)}, 83(3):267--278, 2004.

\bibitem[LP12a]{LamPylyavskyy-LaurentPhenomenonAlgebras}
Thomas Lam and Pavlo Pylyavskyy.
\newblock {L}aurent phenomenon algebras.
\newblock Preprint,
  \href{http://arxiv.org/abs/1206.2611}{\texttt{arXiv:1206.2611}}, 2012.

\bibitem[LP12b]{LamPylyavskyy-LinearLaurentPhenomenonAlgebras}
Thomas Lam and Pavlo Pylyavskyy.
\newblock Linear {L}aurent phenomenon algebras.
\newblock Preprint,
  \href{http://arxiv.org/abs/1206.2612}{\texttt{arXiv:1206.2612}}, 2012.

\bibitem[LP13]{LangePilaud}
Carsten Lange and Vincent Pilaud.
\newblock Associahedra via spines.
\newblock To appear in \emph{Combinatorica} (preprint available at
  \href{http://arxiv.org/abs/1307.4391}{\texttt{arXiv:1307.4391}}), 2013.

\bibitem[LR98]{LodayRonco}
Jean-Louis Loday and Mar{\'{\i}}a~O. Ronco.
\newblock Hopf algebra of the planar binary trees.
\newblock {\em Adv. Math.}, 139(2):293--309, 1998.

\bibitem[MHPS12]{TamariFestschrift}
Folkert M{\"u}ller-Hoissen, Jean~Marcel Pallo, and Jim Stasheff, editors.
\newblock {\em Associahedra, {T}amari Lattices and Related Structures. Tamari
  Memorial Festschrift}, volume 299 of {\em Progress in Mathematics}.
\newblock Birkh{\"a}user, Basel, 2012.

\bibitem[{OEIS}]{OEIS}
The {O}n-{L}ine {E}ncyclopedia of {I}nteger {S}equences.
\newblock Published electronically at \url{http://oeis.org}, 2010.

\bibitem[Pil13]{Pilaud}
Vincent Pilaud.
\newblock Signed tree associahedra.
\newblock Preprint,
  \href{http://arxiv.org/abs/1309.5222}{\texttt{arXiv:1309.5222}}, 2013.

\bibitem[Pos09]{Postnikov}
Alexander Postnikov.
\newblock Permutohedra, associahedra, and beyond.
\newblock {\em Int. Math. Res. Not. IMRN}, (6):1026--1106, 2009.

\bibitem[Pou14]{Pournin}
Lionel Pournin.
\newblock The diameter of associahedra.
\newblock {\em Adv. Math.}, 259:13--42, 2014.

\bibitem[PS12]{PilaudSantos-brickPolytope}
Vincent Pilaud and Francisco Santos.
\newblock The brick polytope of a sorting network.
\newblock {\em European~J.~Combin.}, 33(4):632--662, 2012.

\bibitem[PS15a]{PilaudStump-brickPolytope}
Vincent Pilaud and Christian Stump.
\newblock Brick polytopes of spherical subword complexes and generalized
  associahedra.
\newblock {\em Adv.~Math.}, 276:1--61, 2015.

\bibitem[PS15b]{PilaudStump-barycenter}
Vincent Pilaud and Christian Stump.
\newblock Vertex barycenter of generalized associahedra.
\newblock {\em Proc. Amer. Math. Soc.}, 143(6):2623--2636, 2015.

\bibitem[Rea04]{Reading-latticeCongruences}
Nathan Reading.
\newblock Lattice congruences of the weak order.
\newblock {\em Order}, 21(4):315--344 (2005), 2004.

\bibitem[Rea06]{Reading-CambrianLattices}
Nathan Reading.
\newblock Cambrian lattices.
\newblock {\em Adv.~Math.}, 205(2):313--353, 2006.

\bibitem[Rea14]{Reading-UniversalClusterAlgebra}
Nathan Reading.
\newblock Universal geometric cluster algebras.
\newblock {\em Math. Z.}, 277(1-2):499--547, 2014.

\bibitem[RS09]{ReadingSpeyer}
Nathan Reading and David~E. Speyer.
\newblock Cambrian fans.
\newblock {\em J.~Eur.~Math.~Soc.~(JEMS)}, 11(2):407--447, 2009.

\bibitem[RSS03]{RoteSantosStreinu-polytope}
G{\"u}nter Rote, Francisco Santos, and Ileana Streinu.
\newblock Expansive motions and the polytope of pointed pseudo-triangulations.
\newblock In {\em Discrete and computational geometry}, volume~25 of {\em
  Algorithms Combin.}, pages 699--736. Springer, Berlin, 2003.

\bibitem[Sta63]{Stasheff}
Jim Stasheff.
\newblock Homotopy associativity of {H}-spaces {I}, {II}.
\newblock {\em Trans. Amer. Math. Soc.}, 108(2):293--312, 1963.

\bibitem[Ste13]{Stella}
Salvatore Stella.
\newblock Polyhedral models for generalized associahedra via {C}oxeter
  elements.
\newblock {\em J. Algebraic Combin.}, 38(1):121--158, 2013.

\bibitem[STT88]{SleatorTarjanThurston}
Daniel~D. Sleator, Robert~E. Tarjan, and William~P. Thurston.
\newblock Rotation distance, triangulations, and hyperbolic geometry.
\newblock {\em J. Amer. Math. Soc.}, 1(3):647--681, 1988.

\bibitem[Tam51]{Tamari}
Dov Tamari.
\newblock {\em Mono\"ides pr\'eordonn\'es et cha\^ines de Malcev}.
\newblock PhD thesis, Universit\'e Paris Sorbonne, 1951.

\bibitem[Vol10]{Volodin}
Vadim Volodin.
\newblock Cubical realizations of flag nestohedra and a proof of {G}al's
  conjecture for them.
\newblock {\em Uspekhi Mat. Nauk}, 65(1(391)):183--184, 2010.

\bibitem[Zel06]{Zelevinsky}
Andrei Zelevinsky.
\newblock Nested complexes and their polyhedral realizations.
\newblock {\em Pure Appl. Math. Q.}, 2(3):655--671, 2006.

\bibitem[Zie95]{Ziegler}
G{\"u}nter~M. Ziegler.
\newblock {\em Lectures on polytopes}, volume 152 of {\em Graduate Texts in
  Mathematics}.
\newblock Springer-Verlag, New York, 1995.

\end{thebibliography}
\label{sec:biblio}

\end{document}